\documentclass[11pt,a4paper]{amsart}
\usepackage[T1]{fontenc}
\usepackage{lmodern}
\usepackage{microtype}
\usepackage{mathtools,amssymb,mathrsfs}
\usepackage[margin=27mm]{geometry}
\usepackage{xurl}
\usepackage[colorlinks=true,linkcolor=blue,citecolor=blue,urlcolor=blue,bookmarksnumbered]{hyperref}
\usepackage{bookmark}
\usepackage{enumitem}
\usepackage{tikz}
\definecolor{figOld}{HTML}{567794}
\definecolor{figNew}{HTML}{557E77}
\definecolor{figOverlap}{HTML}{929B9E}
\definecolor{figInk}{HTML}{273846}
\numberwithin{equation}{section}
\newtheorem{theorem}{Theorem}[section]
\newtheorem{proposition}[theorem]{Proposition}
\newtheorem{lemma}[theorem]{Lemma}
\newtheorem{corollary}[theorem]{Corollary}
\newtheorem{conjecture}[theorem]{Conjecture}
\theoremstyle{definition}
\newtheorem{definition}[theorem]{Definition}

\newcommand{\C}{\mathbb C}
\newcommand{\PP}{\mathbb P}
\newcommand{\R}{\mathbb R}
\newcommand{\OO}{\mathcal O}
\newcommand{\LL}{\mathcal L}
\newcommand{\CAP}{\mathrm{CAP}}

\newcommand{\Op}{\operatorname{Op}}
\newcommand{\Int}{\operatorname{int}}

\newcommand{\norm}[1]{\left\lVert #1\right\rVert}
\newcommand{\PD}[2]{\overline\Delta_{#1}^{\,#2}}
\allowdisplaybreaks[2]
\theoremstyle{remark}
\newtheorem{remark}[theorem]{Remark}
\newtheorem{example}[theorem]{Example}

\newcommand{\LAP}{\textup{LAP}}
\newcommand{\LAPName}{local approximation property}
\hypersetup{pdftitle={A Local Approximation Characterization of Oka Manifolds},pdfauthor={Sicheng An, Yun-Heng Du, Bin Guo, Peng-Chao Wang, Song-Yan Xie},pdfsubject={Centered disc families, convex approximation, rational curves and logarithmic complements}}
\title[Local approximation and Oka manifolds]{A Local Approximation Characterization\\ of Oka Manifolds}
\author{Sicheng An}
\address[Sicheng An]{School of Mathematical Sciences, University of Chinese Academy of Sciences, Beijing 100049, China}
\email{ansicheng23@mails.ucas.ac.cn}

\author{Yun-Heng Du}
\address[Yun-Heng Du]{Academy of Mathematics and Systems Science, Chinese Academy of Sciences, Beijing 100190, China}
\email{duyunheng22@mails.ucas.ac.cn}
\author{Bin Guo}
\address[Bin Guo]{Academy of Mathematics and Systems Science, Chinese Academy of Sciences, Beijing 100190, China}
\email{guobin181@mails.ucas.ac.cn}
\author{Peng-Chao Wang}
\address[Peng-Chao Wang]{Academy of Mathematics and Systems Science, Chinese Academy of Sciences, Beijing 100190, China}
\email{wangpengchao26@mails.ucas.ac.cn}
\author{Song-Yan Xie}
\address[Song-Yan Xie]{State Key Laboratory of Mathematical Sciences, Academy of Mathematics and Systems Science, Chinese Academy of Sciences, Beijing 100190, China; School of Mathematical Sciences, University of Chinese Academy of Sciences, Beijing 100049, China}
\email{xiesongyan@amss.ac.cn}
\date{October 14, 2026}

\begin{document}
\begin{abstract}
We introduce the \LAPName{} (\LAP) and prove that it characterizes Oka manifolds. We prove that \LAP{} yields convex approximation and, in an $n$-dimensional target, propagates across sets of zero $(2n-1)$-dimensional Hausdorff measure. Consequently, strong dominability characterizes Oka manifolds, convex approximation in the target dimension suffices, the Oka property is local in the Euclidean topology, and it descends under surjective holomorphic maps. Local families of entire curves give geometric applications: smooth projective rationally connected manifolds are Oka, and deleting an algebraic subset of codimension at least two from any such manifold preserves the Oka property. Structure theorems further show that compact K\"ahler manifolds with nonnegative holomorphic sectional curvature or nef tangent bundle are Oka. Finally, the complement of a smooth degree-$d$ hypersurface in $\PP^n$ is Oka exactly when $d\leq n+1$.
\end{abstract}
\maketitle
\setcounter{tocdepth}{1}
\tableofcontents
\clearpage

\section{Introduction}\label{sec:introduction}
The Oka principle concerns the existence of holomorphic solutions to problems on Stein spaces when there is no topological obstruction, provided the target is flexible enough. It is not a single theorem but a program: one looks for geometric conditions on a complex manifold $Y$ under which continuous maps from Stein spaces into $Y$ can be deformed to holomorphic maps, with approximation, interpolation, and parametric control. For broad accounts, see the survey of Forstneri\v{c} and L\'arusson~\cite{FL11} and Forstneri\v{c}'s recent exposition~\cite{F26}. The subject began with Oka's solution of the second Cousin problem~\cite{Oka39} and Grauert's theorem for holomorphic fiber bundles with homogeneous fibers over Stein spaces~\cite{Grauert58}. Gromov placed it in the framework of the $h$-principle and introduced ellipticity through dominating holomorphic sprays~\cite{Gro89}, and Forstneri\v{c} introduced subellipticity and proved the Oka principle for subelliptic manifolds~\cite{F02}. Forstneri\v{c} then established approximation, interpolation, and finite-order jet interpolation~\cite{F05}, identified convex approximation with the Oka property with approximation~\cite{F06}, and proved that convex approximation implies the full parametric Oka property~\cite{F09}. This is the notion of an Oka manifold used below; see also Forstneri\v{c}'s monograph~\cite{F17}.

Here we show that this global holomorphic flexibility is detected by a local approximation operation: centered holomorphic families on expanding discs characterize Oka manifolds. The characterization also shows that the Oka property propagates across proper analytic subsets and descends under surjective holomorphic maps from Oka manifolds, leading to applications to projective manifolds and complements.

\begin{definition}\label{def:oka}
Let $Y$ be a complex manifold with a distance $d_Y$ inducing its
topology. The manifold $Y$ is an \emph{Oka manifold} if it has the \emph{convex approximation property} $\CAP$~\cite[Definition~1.2]{F09}: for every $m\geq1$, compact convex set $K\subset\C^m$, holomorphic map $f:U\to Y$ on an open neighborhood $U$ of $K$, and $\epsilon>0$, there is a holomorphic map $F:\C^m\to Y$ such that
\[
\sup_{z\in K}d_Y(F(z),f(z))<\epsilon.
\]
\end{definition}

By Forstneri\v{c}'s theorem~\cite[Theorem~1.1]{F09}, this is equivalent to the parametric Oka property, with approximation and interpolation, for maps from reduced Stein spaces.

For $r>0$, write $\Delta_r=\{t\in\C:|t|<r\}$.

\begin{definition}\label{family:condition}
Let $Y$ be a complex manifold of dimension $n\geq1$. One says that
$Y$ satisfies the \emph{\LAPName} (\LAP) at $x\in Y$ if there are an open
neighborhood $V\ni x$, a number $\tau>0$, radii $R_\nu\to\infty$,
and holomorphic maps
\begin{equation*}
 h_{j,\nu}:V\times\Delta_{R_\nu}\longrightarrow Y,
 \qquad h_{j,\nu}(y,0)=y,\qquad 1\leq j\leq n,
\end{equation*}
converging locally uniformly on $V\times\Delta_\tau$ to maps $h_j$,
such that
\[
 \partial_t h_1(x,0),\ldots,\partial_t h_n(x,0)
\]
form a basis of $T_xY$. One says that $Y$ satisfies \LAP{} if it
satisfies \LAP{} at every point. In dimension zero, \LAP{} holds at every point.
\end{definition}

Write $G_Y\subset Y$ for the points at which $Y$ satisfies \LAP{}.
By continuity of the initial velocities, the locus $G_Y$ is open: after
shrinking $V$, the same families satisfy the definition at every point
of $V$.

All the data in this definition are common to the finitely many
families, after restriction and omission of a finite initial segment
if necessary. In particular, $V$ is independent of $\nu$, and
convergence is imposed on the fixed product $V\times\Delta_\tau$
without requiring an entire limiting map.
A finite spanning collection can always be reduced to a basis.

\begin{theorem}\label{thm:families-oka}
A complex manifold is Oka if and only if it satisfies \LAP.
\end{theorem}

The implication $\LAP\Rightarrow\CAP$ proceeds from local planar extension to polydisc approximation and then to arbitrary compact convex sets by induction across real faces. The local extension retains any prescribed passive parameter polydisc, so that earlier source and spray variables can serve as holomorphic parameters at subsequent stages. Alarc\'on and Forstneri\v{c}~\cite[Definition~7.11]{AF25} introduced the local spray approximation property (LSAP) in their study of Oka-1, the Oka property with approximation and jet interpolation from open Riemann surfaces.

An $n$-dimensional manifold $Y$ is \emph{dominable at $y$} if there is a holomorphic map $H:\C^n\to Y$ with $H(0)=y$ and $dH_0$ invertible, and is \emph{strongly dominable} if this holds at every point. Dominability at $y$ gives \LAP{} at $y$ by a local inverse of $H$ (Proposition~\ref{family:strong-domination}), while every Oka manifold is strongly dominable by Forstneri\v{c}'s result~\cite{F17}. Hence Theorem~\ref{thm:families-oka} settles the strong-dominability question raised by Forstneri\v{c} and L\'arusson~\cite{FL25} and subsequently emphasized by Forstneri\v{c}~\cite{F26}.

\begin{theorem}\label{thm:strong-domination}
A complex manifold is Oka if and only if it is strongly dominable.
\end{theorem}

The next three corollaries concern source dimension, stratifications, and open covers. For $d\geq1$, write $\mathrm{CAP}_d$ for convex approximation on compact convex subsets of $\C^d$.

\begin{corollary}\label{cor:capn}
Let $X$ be an $n$-dimensional complex manifold, where $n\geq1$. Then $X$ is Oka if and only if it satisfies $\mathrm{CAP}_n$.
\end{corollary}

Forstneri\v{c} and L\'arusson introduced a \emph{finite Oka stratification}: a finite decreasing filtration by closed complex subvarieties, ending in the empty set, whose successive differences are manifolds with Oka connected components~\cite[Definition~1]{FL14}. They proved that such a stratification implies strong dominability~\cite[Theorem~2]{FL14}.

\begin{corollary}\label{cor:stratification}
Every complex manifold with a finite Oka stratification is Oka.
\end{corollary}

\begin{corollary}\label{cor:euclidean-localization}
If every point of a complex manifold has an open Oka neighborhood, then the manifold is Oka.
\end{corollary}

Corollary~\ref{cor:capn} resolves Forstneri\v{c}'s stabilization question at the proposed threshold $\dim X$~\cite[Problem~6.2]{F06}. Corollary~\ref{cor:stratification} shows that a finite Oka stratification gives an Oka manifold, a question raised by Forstneri\v{c} and L\'arusson~\cite[p.~3716]{FL14}. Kusakabe proved localization for Zariski-open covers and asked whether Euclidean-open covers suffice~\cite{Kus21}; Corollary~\ref{cor:euclidean-localization} settles his question. The proofs are given in Subsection~\ref{sec:direct-consequences}.

Kusakabe's $\mathrm{Ell}_1$ characterization requires every holomorphic map $f:S\to Y$ from a Stein manifold to admit a dominating spray $F:S\times\C^N\to Y$ for some $N$, with $F(\cdot,0)=f$; sprays over maps near compact convex sets also suffice~\cite{Kus21}. Theorem~\ref{thm:strong-domination} obtains the Oka property from the pointwise case, with $S$ a singleton. Kusakabe also characterized Oka manifolds by dense entire curves in spaces of holomorphic maps and by universal maps~\cite{Kus17}.

For $s\geq0$, let $\mathcal H^s$ denote $s$-dimensional Hausdorff
outer measure with respect to a smooth Riemannian metric. Its null
sets are independent of this choice, since smooth metrics are locally
bi-Lipschitz equivalent.

\begin{theorem}
\label{thm:thin-removability}\label{family:hausdorff-exception}
Let $Y$ be a complex manifold of dimension $n\geq1$ and
$E\subset Y$ a set satisfying $\mathcal H^{2n-1}(E)=0$. If $Y$ satisfies
\LAP{} at every point of $Y\setminus E$, then it satisfies
\LAP{} everywhere and is Oka, without assuming that $E$ is closed.
\end{theorem}

Dominability outside such a set therefore implies the full Oka property. This strengthens the Oka-1 conclusion of Alarc\'on and Forstneri\v{c}~\cite{AF25}, which assumes the exceptional set to be closed. Forstneri\v{c} and L\'arusson developed further properties of Oka-1 manifolds and of LSAP~\cite{FL25}. For analytic exceptional sets, Theorem~\ref{thm:thin-removability}
has the following direct consequence:

\begin{samepage}
\begin{theorem}\label{thm:analytic-removability}
Let $Y$ be a complex manifold. Suppose that every point of $Y$ has
a connected open neighborhood $U$ and a proper closed complex analytic
subset $A_U\subset U$ such that $Y$ satisfies \LAP{} at every point
of $U\setminus A_U$. Then $Y$ satisfies \LAP{} everywhere and is Oka.
Moreover, if $X$ is a connected complex manifold and
$A\subset X$ is a proper closed complex analytic subset such that
$X\setminus A$ is Oka, then $X$ is Oka.
\end{theorem}
\end{samepage}

Theorem~\ref{thm:analytic-removability} settles Forstneri\v{c}'s union problem~\cite[Problem~3.11]{Forstneric23} and the first part of Question~D of Forstneri\v{c} and L\'arusson~\cite{FL11}: one nonempty Zariski-open Oka subset suffices.

The preceding criterion for exceptional sets of zero $(2n-1)$-dimensional Hausdorff measure also gives descent along arbitrary surjective holomorphic maps from Oka manifolds (Theorem~\ref{thm:surjective-descent}). Together with Forstneri\v{c}'s surjectivity theorem~\cite{FSurj17}, this characterizes positive-dimensional connected Oka manifolds by surjective entire maps (Corollary~\ref{cor:surjective-entire}) and settles Forstneri\v{c}'s Problem~1.4(a).

\subsection*{Geometric applications}
A smooth complex projective manifold $X$ is \emph{rationally connected} if a general pair of its points lies on the image of a map $\PP^1\to X$, and is \emph{Fano} if $K_X^{-1}$ is ample. Deformations of very free rational curves give:

\begin{theorem}\label{thm:rc}
Every smooth connected complex projective rationally connected manifold is Oka. In particular, every smooth complex Fano manifold is Oka.
\end{theorem}

Theorem~\ref{thm:rc} settles the rationally connected case of Forstneri\v{c}'s Problem~11.3~\cite{Forstneric23}. In particular, every smooth hypersurface of degree $d\leq n$ in $\PP^n$ is Oka, which gives the affirmative answer to the hypersurface part of Problem~E of Forstneri\v{c} and L\'arusson~\cite{FL11} throughout the Fano range.

Campana and Winkelmann constructed dense entire curves in rationally connected projective manifolds whose images contain prescribed countable subsets~\cite{CW}. Alarc\'on and Forstneri\v{c} formulated the Oka-1 conjecture for these manifolds~\cite{AF25}, and Du, Guo, and Xie proved it~\cite{DGX26b}. Related analytic constructions of rational curves were developed by Du, Guo, and Xie~\cite{DGX26a}. Among earlier full Oka results are the smooth quartic hypersurfaces of dimension at least three~\cite{GX26} and all smooth Fano threefolds~\cite{AX26}. Section~\ref{sec:rc} proves the general rationally connected case directly by constructing families of entire curves with independent tangent directions.

Corollaries~\ref{cor:kahler-hsc} and~\ref{cor:kahler-nef-tangent} show that compact K\"ahler manifolds with nonnegative holomorphic sectional curvature or nef tangent bundle are Oka. Forstneri\v{c} and Kusakabe proved that, for $n\geq2$ and $k\geq1$, the unit disc bundle defined by a semipositive Hermitian metric on $\OO_{\PP^n}(k)$ is Oka~\cite{FK25}.

For complements, Kusakabe's theorem covers compact polynomially convex sets in $\C^n$, $n>1$~\cite{KusConvex24}; Forstneri\v{c} and Wold treated broader classes of unbounded convex sets~\cite{FW24}, and further examples come from the periodic-transport constructions of Du, Guo, Wang, and Xie~\cite{DGWX26}. In the projective setting, Theorem~\ref{thm:log} gives a criterion for $X\setminus D$ using a rational curve $f:\PP^1\to X$ with $f^{-1}(D)\subseteq\{\infty\}$ and ample $f^*T_X(-\log D)$, for a reduced simple normal crossings divisor $D$ in a connected projective manifold $X$. Here $T_X(-\log D)$ is the bundle of vector fields tangent to every component of $D$. Chen and Zhu's theory of $\mathbb A^1$-curves on log smooth varieties~\cite{ChenZhu}, combined with Corollary~\ref{cor:a1-connected}, gives applications to ample rationally connected divisors and complete intersections. The same deformation method also shows that deleting an algebraic subset of codimension at least two from a smooth projective rationally connected manifold preserves the Oka property (Corollary~\ref{cor:rc-codim-two}).

A \emph{holomorphic spray} on $Y$ is a holomorphic map $s:E\to Y$ from a holomorphic vector bundle $E\to Y$ such that $s(0_y)=y$. It is \emph{dominating} if $d(s|_{E_y})_{0_y}$ is surjective for every $y$; a manifold admitting such a spray is \emph{holomorphically elliptic}. Gromov's theorem implies that every holomorphically elliptic manifold is Oka~\cite{Gro89}. For smooth hypersurface complements we obtain:

\begin{samepage}
\begin{theorem}\label{thm:sharp-hypersurfaces}
Let $D\subset\PP^n$ be a smooth hypersurface of degree $d\ge1$, where $n\ge1$. The following conditions are equivalent:
\begin{enumerate}[label=\textup{(\roman*)}]
\item $\PP^n\setminus D$ is Oka;
\item $\PP^n\setminus D$ is holomorphically elliptic;
\item $d\le n+1$.
\end{enumerate}
For $n=1$, smoothness means that $D$ consists of $d$ distinct points; the positive cases are $\C$ and $\C^*$.
\end{theorem}
\end{samepage}

For complements, Kusakabe treated the quadric case~\cite{Kus21} and Xie the smooth cubic case~\cite{X26}; the plane cubic case answered Problem~B of Forstneri\v{c} and L\'arusson~\cite{FL11}. Theorem~\ref{thm:sharp-hypersurfaces} resolves Forstneri\v{c}'s Problem~5.3(a) for smooth hypersurfaces~\cite{Forstneric23} and settles the complement part of Problem~E of Forstneri\v{c} and L\'arusson~\cite{FL11} for every smooth hypersurface, with the exact threshold $d\leq n+1$. At the critical degree $d=n+1$, successive polar equations produce families of lines meeting $D$ in at most two distinct points, whose complements are parametrized by entire curves. For $d\geq n+2$, the rank obstruction of Carlson and Griffiths~\cite{CarlsonGriffiths} excludes the Oka property.

\subsection*{Relation to the work of OpenAI}
Our proof that \LAP{} implies the Oka property builds on
the local-to-global strategy of OpenAI \cite{OAI26}: local planar
extension with passive holomorphic parameters leads, through
polynomial source coordinates, to polydisc approximation and then to
convex approximation by removing real faces. The concepts of
transverse complete directions, expanding finite strips, and thin
transverse source coordinates are central inspirations. We also adapt
the center-gluing and boundary-disc propagation arguments used in
that framework. The formulation of \LAP{} is also inspired by the local extension condition introduced in~\cite{OAI26}; Appendix~\ref{app:L} proves their pointwise equivalence.

The contributions of the present authors are the formulation of
\LAP{} in arbitrary dimension, the proof that it implies the Oka
property, and the analytic refinements that make this criterion
effective. These include exact fiber-coordinate
normalization in scalar substitution, an explicit planar corridor collapse
with simultaneous sector gluing, and a near-identity estimate for
normalized inverse iterates that is uniform in the number of iterates.
For the convex step, one coarse entire spray fixes a complex collar;
polynomial approximation of its local fiber inverse then gives the
accuracy needed for ordinary bounded gluing. For the geometric applications, we combine this criterion with the
classical and recent results cited in their proofs.

\subsection*{How to read the proof}
The proof begins by using centered families to extend arbitrary maps with image in a small target neighborhood: scalar substitution and exact partial compositions provide this local operation. Polynomial source coordinates then give thin boundary fibers, and planar enlargement with transverse and spray variables retained as passive parameters gives polydisc approximation. A fixed annular overlap joins the enlarged family to the old map. Induction across real faces finally yields $\CAP$.

The paper follows this progression. Section~\ref{sec:local-geometry} treats centered entire families, graph coordinates, and center gluing. Section~\ref{scalar:section} develops scalar substitution and local planar extension; Sections~\ref{sec:source-geometry}--\ref{sec:convex} pass through polynomial source geometry and polydisc approximation to convex approximation. Subsections~\ref{sec:removability} and~\ref{sec:permanence} treat small exceptional sets, holomorphic images, complements, and group orbits. Sections~\ref{sec:rc} and~\ref{sec:divisor-complements} give the geometric applications, and Section~\ref{sec:conjectures} concludes with three conjectures and their compact K\"ahler analogues.

\section{Local geometry and gluing}\label{sec:local-geometry}

\subsection{The geometric condition behind strong dominability}
\label{family:section}

Let $Y$ be a complex manifold of dimension $n\geq1$. We use \LAP{} and its locus $G_Y$ from
Definition~\ref{family:condition}.

\begin{lemma}\label{family:entire}
Let $V\subset Y$ be an open neighborhood of $x$. Suppose $h_j:V\times\mathbb C\to Y$, $1\leq j\leq n$, are
holomorphic, $h_j(y,0)=y$, and the vectors $\partial_t h_j(x,0)$
form a basis of $T_xY$. Then $x\in G_Y$.
\end{lemma}
\begin{proof}
Immediate from the definition.
\end{proof}

\begin{proposition}\label{family:strong-domination}
Every strongly dominable complex manifold satisfies \LAP.
More precisely, a dominating entire map
at $x$ supplies centered entire families near $x$ whose initial
velocities form a basis.
\end{proposition}
\begin{proof}
Choose a holomorphic map $\Psi:\mathbb C^n\to Y$ and a point
$a\in\mathbb C^n$ such that $\Psi(a)=x$ and $d\Psi_a$ is an
isomorphism. Let $\lambda:V\to\mathbb C^n$ be its local inverse
on a neighborhood of $x$. For the standard coordinate vectors $e_j$,
set
\begin{equation*}
 h_j(y,t)=\Psi\bigl(\lambda(y)+te_j\bigr),
 \qquad (y,t)\in V\times\mathbb C.
\end{equation*}
These families are centered and their initial velocities at $x$ are
$d\Psi_a(e_j)$, a basis. If dominability is stated using a map from
$\mathbb C^N$ with surjective differential, restrict it first to an
affine $n$-plane on which that differential is an isomorphism.
\end{proof}

\subsection{Local coordinates and one-step gluing}\label{sec:tools}
To glue maps defined on two overlapping source domains, we allow each
map to vary in a small holomorphic family. The additional variable $v$
parametrizes small changes of the original map, which corresponds
to $v=0$. We then seek small holomorphic choices
$v=a(z)$ and $v=b(z)$ on the two domains so that the resulting maps
agree on the overlap. The lemmas below construct these coordinates,
split a holomorphic error into contributions from the two sides, and
solve for $a$ and $b$ by a contraction. The output is one glued map.

\subsubsection{Conventions and graph coordinates}
Throughout, $Y$ is a complex manifold of dimension $n$. We fix a
complete distance $d_Y$ inducing its topology; such a distance exists
by the theorem of Nomizu and Ozeki~\cite[Theorem~1]{NO61}, applied
componentwise. Completeness is used
only when taking an exhaustion limit. We abbreviate $\Delta=\Delta_1$
and use the maximum norm $\|v\|_\infty=\max_j|v_j|$ to write
$\Delta_r^q=\{v\in\C^q:\|v\|_\infty<r\}$. The notation
$\OO(\Op K,Y)$ means maps holomorphic on some neighborhood of $K$.
A closed polydisc of dimension zero is a point.

For a holomorphic map $f:U\to Y$ on a complex manifold $U$ and
$b>0$, a \emph{chart spray} centered at $f$ is a holomorphic family
$F:U\times\Delta_b^n\to Y$ with $F(z,0)=f(z)$ for which
$(z,v)\mapsto(z,F(z,v))$ is biholomorphic onto its image.
The following standard graph construction is used by
OpenAI~\cite[Lemma~2.6]{OAI26}.

\begin{lemma}\label{lem:chart-spray}
Let $S$ be Stein, $f:S\to Y$ holomorphic, and suppose
$f^*TY$ is holomorphically trivial. For every $K\Subset S$ there
are a neighborhood $U$ of $K$, a number $b>0$, and a chart spray
$F:U\times\Delta_b^n\to Y$ centered at $f$.
\end{lemma}
\begin{proof}
The graph of $f$ is a Stein submanifold of $S\times Y$ and has a
Stein neighborhood by Siu's theorem~\cite[Main Theorem]{Siu76}. A frame of its vertical tangent bundle extends,
after shrinking that neighborhood, to holomorphic vertical vector
fields: apply Cartan's theorem~B to the restriction sequence for the
vertical tangent sheaf. Compose their local flows. Near the zero
section this gives a fiber-preserving holomorphic map whose vertical
differential is the chosen frame. The inverse function theorem and
compactness of $K$ give one neighborhood $U$ and one fiber radius on
which the graph map is biholomorphic.
\end{proof}

In the applications, the source is a contractible Stein domain or a
planar domain times a polydisc. A planar domain has the homotopy type
of a one-dimensional CW complex, so complex vector bundles on it are
topologically trivial. The Oka--Grauert principle~\cite[Satz~I]{Grauert58}
therefore provides the triviality required in
Lemma~\ref{lem:chart-spray}.
When using the inverse of a chart spray $F$, we restrict the source
to a compact subset $K_0\Subset U$ and the fiber variable to
$\overline\Delta_{b'}^n$ with $b'<b$. The image of
$K_0\times\overline\Delta_{b'}^n$ under $(z,v)\mapsto(z,F(z,v))$
is then a compact subset of the open chart image. Consequently,
sufficiently small perturbations of these values still lie in that
image, so the inverse remains defined uniformly for $z\in K_0$.

\subsubsection{Bounded additive splitting}
For an open set $U\subset\C^m$, write $H^\infty(U,\C^n)$ for the
Banach space of bounded holomorphic maps $u:U\to\C^n$, with
\[
 \|u\|_U=\sup_{z\in U}\max_{1\leq j\leq n}|u_j(z)|.
\]
Derivatives are measured in the operator norm induced by the maximum
norm. We use pairs of bounded open sets $A,B\subset\C^m$ whose
closed fringes $\overline{A\setminus B}$ and
$\overline{B\setminus A}$ are disjoint. Set $C=A\cap B$.

\begin{lemma}\label{lem:additive-splitting}
Suppose $A\cup B$ is a smoothly bounded strongly pseudoconvex domain.
There are a constant $M>0$ and bounded linear operators
\[
 R_A:H^\infty(C,\C^n)\to H^\infty(A,\C^n),\qquad
 R_B:H^\infty(C,\C^n)\to H^\infty(B,\C^n)
\]
with
\[
 R_Au-R_Bu=u,\qquad
 \|R_Au\|_A+\|R_Bu\|_B\leq M\|u\|_C.
\]
For bounded open sets $A,B\subset\C$ with the same separation
condition, the conclusion also holds on $A\times P,B\times P$
for every open polydisc $P$; the operators preserve
holomorphic dependence on $P$, and their bound depends only on $A,B$.
\end{lemma}
\begin{proof}
The first assertion is the classical additive Cousin splitting in
\cite[Lemma~4.6]{F03}. For completeness, set $\Omega=A\cup B$ and
choose a smooth function $\chi:\C^m\to[0,1]$ with compact support,
equal to $0$ near $\overline{A\setminus B}$ and to $1$ near
$\overline{B\setminus A}$. For $u\in H^\infty(C,\C^n)$, the functions
$\chi u$ and $(\chi-1)u$ extend smoothly by zero to $A$ and $B$,
respectively. Their $\bar\partial$ derivatives agree with the
bounded, $\bar\partial$-closed form $\eta=u\bar\partial\chi$,
extended by zero from $C$ to $\Omega$. A bounded linear
$\bar\partial$ solution operator on $\Omega$ gives $v$ with
$\bar\partial v=\eta$. Hence
\[
 R_Au=\chi u-v|_A,\qquad R_Bu=(\chi-1)u-v|_B
\]
are holomorphic, differ by $u$ on $C$, and satisfy the required
bound; the products denote their zero extensions.

For the planar assertion, use the Cauchy--Green operator on the
coefficient $u\bar\partial\chi$, extended by zero to the plane;
see \cite[Lemma~3.2 and the proof of Lemma~3.3]{FArak19}.
The kernel estimate gives a bound depending only on the fixed planar
sets. Since the cutoff and integration involve only the planar
variable, the operators preserve holomorphic dependence on the
parameters in $P$, with the same bound for every $P$ and its
dimension. Apply the construction componentwise.
\end{proof}

\subsubsection{Only the centers need to be glued}
The following form of OpenAI's center-gluing argument~\cite[Lemma~2.8]{OAI26}
is used throughout the paper. It separates the geometry, which
determines the splitting bound $M$, from the accuracy of the families.

\begin{lemma}\label{lem:center-gluing}
Suppose the operators of Lemma~\ref{lem:additive-splitting} are fixed.
Let $F_A$ be a holomorphic chart spray on $A\times\Delta_{\rho_+}^n$
and $F_B$ a holomorphic family on $B\times\Delta_\rho^n$,
where $\rho_+>\rho>0$. Suppose their transition is defined on
$C\times\Delta_\rho^n$ and has the form
\[
 F_A(z,\cdot)^{-1}(F_B(z,v))=v+c(z,v).
\]
Set $e=\|c\|_\infty$ and $\kappa=\|D_vc\|_\infty$ there.
If $Me<\rho/2$ and $M\kappa<1/2$, there are holomorphic corrections
$a$ on $A$ and $b$ on $B$ satisfying
\[
 F_A(z,a(z))=F_B(z,b(z))\quad(z\in C),\qquad
 \|a\|_A+\|b\|_B\leq Me.
\]
The same conclusion holds if the families also depend holomorphically
on $p\in P$, provided the splitting operators act on functions of
$(z,p)$, preserve holomorphic dependence on $p$, and have the same
bound $M$ uniformly on $P$. Here the radii are independent of $p$,
and $e$ and $\kappa$ are the supremum norms on
$C\times P\times\Delta_\rho^n$, subject to the same inequalities.
The corrections are then holomorphic
functions $a(z,p)$ and $b(z,p)$, and the estimates are uniform in $p$.
Lemma~\ref{lem:additive-splitting} supplies such operators for planar
$A,B$.
\end{lemma}
\begin{proof}
On the closed ball $\|u\|_C\leq e$ in $H^\infty(C,\C^n)$, set
\[
 \mathcal T(u)(z)=c(z,R_Bu(z)).
\]
The first inequality keeps $R_Bu$ in the prescribed fiber domain.
We have $\|\mathcal T(u)\|_C\leq e$ and
$\|\mathcal T(u)-\mathcal T(\widetilde u)\|_C
\leq M\kappa\|u-\widetilde u\|_C$.
At the fixed point set $a=R_Au$, $b=R_Bu$. Then
$a-b=u=c(z,b)$, which is exactly the transition identity required
for the centers to agree. The norm estimate follows from the linear
splitting bound.
\end{proof}

Consequently the glued map changes the old center by an amount
tending to zero with $e$, uniformly on every retained compact set.
This is enough to iterate: in our applications the enlarged compact
has a contractible Stein neighborhood inside the domain of the glued
map. Its pulled-back tangent bundle is therefore trivial, and
Lemma~\ref{lem:chart-spray} constructs a new chart spray there for
the next step, so the gluing proceeds without carrying a spray across
the overlap. Inverse fiber charts are required only for the old family.
If the transition is
uniformly small on a larger fiber polydisc, its derivative is small
on a smaller one by Cauchy estimates; this is how the hypothesis on
$\kappa$ will always be obtained.

\section{From disc families to local extension}
\label{scalar:section}

We first work in one complex source coordinate, with the remaining
coordinates serving as holomorphic parameters. The goal is to extend
approximation from an old disc to a larger disc by gluing along a thin
overlap. This is the gluing idea underlying OpenAI's two-strip
argument~\cite[Proposition~2.2]{OAI26}, adapted below to centered disc families.

\begin{definition}\label{def:special-planar-pair}
A pair $D\subset D'\subset\C$ is a \emph{special planar pair}
\cite[Section~2, before Definition~2.1]{OAI26} if $D$ and $D'$ are
closed Jordan discs with piecewise smooth boundaries and there is a
closed Jordan disc $E\subset\C$ with piecewise smooth boundary such
that $D'=D\cup E$ and $D\cap E$ is homeomorphic to $[0,1]$ and is
a proper subset of both $\partial D$ and $\partial E$.
\end{definition}

We use slightly larger neighborhoods when making source splittings.

The model to keep in mind is a centered family
$h_\nu:V\times\Delta_{R_\nu}\to Y$, with $h_\nu(y,0)=y$.
Here $s$ is the source point, $w$ allows a small change of the old
map, and $t$ is the scalar parameter to be replaced by a holomorphic
function of $s$. If $G(s,w)$ is a small chart spray around an old map
$f_0$, with image in $V$, then the families below are
\[
 T_\nu(s,w,t)=h_\nu(G(s,w),t).
\]
The approximation class in the lemma records which
maps on the old product can already be approximated from the enlarged
product.
The lemma says that substitution into the limiting family preserves
this approximation property.

\begin{lemma}\label{scalar:substitution}
Let $D\subset D'$ be a special planar pair and, for an integer $q\geq0$
and a closed polydisc $P\subset\mathbb C^q$, set $K=D\times P$ and
$L=D'\times P$. The class $\mathcal A$ consists of maps to $Y$ holomorphic near $K$ that are uniformly
approximable on $K$ by maps to $Y$ holomorphic near $L$.

Suppose $f_0\in\mathcal A$, $\dim Y=n$, and
\[
 G:\Omega\times\Delta_r^n\longrightarrow Y,\qquad G(s,0)=f_0(s),
\]
is a chart spray on a fixed neighborhood $\Omega$ of $K$; that is,
$(s,w)\mapsto(s,G(s,w))$ is biholomorphic onto its image. Suppose the holomorphic maps
\[
 T_\nu:\Omega\times\Delta_r^n\times\Delta_{R_\nu}\longrightarrow Y,
 \qquad T_\nu(s,w,0)=G(s,w),\qquad R_\nu\longrightarrow\infty,
\]
converge locally uniformly on $\Omega\times\Delta_r^n\times\Delta_\tau$
to $T$, for some $\tau>0$. If $\varphi$ is holomorphic near $K$ and
$\|\varphi\|_K<\tau$, then $T(\cdot,0,\varphi)\in\mathcal A$.
\end{lemma}

Although the conclusion only involves the zero section $w=0$, the
spray factor $\Delta_r^n$ is used twice in the following gluing proof.
An approximant $g$, holomorphic near $L$,
is represented near $K$ by the small section
$\rho(s)=G(s,\cdot)^{-1}(g(s))$, so that
$T(s,\rho(s),\varphi(s))$ remains close to the desired map.
For a chart spray $F(s,v)$ centered at $g$, the change of coordinates
$W(s,v)=G(s,\cdot)^{-1}(F(s,v))$ then gives
$Q_\nu(s,v,t)=T_\nu(s,W(s,v),t)$ with
$Q_\nu(s,v,0)=F(s,v)$ on a fixed smaller fiber polydisc.
This normalization on the whole fiber makes the overlap transition
close to the identity, with the derivative control needed for center
gluing; see also \cite[Proposition~3.1]{F06}. The resulting corrections
may take nonzero fiber values.

The base domains and fiber margins are fixed before choosing the
scalar polynomial $\psi$ and the index $\nu$.
The polynomial approximates $\varphi$ near $D\times P$ and zero near
$C\times P$, with the planar sets shown in
Figure~\ref{fig:planar-gluing-collar}. Local uniform convergence and
Cauchy estimates are used on these fixed products. Finally, $\nu$ is
chosen to give the required accuracy there and to make $R_\nu$ exceed
$\|\psi\|_{B^+\times P^+}$.

\begin{figure}[htbp]
\centering
\definecolor{collarOldPurple}{HTML}{8260A2}
\definecolor{collarTeal}{HTML}{2C847F}
\definecolor{collarOrange}{HTML}{C47832}
\definecolor{collarDeepPurple}{HTML}{613E83}
\definecolor{collarGold}{HTML}{C79742}
\begin{tikzpicture}[x=.64cm,y=.75cm,font=\footnotesize,
  line cap=round,line join=round,
  old/.style={draw=collarOldPurple,line width=.75pt},
  new/.style={draw=collarTeal,line width=.75pt},
  overlap/.style={draw=collarOrange,line width=.85pt},
  guide/.style={draw=figInk!65,line width=.4pt,->,>=latex},
  context/.style={draw=figInk!35,line width=.4pt},
  openold/.style={draw=collarDeepPurple,line width=.95pt,
    dash pattern=on 4.2pt off 2.4pt on 0pt off 2.4pt},
  opennew/.style={draw=collarTeal,line width=1pt,
    dash pattern=on 0pt off 3.2pt}]
\def\planarD{(0,0) rectangle (2.65,2)}
\def\planarE{(2.65,0) rectangle (6,2)}
\def\planarEnlarged{(0,0) rectangle (6,2)}
\def\planarA{(3.4,-.3) rectangle (6.3,2.3)}
\def\planarB{(-.3,-.3) rectangle (3.8,2.3)}
\def\planarC{(3.4,-.3) rectangle (3.8,2.3)}
\def\planarDplus{(-.18,-.18) rectangle (2.83,2.18)}
\def\planarCplus{(3.08,-.62) rectangle (4.12,2.62)}
\def\planarBplus{(-.85,-.85) rectangle (4.35,2.85)}
\def\planarUA{(3.3,-.4) rectangle (6.4,2.4)}
\def\planarUB{(-.4,-.4) rectangle (3.9,2.4)}
\def\planarUoverlap{(3.3,-.4) rectangle (3.9,2.4)}
\def\planarFrame{\path[use as bounding box] (-.95,-2.2) rectangle (6.95,3.95);}
\begin{scope}
\planarFrame
\fill[collarOldPurple!10] \planarB;
\fill[collarTeal!10] \planarA;
\fill[collarOldPurple!20] \planarD;
\fill[collarTeal!20] \planarE;
\fill[collarOrange!30] \planarC;
\draw[old] \planarB;
\draw[new] \planarA;
\draw[figInk,line width=.65pt] \planarEnlarged;
\draw[figInk!55,dashed] (2.65,0) -- (2.65,2);
\draw[collarOrange,densely dashed] (3.6,0) -- (3.6,2);
\node[text=collarOldPurple!85!black] at (1.3,1) {$D$};
\node[text=collarTeal!85!black] at (5,1) {$E$};
\node[text=collarOldPurple!85!black] at (1.3,2.7) {$B$};
\node[text=collarTeal!85!black] at (5.2,2.7) {$A$};
\node at (3.6,3.55) {crosscut};
\draw[guide] (3.6,3.3) -- (3.6,2.05);
\node[text=collarOrange!85!black] at (2.75,-1.35) {$C=A\cap B$};
\draw[guide] (3.5,-1.12) -- (3.6,-.35);
\node at (5.7,-1.35) {$D'$};
\draw[guide] (5.7,-1.12) -- (5.7,-.03);
\node at (3,-2.05) {(a) Gluing sides};
\end{scope}
\begin{scope}[shift={(8,0)}]
\planarFrame
\fill[collarOldPurple!5] \planarBplus;
\fill[collarOldPurple!20] \planarD;
\fill[collarTeal!20] \planarE;
\fill[collarOrange!9] \planarCplus;
\fill[collarOrange!30] \planarC;
\draw[new] \planarA;
\draw[figInk,line width=.65pt] \planarEnlarged;
\draw[old] \planarB;
\draw[old,dashed] \planarBplus;
\draw[old,dashed] \planarDplus;
\draw[overlap,dashed] \planarCplus;
\draw[figInk,line width=.65pt] \planarD;
\draw[overlap] \planarC;
\node[text=collarOldPurple!85!black] at (1.3,1) {$D$};
\node[text=collarTeal!85!black] at (5,1) {$E$};
\node[text=collarOldPurple!85!black] at (1.3,3.55) {$D^+$};
\draw[guide] (1.3,3.28) -- (1.3,2.21);
\node[text=collarOrange!85!black] at (3.6,3.55) {$C^+$};
\draw[guide] (3.6,3.28) -- (3.6,2.65);
\node[text=collarOldPurple!85!black,anchor=west] at (4.60,2.85) {$B^+$};
\draw[guide] (4.58,2.85) -- (4.38,2.85);
\node[text=collarOldPurple!85!black] at (1.3,-1.35) {$B$};
\draw[guide] (1.3,-1.12) -- (1.3,-.33);
\node[text=collarOrange!85!black] at (3.6,-1.35) {$C$};
\draw[guide] (3.6,-1.12) -- (3.6,-.33);
\node at (3,-2.05) {(b) Compact neighborhoods};
\end{scope}
\begin{scope}[shift={(16,0)}]
\planarFrame
\fill[collarOldPurple!5] \planarBplus;
\fill[collarDeepPurple!13] \planarUB;
\fill[collarTeal!12] \planarUA;
\fill[collarOldPurple!20] \planarD;
\fill[collarTeal!20] \planarE;
\fill[collarOrange!9] \planarCplus;
\fill[collarGold!38] \planarUoverlap;
\draw[old,dashed] \planarBplus;
\draw[overlap,dashed] \planarCplus;
\draw[openold] (3.3,-.4) -- (-.4,-.4) -- (-.4,2.4) -- (3.3,2.4);
\draw[openold] (3.9,-.4) -- (3.9,2.4);
\draw[opennew] (3.9,-.4) -- (6.4,-.4) -- (6.4,2.4) -- (3.9,2.4);
\draw[opennew] (3.3,-.4) -- (3.3,2.4);
\foreach \edgeY in {-.4,2.4}{
  \draw[openold,dash pattern=on 4.2pt off 4.3pt on 0pt off 2.4pt]
    (3.3,\edgeY) -- (3.9,\edgeY);
  \fill[collarTeal] ([xshift=6.35pt]3.3,\edgeY) circle[radius=.5pt];
}
\draw[figInk,line width=.65pt] \planarEnlarged;
\draw[figInk!70,dashed,line width=.5pt] (2.65,0) -- (2.65,2);
\node[text=collarDeepPurple!85!black,fill=white,
  inner xsep=1.2pt,inner ysep=.5pt] at (1.3,2.4) {$U_B$};
\node[text=collarOldPurple!85!black] at (1.3,1) {$D$};
\node[text=collarTeal!85!black,fill=white,
  inner xsep=1.2pt,inner ysep=.5pt] at (5.3,2.4) {$U_A$};
\node[text=collarTeal!85!black] at (5.3,1) {$E$};
\node[text=collarOldPurple!85!black] at (1.3,3.55) {$B^+$};
\draw[guide] (1.3,3.28) -- (1.3,2.88);
\node[text=collarOrange!85!black] at (3.6,3.55) {$C^+$};
\draw[guide] (3.6,3.28) -- (3.6,2.65);
\node[text=collarGold!70!black] at (2.75,-1.35) {$U_A\cap U_B$};
\draw[guide] (3.5,-1.12) -- (3.6,-.43);
\node at (5.7,-1.35) {$D'$};
\draw[guide] (5.7,-1.12) -- (5.7,-.03);
\node at (3,-2.05) {(c) Open gluing domains};
\end{scope}
\node[anchor=east] at (8.1,-3.65) {Passive parameters};
\draw[old,dashed] (9,-4.55) rectangle (14.4,-2.8);
\draw[opennew] (9.35,-4.3) rectangle (13.15,-3.05);
\draw[old,fill=collarOldPurple!14] (9.7,-4.02) rectangle (11.7,-3.35);
\node[text=collarOldPurple!85!black] at (10.7,-3.68) {$P$};
\node[text=collarTeal!85!black] at (12.4,-3.66) {$U_P$};
\node[text=collarOldPurple!85!black] at (13.75,-3.66) {$P^+$};
\end{tikzpicture}
\caption{The planar choices in Lemma~\ref{scalar:substitution},
with identical coordinates and scale in all three panels.
(a) $D'=D\cup E$ is covered by the thickened sides $A,B$, and
$C=A\cap B$ is disjoint from $D$.
(b) The disjoint compact neighborhoods $D^+,C^+$ and the set $B$
lie in $\operatorname{Int}B^+$.
(c) The open domains satisfy $D'\subset U_A\cup U_B$,
$U_A\cap U_B\Subset\operatorname{Int}C^+$, and
$U_B\Subset\operatorname{Int}B^+$.
The separate inset shows $P\Subset U_P\Subset\operatorname{Int}P^+$;
dots mark $U_A,U_P$; dash-dots mark $U_B$.}
\label{fig:planar-gluing-collar}
\end{figure}
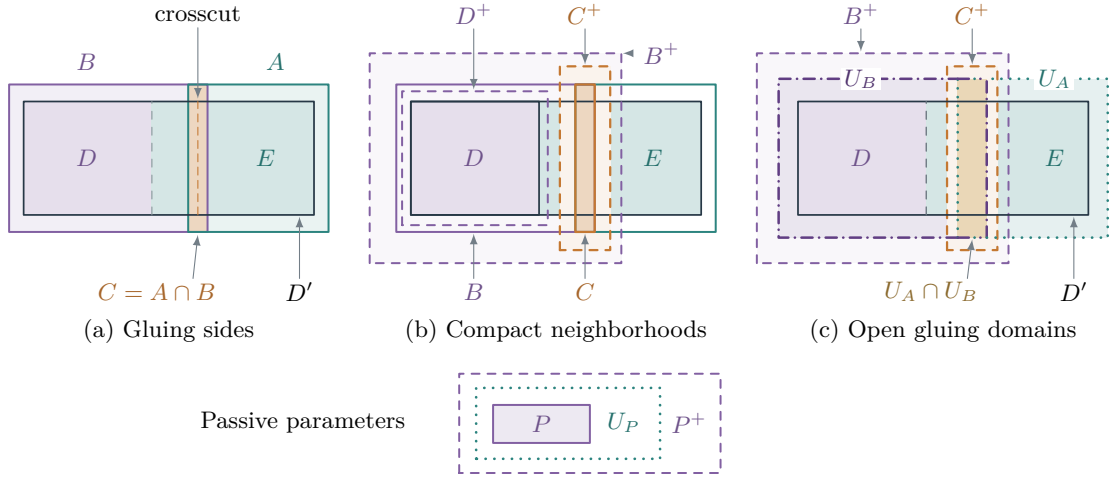

\begin{proof}
\noindent\emph{Step 1: fix the old approximant, domains, and normalization.}
Fix $\epsilon>0$ and $\|\varphi\|_K<t_1<\tau$. Choose a map $g$
holomorphic near $L$, sufficiently close to $f_0$ on $K$ that
\[
 \rho(s)=G(s,\cdot)^{-1}(g(s))
\]
is holomorphic near $K$, has values compactly inside $\Delta_r^n$, and
\begin{equation*}
 \sup_K d_Y\bigl(T(s,\rho(s),\varphi(s)),T(s,0,\varphi(s))\bigr)
 <\epsilon/3.
\end{equation*}
This follows from compactness and continuity of the inverse chart and
of $T$.

A crosscut in the attached disc, close to but disjoint from $D$,
gives a thin closed collar $B$ containing $D$ and a complementary
closed region $A$, with slightly thickened boundaries, such that
$D'\subset\operatorname{Int}(A\cup B)$ and the overlap $C=A\cap B$
is a closed disc disjoint from $D$; see Figure~\ref{fig:planar-gluing-collar}(a).
Take the collar sufficiently thin
that the original $G$-chart and $\rho$ are available near its product
with $P$. Choose disjoint closed disc neighborhoods $D^+,C^+$ of
$D,C$, a slightly larger closed neighborhood $B^+$ of $B$ containing
them, as in Figure~\ref{fig:planar-gluing-collar}(b), and fixed polydiscs
in the remaining $q$ source coordinates
\[
 P\Subset\operatorname{Int}P^+
   \Subset\operatorname{Int}P^{++}.
\]
Here $\operatorname{Int}$ denotes topological interior.
These coordinates are \emph{passive} only in the geometric construction:
we enlarge the planar disc while keeping these polydiscs fixed.
The maps may still depend holomorphically on all source coordinates.
Make these choices inside the relevant holomorphic domains, with
$B^+\times P^{++}\Subset\Omega$, with $g$ in the image of the
$G$-chart there, and with $\varphi$ holomorphic near
$D^+\times P^+$. Fix separated open gluing domains $U_A,U_B$ and
an open polydisc $U_P$ satisfying
\[
 D'\subset U_A\cup U_B,\quad
 P\Subset U_P\Subset\operatorname{Int}P^+,\quad
 U_A\cap U_B\Subset\operatorname{Int}C^+,\quad
 U_B\Subset\operatorname{Int}B^+.
\]
Figure~\ref{fig:planar-gluing-collar}(c) shows the planar choices.
Their planar additive-splitting constant is now fixed.

Take a chart spray $F(s,v)$ centered at $g$ over a neighborhood of
the two sides. Lemma~\ref{lem:chart-spray} applies on a sufficiently small
contractible Stein neighborhood of the disc product. Choose a fixed $b>0$ so that $F$ is available for
$v\in\Delta_{4b}^n$ and
\[
 W(s,v)=G(s,\cdot)^{-1}(F(s,v))
\]
is holomorphic near
$B^+\times P^{++}\times\overline\Delta_{3b}^{\,n}$, with image
compactly inside $\Delta_r^n$. Here $W(s,0)=\rho(s)$. Set
\[
 Q_\nu(s,v,t)=T_\nu(s,W(s,v),t),\qquad
 Q(s,v,t)=T(s,W(s,v),t).
\]
These families converge locally uniformly for $|t|<\tau$, and,
for every allowed fiber value,
\begin{equation}\label{scalar:exact}
 Q_\nu(s,v,0)=Q(s,v,0)=F(s,v).
\end{equation}
Normalization on the whole fiber, rather than only at its center,
makes the overlap transition for the limiting family tend to the
identity as the scalar parameter tends to zero. The fixed fiber
margin then gives derivative control by Cauchy estimates.

\smallskip
\noindent\emph{Step 2: choose the accuracies, then the scalar polynomial.}
By uniform continuity of $Q$, choose $0<\alpha<b$ and $\xi>0$ so
that, for $s\in K$, $\|v\|\leq\alpha$, and
$|u-\varphi(s)|<\xi$, one has $|u|<t_1$ and
\begin{equation*}
 d_Y\bigl(Q(s,v,u),Q(s,0,\varphi(s))\bigr)<\epsilon/3.
\end{equation*}
Choose a transition tolerance $e>0$ so small that a transition
within $e$ of the identity on the fixed overlap and $\Delta_{2b}^n$
admits zero-section gluing corrections of norm less than $\alpha$.
This is Lemma~\ref{lem:center-gluing}; its derivative hypothesis
on $\Delta_b^n$ follows from Cauchy estimates on the larger fiber
polydisc.

Let $\Psi_s=F(s,\cdot)^{-1}$. By
\eqref{scalar:exact}, compactness, and the inverse-chart
margins, choose $0<\eta<\tau$ such that
$\Psi_s(Q(s,v,t))$ is defined and within $e/2$ of $v$ on
$C^+\times P^+\times\overline\Delta_{2b}^{\,n}$ whenever
$|t|\leq\eta$. Local uniform convergence gives the same assertion
with $Q_\nu$ and tolerance $e$ for every sufficiently large $\nu$,
on this same fixed compact product.

The compact set $(D^+\cup C^+)\times P^+$ is polynomially convex.
Oka--Weil approximation \cite[\S2.3]{F17} therefore supplies a polynomial
$\psi:\C^{1+q}\to\C$ such that
\begin{equation}\label{scalar:polynomial}
 \|\psi-\varphi\|_{D^+\times P^+}<\xi,
 \qquad \|\psi\|_{C^+\times P^+}<\eta.
\end{equation}
\smallskip
\noindent\emph{Step 3: choose the family index last, then glue.}
Fix $\psi$ and choose $\nu$ so large that the preceding
overlap estimates hold,
\begin{equation}\label{scalar:radius}
 R_\nu>1+\|\psi\|_{B^+\times P^+},
\end{equation}
and
\begin{equation*}
 \sup_{K\times\overline\Delta_b^{\,n}\times\overline\Delta_{t_1}}
 d_Y(Q_\nu,Q)<\epsilon/3.
\end{equation*}
The convergence estimate is on a compact product fixed before
$\psi$. Its possibly large values elsewhere on $B^+\times P^+$
enter only the radius requirement~\eqref{scalar:radius}.

Use the two families
\[
 F_A(s,v)=F(s,v),\qquad F_B(s,v)=Q_\nu(s,v,\psi(s)).
\]
The second is defined throughout the fixed $B$ side by
\eqref{scalar:radius}. On the overlap its transition into
the $F_A$-chart is within $e$ of the identity by
\eqref{scalar:polynomial}. Lemma~\ref{lem:center-gluing} provides holomorphic corrections $a$ and $b_*$, both of norm less
than $\alpha$, with
\[
 F_A(s,a(s))=F_B(s,b_*(s)).
\]
Thus the two expressions define $\widetilde g$ holomorphic on
$(U_A\cup U_B)\times U_P$, a neighborhood of $L$. On $K$,
\[
 \widetilde g(s)=Q_\nu(s,b_*(s),\psi(s)).
\]
The three error terms below are each less than $\epsilon/3$ by the
choices above:
\begin{align*}
 &d_Y\bigl(\widetilde g(s),T(s,0,\varphi(s))\bigr)\\
 &\quad\leq
 d_Y\bigl(Q_\nu(s,b_*(s),\psi(s)),Q(s,b_*(s),\psi(s))\bigr)\\
 &\qquad+
 d_Y\bigl(Q(s,b_*(s),\psi(s)),Q(s,0,\varphi(s))\bigr)\\
 &\qquad+
 d_Y\bigl(T(s,\rho(s),\varphi(s)),T(s,0,\varphi(s))\bigr)
 <\epsilon.\qedhere
\end{align*}
\end{proof}

\begin{remark}\label{scalar:entire}
For an actual centered family $h:V\times\mathbb C\to Y$, take
$T_\nu(s,w,t)=h(G(s,w),t)$ on increasing scalar discs. The lemma
then permits every scalar function holomorphic near $K$: choose
$\tau>\|\varphi\|_K$. The family $F_B$ need neither dominate
nor remain in the old fiber chart away from the overlap.
\end{remark}

\subsection{Local extension by exact partial compositions}

\begin{lemma}\label{family:local-sector}
For every $x$ in the \LAP{} locus $G_Y$, there is an open neighborhood $W_x\ni x$ with
the following property: if $D\subset D'$ is a special planar pair,
$P\subset\mathbb C^q$ is a closed polydisc, and $k$ is holomorphic
near $D\times P$ with values in $W_x$, then $k$ is uniformly
approximable on $D\times P$ by maps holomorphic near $D'\times P$
with values in $Y$.
\end{lemma}
\begin{proof}
For $a=(a_1,\ldots,a_n)$ in a sufficiently small polydisc
$Q\subset\C^n$ about $0$, use the families of
Definition~\ref{family:condition} to form the exact partial compositions
\[
 \theta_0(a)=x,\qquad
 \theta_j(a)=h_j(\theta_{j-1}(a),a_j),\qquad
 \Theta(a)=\theta_n(a).
\]
Shrink $Q$ so that its closure lies in
their domains, every partial image is relatively compact in $V$,
and $|a_j|<\tau$ on $\overline Q$. Since $h_j(y,0)=y$, one has
$d_yh_j(y,0)=\operatorname{id}$. The columns of $d\Theta_0$ are
therefore $\partial_t h_j(x,0)$. The inverse function theorem gives
a neighborhood $W_x$ whose closure lies in an inverse chart for
$\Theta$, with $\Theta^{-1}(\overline W_x)\Subset Q$.

For the given $k$, write $a=\Theta^{-1}\circ k$ and set
\[
 k_0=x,\qquad k_j=h_j(k_{j-1},a_j),\qquad k_n=k.
\]
Let $\mathcal A$ be the approximation class in
Lemma~\ref{scalar:substitution} for this fixed pair and this fixed
$P$. The constant map $k_0$ belongs to $\mathcal A$. Suppose
$k_{j-1}\in\mathcal A$. By Lemma~\ref{lem:chart-spray}, choose a
chart spray $H_j(s,w)$ centered at $k_{j-1}$. Compactness permits
fixed base and fiber neighborhoods on which its image lies in $V$.
On these neighborhoods set
\[
 T_{j,\nu}(s,w,t)=h_{j,\nu}(H_j(s,w),t).
\]
The families have exact normalization
$T_{j,\nu}(s,w,0)=H_j(s,w)$ and converge locally uniformly for
$|t|<\tau$. Since $\|a_j\|_{D\times P}<\tau$,
Lemma~\ref{scalar:substitution} gives $k_j\in\mathcal A$.
Induction proves the result.\qedhere
\end{proof}

\section{Polynomial source geometry}\label{sec:source-geometry}

Rosay--Rudin's attracting-basin construction~\cite[Section~9]{RR88}
obtains coordinates by compensating holomorphic iterates with inverse
iterates of the linear part.
The maps below are scaled compositions of automorphisms of the
shift-like type studied by Bedford and Pambuccian \cite[Section~1]{BP98}.
The finite-iterate source construction and its thin-fiber geometry
are adapted from OpenAI~\cite[Lemma~3.3 and the proof of Proposition~3.4]{OAI26}. We give the estimates needed
for polydisc extension, including a bound for the normalized inverse
iterates that is uniform in the number of iterates.

The source construction has two geometric goals: to move points only
slightly on the old polydisc, and to provide coordinates on the outer
region in which each transverse slice over a scalar boundary point
has small diameter. Figure~\ref{fig:source-transport} illustrates these
two features. The first preserves the approximation already obtained.
The second ensures that, after applying the given map and allowing a
small spray variation, each such slice lies in a single target
neighborhood where the local extension lemma applies.

The roles of two parameters deserve emphasis. The polynomial degree
controls the near-identity error and the thickness of the boundary
slices, while the number of iterates allows us to reach any prescribed
outer radius. The estimates are uniform in the number of iterates,
so these parameters can be chosen in that order.

\begin{lemma}\label{lem:source-geometry}
Fix $m\ge2$ and numbers
\[
 \begin{gathered}
 1<r_m<\cdots<r_2<r_1<c<c_+<b_0,\qquad 0<a<1,\\
 r_j<a r_{j-1}\quad(2\le j\le m).
 \end{gathered}
\]
For all sufficiently large integers $d$, set
\[
 \begin{gathered}
 P_j(w)=(w/r_j)^d,\qquad L(x)=(a x_m,-a x_1,\ldots,-a x_{m-1}),\\
 H_d(x)=(a x_m,P_2(x_2)-a x_1,\ldots,P_m(x_m)-a x_{m-1}).
 \end{gathered}
\]
There are constants $C\ge1$ and $B>0$, independent of $d$ and $J$,
with the following properties: the map $H_d$ is a polynomial
automorphism. For every $J\ge1$, let
\[
 E_{d,J}=H_d^{-J}(\overline\Delta_1^m),\qquad
 h_{d,J}=H_d^{-J}\circ L^J.
\]
Then:
\begin{enumerate}
\item On $E_{d,J}$ one has
\begin{equation}\label{eq:source-power-bound}
 |P_j(x_j)|\le C\max(1,|x_1|)\quad(2\le j\le m).
\end{equation}
Moreover, if $|x_1|\le c_+$, then $x\in\Delta_{b_0}^m$.
\item On a neighborhood of
$E_{d,J}\cap\{|x_1|>r_1\}$ there are single-valued holomorphic
functions $s,\xi_2,\ldots,\xi_m$ such that $|\xi_j|\le B$ on
that set and
\begin{equation}\label{eq:source-monomial}
 \begin{gathered}
 x=\Phi_d(s,\xi),\qquad
 \Phi_{d,j}(s,\xi)=\rho_j s^{d^{m-j}}
 \exp\left(\sum_{k=2}^j\frac{\xi_k}{d^{j-k+1}}\right),
\\
 \rho_1=1,\quad \rho_j=r_j(a\rho_{j-1})^{1/d}.
 \end{gathered}
\end{equation}
The sum for $j=1$ is empty; thus $x_1=s^{D_d}$, where $D_d=d^{m-1}$.
The map $\Phi_d$ is entire in $(s,\xi)$.
\item Set
$\Xi_i=\{\xi\in\mathbb C^{m-1}:|\xi_j|\le B+i\}$ for $i=1,2$,
and $r=c^{1/D_d}$. Then
\[
 \Phi_d(\overline\Delta_r\times\Xi_2)\Subset\Delta_{b_0}^m,
 \qquad
 \sup_{|s|=r}\operatorname{diam}_{\infty}
       \Phi_d(\{s\}\times\Xi_2)=O(d^{-1}).
\]
The constants in these assertions do not depend on $J$.
\item For any fixed $1<u_0<u_1<r_m$, set
$\theta=u_1/r_m$ and $\alpha_d=a+\theta^d/u_1$. For sufficiently
large $d$, simultaneously for every $J\ge1$,
\begin{equation}\label{eq:source-uniform-near-identity}
 \|h_{d,J}-\operatorname{id}\|_{\overline\Delta_{u_0}^m}
 \le\frac{\theta^d}{a-\alpha_d^d}\le\frac{2}{a}\theta^d.
\end{equation}
For any $R_+>1$, one may first choose $d$ to obtain
any prescribed near-identity accuracy and thinness, and then choose
$J$ with $a^JR_+<1$. The resulting $h=h_{d,J}$ satisfies
\[
 h(\overline\Delta_{R_+}^m)\subset
 H_d^{-J}(\Delta_1^m)\subset E_{d,J}.
\]
Moreover, for sufficiently large $d$, $|h_1|<c$ on
$\overline\Delta_1^m$, uniformly in $J$.
\end{enumerate}
All inclusions asserted on compact sets have a positive margin, so
they persist on a sufficiently small open neighborhood of those sets.
\end{lemma}

The proof is completed at the end of this section, after Lemma~\ref{lem:source-near-identity}.

In the extension argument, assertion~(1) keeps the inner part of the
transported source in the domain of the old map, while assertion~(2)
expresses its outer part in one scalar coordinate and passive variables.
Assertion~(3) places each boundary fiber in a target neighborhood where
local planar extension applies. The resulting extension is pulled back
and glued to the old map on an inner overlap; assertion~(4) controls
the change on the original compact set.

In dimension two, the outer relation and its parametrization are
\[
 (x_2/r_2)^d=a x_1e^{\xi_2},\qquad
 \Phi_d(s,\xi_2)=(s^d,r_2a^{1/d}s e^{\xi_2/d}).
\]
Thus $s$ runs along the end, while bounded changes in $\xi_2$
move points in a fixed boundary fiber by only $O(d^{-1})$. In higher dimension,
successive relations of this form produce the weighted powers of
$s$ in \eqref{eq:source-monomial}.

\begin{figure}[ht]
\centering
\begin{tikzpicture}[x=1cm,y=1cm,font=\small,
  line cap=round,line join=round,
  arrow/.style={->,>=latex,draw=figInk,line width=.65pt},
  sourcebox/.style={draw=figInk!65,rounded corners=3pt,
    minimum height=.8cm,inner xsep=8pt,fill=figOld!12}]
  \node[font=\small\bfseries] at (2.35,2.7) {(a) Normalized transport};
  \node[sourcebox] (large) at (.45,1.35) {$\overline\Delta_{R_+}^{\,m}$};
  \node[sourcebox,fill=figNew!12] (inverse) at (4.35,1.35) {$E_{d,J}$};
  \node[sourcebox] (unit) at (4.35,-.45) {$\overline\Delta_1^{\,m}$};
  \draw[arrow] (large) -- node[above=3pt] {$h_{d,J}$} (inverse);
  \draw[arrow] (inverse) -- node[right=4pt] {$H_d^J$} (unit);
  \draw[arrow] (large.south) -- node[above right=2pt] {$L^J$} (unit.west);
  \node[text=figOld!70!black] at (2.3,-1.35)
    {$h_{d,J}\approx\operatorname{id}\quad\hbox{on }\overline\Delta_1^m$};
  \node[font=\small\bfseries] at (9.65,2.7) {(b) One boundary fiber};
  \path[draw=figOld,fill=figOld!12,rounded corners=2pt]
    (6.65,.05) rectangle (8.1,1.95);
  \node[text=figInk] at (7.375,1) {$\{s\}\times\Xi_2$};
  \node[text=figInk,font=\footnotesize] at (7.375,-.25) {$|s|=r$ fixed};
  \draw[arrow] (8.35,1) -- node[above=3pt] {$\Phi_d$} (10.35,1);
  \path[draw=figNew!85!black,fill=figNew!25,line width=.65pt]
    (10.95,1) ellipse[x radius=.12,y radius=.34];
  \node[text=figInk] at (10.95,1.9) {$\Phi_d(\{s\}\times\Xi_2)$};
  \draw[figOverlap,line width=.65pt] (11.40,.66) -- (11.40,1.34);
  \draw[figOverlap,line width=.65pt] (11.34,.66) -- (11.46,.66);
  \draw[figOverlap,line width=.65pt] (11.34,1.34) -- (11.46,1.34);
  \node[anchor=west,text=figOverlap!75!black,font=\footnotesize]
    at (11.54,1) {$O(d^{-1})$};
  \node[text=figNew!75!black] at (9.65,-1.35)
    {choose $d$ first, then $J$};
\end{tikzpicture}
\caption{Two roles of the source construction, schematically.
(a) Normalization gives $H_d^J\circ h_{d,J}=L^J$; the old polydisc
is changed only slightly, while the prescribed larger source is sent
into $E_{d,J}$. (b) At a fixed boundary value of $s$, the image of the
transverse polydisc has maximum-norm diameter $O(d^{-1})$.
}
\label{fig:source-transport}
\end{figure}
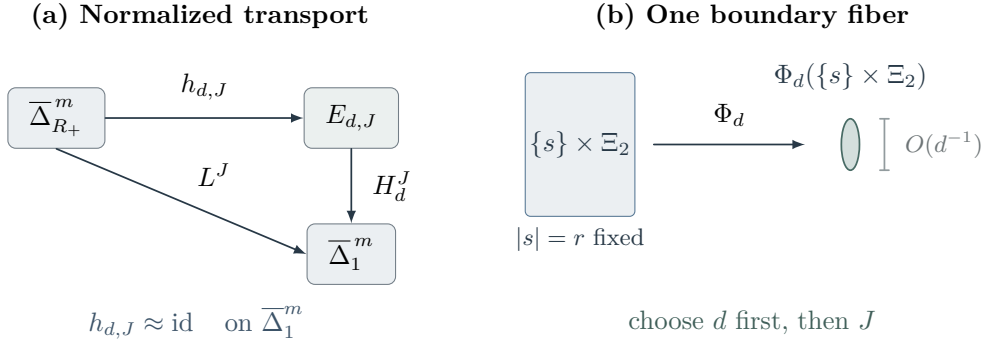

We use the fixed data and notation of Lemma~\ref{lem:source-geometry}.
The three estimates below prove its power bound, outer coordinates,
and uniform near-identity assertion.
First, the inverse is obtained without solving any polynomial equation:
if $y=H_d(x)$, set $x_m=y_1/a$ and then, successively,
\[
 x_{j-1}=(P_j(x_j)-y_j)/a\quad(j=m,m-1,\ldots,2).
\]
This proves that $H_d$ and its inverse are polynomial.

\begin{lemma}\label{lem:source-escape}
There are $C\ge1$ and $q\in(0,1)$, independent of $d$ and $J$,
such that, for all sufficiently large $d$, the power bound
\eqref{eq:source-power-bound} holds on $E_{d,J}$. If $x\in E_{d,J}$
and $|x_1|>r_1$, every $x_j$ is nonzero and
\[
 |H_d(x)_j|\le q a|x_{j-1}|\quad(2\le j\le m).
\]
Moreover, $|x_1|\le c_+$ implies $x\in\Delta_{b_0}^m$.
\end{lemma}

\begin{proof}
Choose $S>\max(2,r_1)$. Increasing the degree threshold ensures
$(M/r_1)^d\ge4M$ whenever $M\ge S$. The region
\[
 \mathcal E=\{x:\max_{j\ge2}|x_j|>\max(S,|x_1|)\}
\]
is forward invariant: if $M=|x_k|=\max_{j\ge2}|x_j|$, then
$|H_d(x)_k|\ge3M$, whereas $|H_d(x)_1|\le aM$.
A forward orbit ending in $\overline\Delta_1^m$ cannot enter
$\mathcal E$.
Take $C>\max(1,4a,16a^2r_1^2,4S)$.
If $A=\max_{j\ge2}|P_j(x_j)|>C\max(1,|x_1|)$,
then, since $d\ge2$ and $A>1$,
$|x_j|\le r_jA^{1/d}\le r_1\sqrt A$ for $j\ge2$.
For an index $k$ attaining $A$, if $k=2$, then
$a|x_1|<aA/C<A/4$. If $k>2$, then
$a|x_{k-1}|\le ar_1\sqrt A<A/4$, since
$A>C>16a^2r_1^2$. Thus
\[
 |H_d(x)_k|\ge3A/4>S,
 \qquad |H_d(x)_1|\le ar_1\sqrt A\le A/4.
\]
The next iterate belongs to $\mathcal E$, a contradiction.
This proves \eqref{eq:source-power-bound}; the same bound holds at
every subsequent iterate, including the endpoint.
If $|x_1|\le c_+$, it gives
$|x_j|\le r_j(Cc_+)^{1/d}<b_0$ for large $d$.
Together with $|x_1|\le c_+<b_0$, this proves the inner inclusion.

To obtain the domination estimate, write $y=H_d(x)$, where $x\in E_{d,J}$, $J\geq1$, and
$|x_1|\geq r_1$. The power bound at $y$ gives
\begin{equation}\label{eq:source-image-bound}
 |y_j|\leq r_j\{C\max(1,a|x_m|)\}^{1/d},\qquad 2\leq j\leq m.
\end{equation}
If $J=1$, this follows instead from $|y_j|\leq1<r_j$.
Set
\[
 e_j=d^{m-j},\qquad D_d=e_1=d^{m-1},\qquad
 t=(|x_1|/r_1)^{1/D_d}\geq1,
\]
and fix $R_*>r_1$ and $A_*=C\max(1,aR_*)$. We claim that
\begin{equation}\label{eq:source-weighted-bound}
 |x_j|\leq R_*t^{e_j}\qquad (2\leq j\leq m).
\end{equation}
Indeed, let
\[
 \Lambda=\max\left\{1,\max_{2\leq j\leq m}
       \left(\frac{|x_j|}{R_*t^{e_j}}\right)^{1/e_j}\right\}.
\]
Then $|x_j|\leq R_*(\Lambda t)^{e_j}$ for every $j$, including
$j=1$, since $|x_1|=r_1t^{D_d}$. If $\Lambda>1$, equality holds at
some $k\geq2$. Using $P_k(x_k)=y_k+ax_{k-1}$,
$e_{k-1}=de_k$, \eqref{eq:source-image-bound}, and the bound
$\max(1,a|x_m|)\leq\max(1,aR_*)\Lambda t$, we obtain
\begin{align*}
 (R_*/r_k)^d(\Lambda t)^{de_k}
 &\leq aR_*(\Lambda t)^{de_k}
       +r_k A_*^{1/d}(\Lambda t)^{1/d},\\
 (R_*/r_k)^d
 &\leq aR_*+r_k A_*^{1/d}(\Lambda t)^{1/d-de_k}
 \leq aR_*+r_k A_*^{1/d}.
\end{align*}
For all sufficiently large $d$ this is impossible, simultaneously for
$2\leq k\leq m$, because $R_*/r_k>1$. Hence $\Lambda=1$, proving
\eqref{eq:source-weighted-bound}. In particular, $|x_m|\leq R_*t$, so
\begin{equation}\label{eq:source-image-weighted}
 |y_j|\leq r_j A_*^{1/d}t^{1/d}\qquad (2\leq j\leq m).
\end{equation}

Choose $\delta>0$ so small that $r_j-\delta>0$ and
\[
 q=\max_{2\leq j\leq m}
       \frac{r_j+\delta}{a(r_{j-1}-\delta)}<1.
\]
For sufficiently large $d$, simultaneously for all $j\geq2$,
\[
 r_jA_*^{1/d}\leq r_j+\delta,\qquad
 r_j\{(1-q)a(r_{j-1}-\delta)\}^{1/d}\geq r_j-\delta.
\]
Starting from $|x_1|=r_1t^{D_d}$, induction gives
\begin{equation*}
 |x_j|\geq(r_j-\delta)t^{e_j}\quad(1\leq j\leq m),\qquad
 \frac{|y_j|}{a|x_{j-1}|}\leq q\quad(2\leq j\leq m).
\end{equation*}
For if the lower bound holds at $j-1$, then \eqref{eq:source-image-weighted}
implies
\[
 \frac{|y_j|}{a|x_{j-1}|}
 \leq q\,t^{1/d-e_{j-1}}\leq q.
\]
Consequently
\[
 |P_j(x_j)|\geq(1-q)a|x_{j-1}|
 \geq (1-q)a(r_{j-1}-\delta)t^{e_{j-1}},
\]
and taking the $d$th root proves the next lower bound. Thus every
$x_j$ is nonzero, and the required domination estimate follows.
All degree thresholds are independent of $J$.
\end{proof}

\begin{lemma}\label{lem:source-thin}
Let $B=-\log(1-q)$, with $q$ from Lemma~\ref{lem:source-escape}.
The coordinate representation and thin-fiber assertions in parts
\textup{(2)} and \textup{(3)} of Lemma~\ref{lem:source-geometry} hold.
\end{lemma}

The escape estimate makes each power term a controlled perturbation
of $a$ times the preceding coordinate. Taking logarithms records this
perturbation in bounded variables; their coefficients in the new
coordinates decrease with the degree, yielding thin transverse fibers.

\begin{proof}
Write $y=H_d(x)$. The ratios $P_j(x_j)/(a x_{j-1})=1+y_j/(a x_{j-1})$ belong to
$\{|w-1|\le q\}$. The branch of the logarithm centered at $1$
therefore defines $\xi_j$ holomorphically on a neighborhood of the
end. Indeed, with $u_j=y_j/(a x_{j-1})$, the denominators are nonzero
there and $|u_j|\le q<1$ on the end, while
\[
 |\xi_j|=|\log(1+u_j)|
 \le\sum_{\ell=1}^{\infty}\frac{|u_j|^\ell}{\ell}
 \le-\log(1-q)=B.
\]
Define
\[
 s=\frac{x_m}{\rho_m}
       \exp\left(-\sum_{k=2}^m\frac{\xi_k}{d^{m-k+1}}\right).
\]
Set $S_1=0$ and $S_j=\sum_{k=2}^j\xi_k/d^{j-k+1}$ for $j\ge2$.
The definition of $s$ gives $x_m=\rho_m s e^{S_m}$.
If $x_j=\rho_j s^{d^{m-j}}e^{S_j}$, then
\[
 x_{j-1}=\frac{x_j^d}{ar_j^d e^{\xi_j}}
 =\rho_{j-1}s^{d^{m-j+1}}e^{S_{j-1}},
\]
since $\rho_j^d=ar_j^d\rho_{j-1}$ and $dS_j-\xi_j=S_{j-1}$.
Backward induction proves \eqref{eq:source-monomial}, including
$s^{D_d}=x_1$, without any choice of a root.

For the thin-fiber estimate, the recurrence defining $\rho_j$ shows that $|\log\rho_j|$ is
bounded independently of $d$, and
$\log\rho_j-\log r_j=O(d^{-1})$ for $j\ge2$.
There is a constant $K>0$, independent of $d$, such that for
$|s|^{D_d}=c$, $\xi\in\Xi_2$, and $2\leq j\leq m$, the logarithm of
$|\Phi_{d,j}/r_j|$ is bounded above by $K/d$:
the power of $s$ contributes $\log c/d^{j-1}$ and the transverse
sum has modulus at most $(B+2)/(d-1)$. Thus
\[
 |\Phi_{d,j}(s,\xi)|\le r_j e^{K/d}\quad(j\ge2).
\]
The corresponding transverse sums at two points of $\Xi_2$ differ
by at most $2(B+2)/(d-1)$. Using
$|e^u-e^v|\le |u-v|e^{\max(|u|,|v|)}$ proves the asserted
$O(d^{-1})$ diameter bound. The first component is independent of
$\xi$ and has modulus $c<b_0$. Every component contains a positive
integral power of $s$, so the boundary bounds hold on the full
closed scalar disc as well. Since $r_j<b_0$, they imply the stated
compact inclusion for large $d$.
\end{proof}

\begin{lemma}\label{lem:source-near-identity}
For each fixed $1<u_0<u_1<r_m$, the estimate
\eqref{eq:source-uniform-near-identity} holds for all sufficiently
large $d$, uniformly over all $J\ge1$.
\end{lemma}

\begin{proof}
Write $H_d=L+N_d$. In the maximum norm,
$\|L^k x\|_\infty=a^k\|x\|_\infty$ for every integer $k$ and
\[
 \|N_d(x)\|_\infty\le(\|x\|_\infty/r_m)^d.
\]
On $\overline\Delta_{u_1}^m$ this gives
$\|H_d(x)\|_\infty\le\alpha_d\|x\|_\infty$.
Fix $a<\eta<1$. Since $\alpha_d\to a$, for large $d$ we have
$\alpha_d\le\eta<1$ and $\alpha_d^d\le\eta^d\le a/2$.
All forward iterates stay in this polydisc and satisfy
$\|H_d^k(x)\|_\infty\le\alpha_d^k u_1$.
For $g_{d,J}=L^{-J}\circ H_d^J$, telescoping gives
\[
 g_{d,J}(x)-x=\sum_{k=0}^{J-1}L^{-k-1}N_d(H_d^k(x)).
\]
Each summand has norm at most
$a^{-k-1}(\alpha_d^k u_1/r_m)^d
=(\theta^d/a)(\alpha_d^d/a)^k$. Hence
\[
 \|g_{d,J}-\operatorname{id}\|_{\overline\Delta_{u_1}^m}
 \le\frac{\theta^d}{a}\sum_{k=0}^{\infty}
       \left(\frac{\alpha_d^d}{a}\right)^k
 =\varepsilon_d:=\frac{\theta^d}{a-\alpha_d^d},
\]
uniformly in $J$.
The global inverse of $g_{d,J}$ is $h_{d,J}$, and both fix zero.
Increase $d$ so that $\varepsilon_d<u_1-u_0$.
For $z\in\overline\Delta_{u_0}^m$, the path $h_{d,J}(tz)$,
$0\le t\le1$, cannot first reach the boundary of
$\Delta_{u_1}^m$: at that point the preceding estimate would give
\[
 \|tz\|_\infty
 =\|g_{d,J}(h_{d,J}(tz))\|_\infty
 \ge u_1-\varepsilon_d>u_0.
\]
Applying the estimate for $g_{d,J}$ at $h_{d,J}(z)$ proves
\eqref{eq:source-uniform-near-identity}. All degree thresholds are
independent of $J$.
\end{proof}

\begin{proof}[Proof of Lemma~\ref{lem:source-geometry}]
The inverse recurrence proves polynomial invertibility.
Lemmas~\ref{lem:source-escape}--\ref{lem:source-near-identity} give
all four estimates, with one common degree threshold independent
of $J$. Choose $d$ beyond this threshold and with the required
near-identity accuracy and fiber diameter. For the prescribed
$R_+$, choose $J$ so that $a^JR_+<1$. Then $L^J$ sends
$\overline\Delta_{R_+}^m$ into $\Delta_1^m$, proving the
outer-domain inclusion. On $\overline\Delta_1^m$, the near-identity estimate gives
\[
 |h_1(z)|\leq |z_1|+\|h_{d,J}-\operatorname{id}\|
 \leq 1+(2/a)\theta^d<c
\]
for sufficiently large $d$, uniformly in $J$, since $\theta<1<c$. Strict inequalities on the compact sets in question
persist on small neighborhoods, completing the proof.
\end{proof}

\section{Extension across polydiscs}\label{sec:polydisc}

We now convert the local sector extension of
Lemma~\ref{family:local-sector} into approximation on polydiscs.
The thin boundary fibers from Section~\ref{sec:source-geometry}
provide the target-neighborhood containment needed for this step.

Throughout this section, the neighborhoods $W$ are those supplied by
Lemma~\ref{family:local-sector}: a map near $J\times P$ with values
in $W$ admits approximation by maps near $L\times P$, for every
special planar pair $J\subset L$ and prescribed closed passive
polydisc $P$. The same $P$ occurs in both products.

\subsection{The scalar collapse}

One entire function simultaneously preserves the old disc and brings
all the attaching corridors back to its boundary.

\begin{lemma}\label{poly:collapse}
For $r>0$ and an integer $N\geq1$, set
\[
 p_N(s)=\int_0^s\exp\bigl(-(t/r)^N\bigr)\,dt,
 \qquad \zeta_j=re^{2\pi i j/N}\quad(0\leq j\leq N).
\]
Then
\begin{equation*}
 \sup_{|s|\leq r}|p_N(s)-s|\leq\frac{r(e-1)}{N+1},
 \qquad
 \sup_{t\geq r}|p_N(te^{2\pi i j/N})-\zeta_j|
 \leq\frac r{N+1}.
\end{equation*}
\end{lemma}
\begin{proof}
Termwise integration of the exponential series gives the first estimate.
On the positive real ray,
\[
 0\leq r-p_N(r)\leq\frac r{N+1},\qquad
 \int_r^\infty e^{-(t/r)^N}\,dt
 \leq\frac r{eN}\leq\frac r{N+1}.
\]
For the integral estimate substitute $u=(t/r)^N$ and use
$u^{1/N-1}\leq1$ for $u\geq1$. For $t\geq r$, write
\[
 p_N(t)-r=\int_r^t e^{-(u/r)^N}\,du-(r-p_N(r)).
\]
Both terms on the right lie in $[0,r/(N+1)]$, so their difference
has absolute value at most $r/(N+1)$. The identity
$p_N(\omega s)=\omega p_N(s)$ for $\omega^N=1$, which follows
from the power series, gives the remaining rays.
\end{proof}

\subsection{Disc enlargement with fixed passive parameters}

The planar model is OpenAI's spokes-and-sectors enlargement~\cite[Proposition~2.9]{OAI26}: boundary values are kept in fixed target
neighborhoods while the sectors between radial spokes are filled.
Here the scalar collapse constructs all corridors at once, and the
sectors are filled simultaneously.

\begin{lemma}\label{poly:boundary-disc}
Let $r>0$, $P$ be a closed polydisc, and $f$ holomorphic
near $\overline\Delta_r\times P$. Suppose that, for every $|\zeta|=r$,
the compact fiber $f(\{\zeta\}\times P)$ is contained in one of
the neighborhoods $W$. For every $S>r$, $f$ is uniformly approximable
on $\overline\Delta_r\times P$ by maps holomorphic near
$\overline\Delta_S\times P$.
\end{lemma}
\begin{proof}
By Lemma~\ref{lem:chart-spray}, choose a chart spray $f^\sharp$
centered at $f$. Compactness permits a slightly larger passive
polydisc $P^+$ and a fixed spray radius $\beta>0$ such that finitely
many boundary-arc products, including $P^+$ and
$\overline\Delta_\beta^n$, map into their respective neighborhoods $W$ with positive margins. Fix
$0<\rho<\rho_+<\beta$.

Choose a sufficiently fine uniform subdivision of the circle and then
$p_N$ from Lemma~\ref{poly:collapse}. The family
$f^\sharp(p_N(s),p,v)$ is defined on the old disc and on thin radial
corridors extending beyond $S$, as in
Figure~\ref{fig:scalar-corridors}(a). Its values near each boundary arc
and the two adjacent corridors lie in the corresponding neighborhood
$W$, uniformly for $(p,v)\in P^+\times\overline\Delta_\beta^n$.
This follows from the near-identity and collapsed-ray estimates,
respectively.

Take the old side to be the inner disc together with the corridors,
and the new side to be the disjoint intervening sectors, covering a
neighborhood of $\overline\Delta_S$. Small boundary buffers within
the controlled neighborhoods give separated closed fringes. For each sector,
the overlap lies in the U-shaped collar $J_j$ shown in
Figure~\ref{fig:scalar-corridors}(b). Its bottom follows the boundary
arc and its sides follow the two corridors, so the whole parameter
family on $J_j$ takes values in one $W$. Filling the U gives a sector
$L_j$ whose outer edge lies beyond the new side; the added disc meets $J_j$
along a proper boundary arc, hence $J_j\subset L_j$ is a special
planar pair. Compactness allows all collars and buffers to be chosen
uniformly in $(p,v)$. Fix them before choosing the sector-approximation
accuracy.

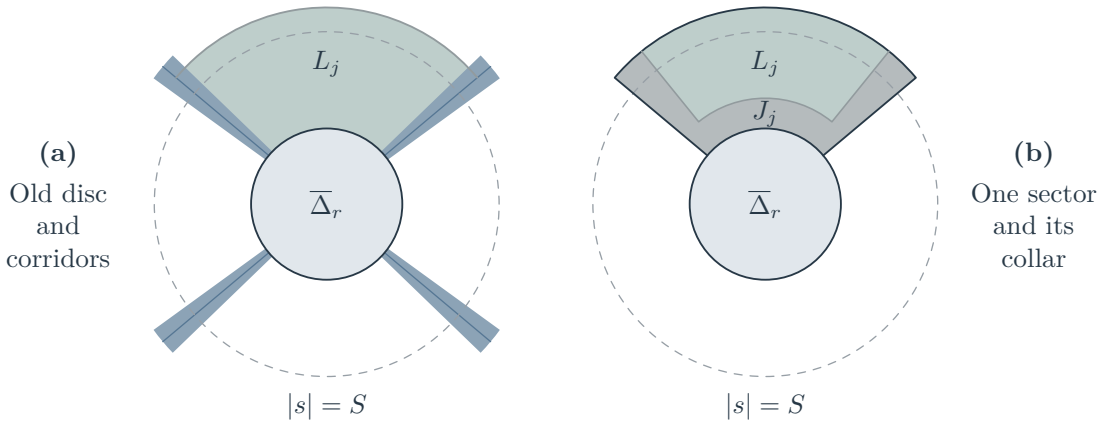
\begin{figure}[ht]
\centering
\begingroup
\begin{tikzpicture}[x=1cm,y=1cm,font=\small,
 line cap=round,line join=round,
 boundary/.style={draw=figInk,line width=.7pt},
 leader/.style={->,>=latex,draw=figInk!80,line width=.55pt}]
\begin{scope}
 \node[align=center,text width=1.7cm,text=figInk] at (-3.55,0)
  {\textbf{(a)}\\[3pt]Old disc\\and corridors};
 \path[fill=figNew!38]
  (40:1) -- (40:2.60)
  arc[start angle=40,end angle=140,radius=2.60] -- (140:1)
  arc[start angle=140,end angle=40,radius=1] -- cycle;
 \foreach \a in {40,140,220,320} {
  \path[fill=figOld!68]
   (\a-4:.92) -- (\a-4:2.83)
   arc[start angle=\a-4,end angle=\a+4,radius=2.83]
   -- (\a+4:.92) -- cycle;
  \draw[figOld,line width=.5pt] (\a:1) -- (\a:2.83);
 }
 \fill[figOld!18] (0,0) circle[radius=1];
 \draw[boundary] (0,0) circle[radius=1];
 \draw[dashed,figInk!48,line width=.5pt] (0,0) circle[radius=2.28];
 \draw[figOverlap,line width=.7pt]
  (40:2.60) arc[start angle=40,end angle=140,radius=2.60];
 \node[text=figInk] at (0,0) {$\overline\Delta_r$};
 \node[text=figInk] at (0,1.86) {$L_j$};
 \node[anchor=north,text=figInk] at (0,-2.36) {$|s|=S$};
\end{scope}
\begin{scope}[xshift=5.80cm]
 \node[align=center,text width=1.7cm,text=figInk] at (3.55,0)
  {\textbf{(b)}\\[3pt]One sector\\and its collar};
 \path[fill=figNew!38]
  (40:1) -- (40:2.60)
  arc[start angle=40,end angle=140,radius=2.60] -- (140:1)
  arc[start angle=140,end angle=40,radius=1] -- cycle;
 \path[fill=figOverlap!68]
  (40:1) -- (40:2.60)
  arc[start angle=40,end angle=51,radius=2.60] -- (51:1.40)
  arc[start angle=51,end angle=129,radius=1.40] -- (129:2.60)
  arc[start angle=129,end angle=140,radius=2.60] -- (140:1)
  arc[start angle=140,end angle=40,radius=1] -- cycle;
 \fill[figOld!18] (0,0) circle[radius=1];
 \draw[boundary] (0,0) circle[radius=1];
 \draw[dashed,figInk!48,line width=.5pt] (0,0) circle[radius=2.28];
 \draw[boundary] (40:1) -- (40:2.60)
  arc[start angle=40,end angle=140,radius=2.60] -- (140:1);
 \draw[figOverlap,line width=.6pt]
  (51:2.60) -- (51:1.40)
  arc[start angle=51,end angle=129,radius=1.40] -- (129:2.60);
 \node[text=figInk] at (0,0) {$\overline\Delta_r$};
 \node[text=figInk] at (0,1.86) {$L_j$};
 \node[text=figInk] at (0,1.20) {$J_j$};
 \node[anchor=north,text=figInk] at (0,-2.36) {$|s|=S$};
\end{scope}
\end{tikzpicture}
\endgroup
\caption{The planar construction in Lemma~\ref{poly:boundary-disc}. (a) The old disc and radial
corridors form the old side; the sectors are filled simultaneously.
The sector $L_j$ is shown in both panels. (b) Its collar $J_j$ consists of bottom and side
strips. The middle of the outer edge is not part of the collar.
Widths are enlarged
for visibility; small boundary buffers and the parameter factors are omitted.}
\label{fig:scalar-corridors}
\end{figure}

Denote the scalar old side by $A$ and the new side by
$B=\bigsqcup_j B_j$, with $B_j\subset L_j$ and $A\cap B_j\subset J_j$.
On $A$ use the chart spray
$F_A(s,p,v)=f^\sharp(p_N(s),p,v)$.
Local sector extension, with $(p,v)$ passive, supplies maps $G_j$ near
$L_j\times P^+\times\overline\Delta_\beta^n$ arbitrarily close
to $F_A$ on the respective collars. Define $F_B=G_j$ on each
$B_j$; the disjointness of the branches makes this a holomorphic family
on the whole new side. Since there are finitely many sectors, their
approximation errors can all satisfy one prescribed gluing tolerance.
Apply the parameter version of Lemma~\ref{lem:center-gluing} once to
$F_A,F_B$ on $A,B$, retaining $p$ in a fixed neighborhood of $P$. The splitting operators act in
the scalar variable and preserve holomorphic dependence on $p$.
The fixed gap $\rho_+-\rho$ supplies the required derivative
smallness by Cauchy estimates. This gives a map holomorphic near
$\overline\Delta_S\times P$, as close as desired to
$f(p_N(s),p)$ on the old product. First making $N$ large and then
making the sector approximations accurate proves the assertion.
\end{proof}

\subsection{Finite extension}

\begin{proposition}\label{poly:finite-extension}
Suppose $Y$ satisfies \LAP. Let $m\geq1$ and $f$ be holomorphic
near $Q=\overline\Delta^m$. For every $R>1$ and $\epsilon>0$, there
is a map $\widetilde f$ holomorphic near $\overline\Delta_R^m$ such
that
\[
 \sup_Q d_Y(\widetilde f,f)<\epsilon.
\]
\end{proposition}
\begin{proof}
When $m=1$, choose a scalar radius $r>1$ inside the domain of $f$.
Each boundary value belongs to a neighborhood $W$, so
Lemma~\ref{poly:boundary-disc}, with no passive variables, gives
the required finite extension.

Suppose $m\geq2$. Choose $1<b_0<b'<b$ such that $f$ is holomorphic
near $\overline\Delta_b^m$, and a chart spray $f^\sharp(x,v)$
centered at $f$ near $\overline\Delta_{b'}^m\times\{0\}$.
For $R_+>R$, take a smoothly bounded strongly convex domain $D_0$
with
\[
 \overline\Delta_R^m\Subset D_0\Subset\Delta_{R_+}^m.
\]
Choose the constants $r_j,a,c,c_+$ from
Lemma~\ref{lem:source-geometry}, with $c_+<b_0$, and choose
\[
 r_1<c_0<c_1<c.
\]

Compactness gives $\lambda>0$ such that, for every
$y\in f(\overline\Delta_{b_0}^m)$, the ambient $d_Y$-ball
$B_{d_Y}(y,3\lambda)$ lies in one neighborhood $W$.
Choose a common spray radius $\beta>0$ so small that
\[
 d_Y(f^\sharp(x,v),f(x))<\lambda
 \quad\bigl(x\in\overline\Delta_{b_0}^m,
             \ v\in\overline\Delta_\beta^n\bigr).
\]
Uniform continuity of $f$ and the thin-fiber estimate in
Lemma~\ref{lem:source-geometry} allow us to choose $d$ so large that
\[
 d_Y\bigl(f(\Phi_d(s,\xi)),f(\Phi_d(s,0))\bigr)<\lambda
 \quad\bigl(|s|=c^{1/D_d},\ \xi\in\Xi_2\bigr).
\]
The same choice of $d$ can ensure, uniformly for $J\geq1$,
\[
 |(h_{d,J})_1|<c_0\quad\hbox{on }Q,\qquad
 \sup_Qd_Y(f\circ h_{d,J},f)<\epsilon/2.
\]
Finally choose $J$ so that $a^JR_+<1$. Write $h=h_{d,J}$,
$\Phi=\Phi_d$, $D_d=d^{m-1}$, and $r=c^{1/D_d}$.
For every $|s|=r$, the whole compact fiber
\[
 f^\sharp(\Phi(s,\xi),v),\qquad
 \xi\in\Xi_2,\quad v\in\overline\Delta_\beta^n,
\]
lies in $B_{d_Y}(f(\Phi(s,0)),2\lambda)$ and hence in one
neighborhood $W$.
These choices are independent of the approximation accuracy below.
If $|h_1|\leq c_0$ throughout $D_0$, the inner inclusion in
Lemma~\ref{lem:source-geometry} makes $f\circ h$ holomorphic on
$D_0$, and the assertion follows. Otherwise the outer side below
is nonempty.

Here $z$ denotes the original source point and $x=h(z)$ its transported
image. The coordinates $(s,\xi)$ describe only the outer part of that
image; $v$ is the spray variable. Scalar enlargement will extend the
family in $(s,\xi,v)$, after which it will be pulled back to the outer
$z$-domain and glued to the old family.

The map
\[
 k(s,\xi,v)=f^\sharp(\Phi(s,\xi),v)
\]
is holomorphic near
$\overline\Delta_r\times\Xi_2\times\overline\Delta_\beta^n$.
The outer coordinates $s,\xi$
are single-valued and holomorphic near
$h(\overline D_0)\cap\{|x_1|\geq c_0\}$, and their actual
transverse values lie strictly inside $\Xi_1$.
Choose $S>r$ larger, with a positive margin, than all the scalar
radii on this compact set. With $(\xi,v)$ as passive variables,
Lemma~\ref{poly:boundary-disc} gives maps $G(s,\xi,v)$
holomorphic near
\[
 \overline\Delta_S\times\Xi_1\times\overline\Delta_\beta^n
\]
and arbitrarily close to $k$ on
$\overline\Delta_r\times\Xi_1\times\overline\Delta_\beta^n$.

Consider the fixed source pair
\[
 U_A=D_0\cap\{|h_1|<c_1\},\qquad
 U_B=D_0\cap\{|h_1|>c_0\}.
\]
It satisfies $U_A\cup U_B=D_0$, has separated closed fringes, and
contains $Q$ in $U_A\setminus\overline{U_B}$. On the two sides use
\[
 F_A(z,v)=f^\sharp(h(z),v),\qquad
 F_B(z,v)=G(s(h(z)),\xi(h(z)),v).
\]
The first family is a chart spray near $\overline{U_A}$ by the
inner inclusion in Lemma~\ref{lem:source-geometry}; the second
is holomorphic near $\overline{U_B}$. On the overlap,
\[
 h(z)=\Phi(s(h(z)),\xi(h(z))),\qquad
 |s(h(z))|\leq c_1^{1/D_d}<r.
\]
Thus the pullback of $k$ is exactly $F_A$ on the overlap, while
that of $G$ is $F_B$. The approximation $G\approx k$ therefore makes
$F_B$ arbitrarily close to $F_A$ on this fixed overlap and the full
fiber product $\overline\Delta_\beta^n$.

Fix $0<\rho<\rho_+<\beta$. The source pair and its splitting
constant are already fixed. The inverse charts of $F_A$ and the
Cauchy estimate across the fixed fiber gap make the transition and
its fiber derivative as close to the identity as required by
Lemma~\ref{lem:center-gluing}. That lemma gives a holomorphic map
$\widetilde f:D_0\to Y$ with
\[
 \sup_Qd_Y(\widetilde f,f\circ h)<\epsilon/2
\]
when the scalar-enlargement approximation is sufficiently accurate.
The two estimates give the desired approximation to $f$ on $Q$.
\end{proof}

\subsection{Exhaustion}

\begin{theorem}\label{thm:polydisc}
Suppose $Y$ satisfies \LAP. A map holomorphic near a closed polydisc in
$\mathbb C^m$ can be approximated uniformly there by entire maps
$\mathbb C^m\to Y$.
\end{theorem}
\begin{proof}
Let $f$ be the given map. Normalize its original polydisc to
$Q=\overline\Delta^m$, fix $\epsilon>0$, and set $f_0=f$. Choose
$1=R_0<R_1<\cdots\to\infty$ and apply
Proposition~\ref{poly:finite-extension} successively after rescaling.
We obtain maps $f_j$
holomorphic near
$\overline\Delta_{R_j}^m$ such that
\[
 \sup_{\overline\Delta_{R_{j-1}}^m}d_Y(f_j,f_{j-1})
 <\epsilon 2^{-j-1}.
\]
Completeness of $d_Y$ gives a locally uniform limit
$F:\mathbb C^m\to Y$. Target coordinate charts show that $F$
is holomorphic, and the summable errors give the required
approximation on $Q$.
\end{proof}

\begin{remark}\label{poly:direct-domination}
For a strongly dominable target, the sector-filling step in
Lemma~\ref{poly:boundary-disc} has a direct form. Near each target
point choose an entire map $\Psi:\mathbb C^n\to Y$ with an
inverse branch $\lambda$ on $W$. Let $k_j$ be the old family on a
collar product $J_j\times P$, with image in $W$ and with all
passive variables included in $P$. Its lift $\lambda\circ k_j$
is holomorphic near the polynomially convex product $J_j\times P$.
Oka--Weil approximation by a polynomial $q_j$, followed by
$\Psi\circ q_j$, fills the sector. The finite extension then
follows from scalar enlargement and the fixed annular gluing above.
\end{remark}

\section{From polydisc approximation to convex approximation}
\label{sec:convex}

The passage from polydiscs
to arbitrary convex sets has two parts.
First, approximation on lower-dimensional convex sets gives uniform
approximation on compact convex subsets of a real hyperplane. Second, this real-face approximation
allows us to remove one half-space constraint at a time. The second part
requires a complex collar that remains fixed as the gluing accuracy
improves; Subsection~\ref{sec:fixed-collar} explains how to obtain it.

The real-face induction follows OpenAI~\cite[Section~4]{OAI26}.
There the lack of a prescribed complex collar is handled by gluing
estimates on the real interface~\cite[Lemmas~4.4--4.5]{OAI26}; here the
fixed-inverse approximation described in Subsection~\ref{sec:fixed-collar} supplies arbitrarily accurate
transitions on one fixed collar. We use the approximation property with
passive parameters from that argument, stated here in arbitrary target
dimension.

In the following definition, $r$ counts the convex source variables.
The additional polydisc variables are retained throughout the induction:
they include the auxiliary holomorphic parameters used to join local
approximants and the spray variables needed for gluing.

\begin{definition}[Convex approximation with passive parameters]
\label{def:approximation-level}
Let $r\geq0$ be an integer. A complex manifold $Y$ satisfies
$\mathsf A_r$ if, for every integer $q\geq0$, every compact convex
set $K\subset\C^r$, every closed polydisc $P\subset\C^q$, every
open neighborhood $U\subset\C^{r+q}$ of $K\times P$, every
holomorphic map $f:U\to Y$, and every $\epsilon>0$, there is a
holomorphic map $E:\C^{r+q}\to Y$ such that
\[
 \sup_{K\times P}d_Y(E,f)<\epsilon.
\]
A zero-dimensional factor is a point. In particular, $\mathsf A_0$
is approximation on closed polydiscs in every dimension.
\end{definition}

\begin{theorem}\label{thm:a0-cap}
If $Y$ satisfies $\mathsf A_0$, then it satisfies $\mathsf A_r$ for
every $r\geq0$. In particular, $Y$ has the convex approximation
property.
\end{theorem}

The proof is completed in Subsection~\ref{sec:finite-cuts}, using the lemmas developed below.

We prove the implication $\mathsf A_{r-1}\Rightarrow\mathsf A_r$.
All approximation assertions below concern uniform convergence; only
fiberwise Cauchy estimates will be needed in the gluing step.

\subsection{Complexifying a real parameter}

We first replace a smooth real parameter by a holomorphic one. The
complex neighborhood obtained in this step may depend on the desired
accuracy, which is sufficient for the argument.

\begin{lemma}\label{convex:complexify}
Let $I\subset\R$ be a compact interval and $C\subset\C^k$
compact and convex. Suppose that $H$ is a smooth map to $Y$ on a
neighborhood of $I\times C$ in $\R\times\C^k$, holomorphic in the
second variable for each real value of the first. For every
$\epsilon>0$ there is a map holomorphic on a complex neighborhood of
$I\times C$ whose distance from $H$ on $I\times C$ is less than
$\epsilon$.
\end{lemma}
\begin{proof}
Choose a larger interval $I^+=[a,b]$, a compact convex neighborhood
$C^+$ of $C$, with all products under consideration in the domain of
$H$, and a bounded convex open neighborhood $X$ of $C^+$
with $I^+\times X$ in that domain. The Stein-neighborhood theorem
we need is the following specialization of
\cite[Corollary~2.2]{FW10}: a continuous family of holomorphic
sections over a Stein base, parametrized by a compact real interval,
has a common Stein neighborhood of its graphs over any fixed
holomorphically convex compact subset of the base. Here $X$ is
Stein, $C^+$ is $\OO(X)$-convex because it is convex, and
$q\mapsto(q,H(s,q))$ is such a family of sections of
$X\times Y\to X$. The corollary gives a Stein neighborhood
$\Theta\subset\C\times X\times Y$ of
\[
 \{(s,q,H(s,q)):s\in I^+,\ q\in C^+\}.
\]
The classical Stein embedding theorem~\cite{F17} gives a closed
holomorphic embedding $e:\Theta\hookrightarrow\C^N$.
The holomorphic neighborhood-retraction theorem, in the form of
\cite[Corollary~1]{Siu76}, gives a holomorphic retraction $\rho:\Omega\to e(\Theta)$
from an open neighborhood in $\C^N$.

Choose open neighborhoods $U_0,U_1$ of $C$ with
$C\Subset U_0\Subset U_1\Subset\Int C^+$.
The function $g(s,q)=e(s,q,H(s,q))$ is continuous as a map from
$I^+$ into the Banach space $H^\infty(U_1,\C^N)$. Set $x=(s-a)/(b-a)$. The Bernstein
polynomials of $g$
\[
 P_\nu(s,q)=\sum_{j=0}^{\nu}
 g\left(a+\frac{j(b-a)}\nu,q\right)
 \binom\nu j x^j(1-x)^{\nu-j},
\]
are holomorphic in $(s,q)$ and converge uniformly to $g$ on
$I^+\times U_1$. Indeed, uniform continuity controls the terms with
$|j/\nu-x|\leq\nu^{-1/4}$, while the variance bound
$x(1-x)/\nu\leq1/(4\nu)$ makes the total weight of the remaining
terms at most $1/(4\sqrt\nu)$.

For large $\nu$, the compact set
$P_\nu(I\times\overline U_0)$ lies in $\Omega$. Hence it remains in
$\Omega$ when $I$ is replaced by a sufficiently small complex
neighborhood $V_\nu$. Let $\pi_Y:\C\times X\times Y\to Y$
be the projection. The maps
\[
 \pi_Y\circ e^{-1}\circ\rho\circ P_\nu:
 V_\nu\times U_0\longrightarrow Y
\]
approximate $H$ uniformly on $I\times C$, by continuity on a compact
neighborhood of the embedded graph.
\end{proof}

An interval contributes no new convex dimension: a planar coordinate
change turns it into a passive disc, and scalar polynomial approximation
removes that coordinate change afterward.

\begin{lemma}\label{convex:interval}
Assume $\mathsf A_s$ for some $s\geq0$. If $I\subset\R$ is a compact
interval, $C\subset\C^s$ is compact convex, and $P$ is a closed
polydisc, every map holomorphic near $I\times C\times P$ can be
approximated uniformly there by maps entire in all displayed complex
variables.
\end{lemma}
\begin{proof}
For the given map $H$, choose a bounded simply connected complex
neighborhood $V$ of $I$ such that $H$ is holomorphic near
$V\times C\times P$. A Riemann map $\varphi:V\to\Delta$ satisfies
$\varphi(I)\Subset\Delta_\rho$ for some $\rho<1$. Apply
$\mathsf A_s$ to
$H(\varphi^{-1}(\zeta),y,u)$ on
$C\times(\overline\Delta_\rho\times P)$, regarding $(\zeta,u)$ as
passive variables, and obtain an entire map $E(\zeta,y,u)$.
Once $E$ is fixed, approximate $\varphi$ uniformly on $I$ by a scalar
polynomial $p$. The entire map $E(p(t),y,u)$ gives the assertion.
Uniform continuity of the fixed $E$ on a compact neighborhood of
$\varphi(I)\times C\times P$ controls the substitution error.
\end{proof}

\subsection{Joining local approximants on a common domain}

We isolate the path construction in the proof of
OpenAI~\cite[Lemma~4.3]{OAI26}. The local approximants along a real
face will be entire, but the short path joining two nearby approximants
is constructed in a chart of the original map. That path is initially defined only on a small joining
product. To complexify the assembled family, we need all its paths on
one common source polydisc. The next lemma obtains this by expanding
the source near the entire endpoints, approximating the resulting
family, and scaling back. The scaling gives endpoint control on the
larger polydisc, where a small chart correction restores exact agreement.

\begin{lemma}\label{convex:path}
Assume $\mathsf A_s$, $s\geq0$. Let $C^-,C\subset\C^s$ be compact
convex sets with $C^-\Subset\Int C$, $B\subset\C^s$ a closed
polydisc containing $C$ in its interior, and $P\subset\C^q$ a
closed polydisc. Suppose that $h_\tau(y,u)$, $0\leq\tau\leq1$, is a
smooth family holomorphic in $(y,u)$ near $C\times P$, and that it
agrees, on intervals near $\tau=0,1$, with entire maps $g_0,g_1$,
respectively. For every $\epsilon>0$ there is a smooth family
$\widehat h_\tau$, holomorphic in $(y,u)$ on a common neighborhood
of $B\times P$, which equals $g_0,g_1$ on endpoint intervals and
satisfies
\[
 \sup_{[0,1]\times C^-\times P}d_Y(\widehat h_\tau,h_\tau)
 <\epsilon.
\]
For $s=0$, the convex and polydisc factors in $\C^s$ are omitted.
\end{lemma}
\begin{proof}
Suppose first that $s>0$. Choose $d\in\Int C$ and $L>1$ so large
that $d+L^{-1}(B-d)\Subset\Int C$. Take smooth functions
$\kappa_0,\kappa_1:[0,1]\to[0,1]$ with disjoint supports inside the
intervals where $h_\tau=g_0,g_1$, respectively, and equal to one near
the corresponding endpoints. Set
\[
 \lambda(\tau)=1+(L-1)(\kappa_0(\tau)+\kappa_1(\tau)),\qquad
 \widetilde h_\tau(y,u)=h_\tau(d+\lambda(\tau)(y-d),u).
\]
This family is defined near $[0,1]\times C\times P$: wherever
$\lambda>1$, the map being evaluated is an entire endpoint map.
Fix chart sprays of $g_0,g_1$ over a neighborhood of $B\times P$
using Lemma~\ref{lem:chart-spray}.

Apply Lemma~\ref{convex:complexify} to $\widetilde h$ and then
Lemma~\ref{convex:interval} under $\mathsf A_s$. This gives an
entire $E(\tau,y,u)$ arbitrarily close to $\widetilde h$ on
$[0,1]\times C\times P$. Set
\[
 A_\tau(y,u)=E(\tau,d+\lambda(\tau)^{-1}(y-d),u).
\]
For $y\in C^-$ the rescaled argument lies in
$\operatorname{conv}(C^-\cup\{d\})\Subset\Int C$. Substitution
into $\widetilde h_\tau$ gives exactly $h_\tau(y,u)$.
Consequently $A_\tau$ approximates $h_\tau$ uniformly on the required
smaller product. On endpoint intervals where $\lambda=L$, the same
identity holds for every $y\in B$, and $h_\tau=g_\ell$ there.
Thus $A_\tau$ is arbitrarily close to $g_\ell$ on the full product
$B\times P$ on those intervals.

The fixed endpoint charts therefore give small fiber coordinates
$\beta^{(\ell)}_\tau$ such that
\[
 A_\tau=(g_\ell)^\sharp(\,\cdot\,,\beta^{(\ell)}_\tau).
\]
On a smaller closed endpoint interval, choose a smooth cutoff
$\vartheta_\ell$ equal to zero near $\ell$ and to one near its inner
end. Replace $A_\tau$ there by
$(g_\ell)^\sharp(\,\cdot\,,\vartheta_\ell\beta^{(\ell)}_\tau)$.
The replacement agrees with $A_\tau$ near the inner end and with
$g_\ell$ near the outer end. It is smooth in $\tau$, holomorphic on
a common neighborhood of $B\times P$, and remains arbitrarily close
to $h_\tau$ on $C^-\times P$. These assertions follow from uniform
continuity of the fixed charts and their inverses on compact sets.
For $s=0$, take $\lambda=1$, omit scaling, and use the same entire
approximation and endpoint repair.
\end{proof}

\begin{lemma}\label{convex:face}
Assume $\mathsf A_{r-1}$, $r\geq1$. Let
$F\subset\R\times\C^{r-1}$ be compact convex and $P$ a closed
polydisc. Every map holomorphic near $F\times P$ can be approximated
uniformly there by maps entire in all displayed complex variables.
\end{lemma}

The slices of $F$ vary with the real coordinate $t$, whereas
Lemma~\ref{convex:complexify} requires a family on a fixed product.
We join local entire approximants using Lemma~\ref{convex:path}
on a common polydisc $B$, then complexify the assembled family and
approximate it by an entire map.
Figure~\ref{fig:real-face-joining} shows one join.

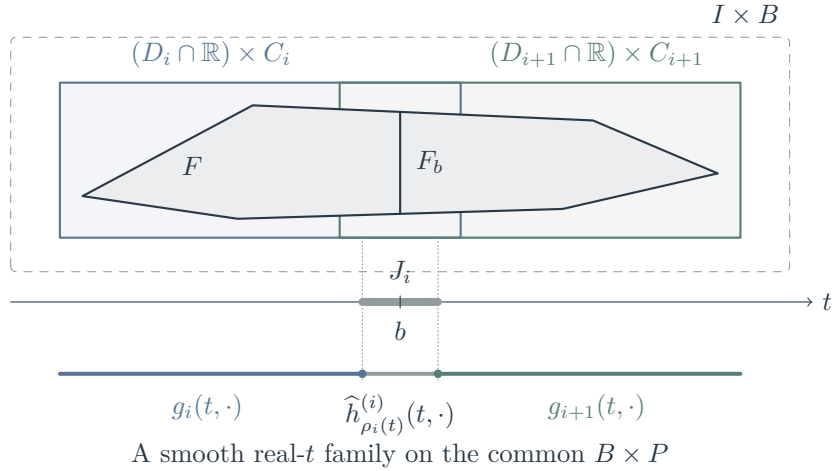
\begin{figure}[htbp]
\centering
\begin{tikzpicture}[x=1cm,y=1cm,font=\small,text=figInk,
  line cap=round,line join=round]
\draw[figOverlap,dashed,rounded corners=2pt]
  (0,0.35) rectangle (10.3,3.45);
\node[anchor=south east] at (10.3,3.48) {$I\times B$};
\fill[figOld!7] (0.65,0.8) rectangle (5.95,2.85);
\fill[figNew!7] (4.35,0.8) rectangle (9.65,2.85);
\draw[figOld,thick] (0.65,0.8) rectangle (5.95,2.85);
\draw[figNew,thick] (4.35,0.8) rectangle (9.65,2.85);
\node[figOld,anchor=south] at (2.65,2.91)
  {$(D_i\cap\R)\times C_i$};
\node[figNew,anchor=south] at (7.75,2.91)
  {$(D_{i+1}\cap\R)\times C_{i+1}$};
\filldraw[fill=figOverlap!18,draw=figInk,thick]
  (0.95,1.35)--(3.0,1.05)--(7.3,1.18)--(9.35,1.65)
  --(7.7,2.35)--(3.2,2.55)--cycle;
\node at (2.4,1.75) {$F$};
\draw[figInk,thick] (5.15,1.115)--(5.15,2.455);
\node[anchor=west,inner sep=1.5pt] at (5.3,1.8) {$F_b$};
\draw[figOverlap,densely dotted] (4.65,0.8)--(4.65,-1.08);
\draw[figOverlap,densely dotted] (5.65,0.8)--(5.65,-1.08);
\draw[figInk,->] (0,-0.05)--(10.6,-0.05) node[right] {$t$};
\draw[figOverlap,line width=3pt] (4.65,-0.05)--(5.65,-0.05);
\draw[figInk] (5.15,-0.13)--(5.15,0.03);
\node[anchor=north] at (5.15,-0.17) {$b$};
\node[anchor=south] at (5.15,0.07) {$J_i$};
\draw[figOld,line width=1.5pt] (0.65,-1.0)--(4.65,-1.0);
\draw[figOverlap,line width=1.5pt] (4.65,-1.0)--(5.65,-1.0);
\draw[figNew,line width=1.5pt] (5.65,-1.0)--(9.65,-1.0);
\fill[figOld] (4.65,-1.0) circle (1.6pt);
\fill[figNew] (5.65,-1.0) circle (1.6pt);
\node[figOld,anchor=north] at (2.6,-1.13) {$g_i(t,\cdot)$};
\node[anchor=north] at (5.15,-1.13)
  {$\widehat h^{(i)}_{\rho_i(t)}(t,\cdot)$};
\node[figNew,anchor=north] at (7.75,-1.13) {$g_{i+1}(t,\cdot)$};
\node[anchor=north] at (5.15,-1.79)
  {A smooth real-$t$ family on the common $B\times P$};
\end{tikzpicture}
\caption{Joining local approximants along a real face. The upper panel
is schematic in the $y$ directions and omits the passive factor $P$.
The local joining path is extended to $B$ before it is inserted over
$J_i$. Exact endpoint agreement makes the lower family smooth;
holomorphic dependence on $t$ is obtained afterward.}
\label{fig:real-face-joining}
\end{figure}

\begin{proof}
Write the map as $f(t,y,u)$. By first enlarging $F$ slightly inside
its real hyperplane, and reserving a slightly larger passive polydisc,
we may suppose that its projection $I$ to $\R$ has positive length.
For $t\in I$, write $F_t=\{y:(t,y)\in F\}$ for the slice,
and choose a compact convex neighborhood $C_t$ of $F_t$ and a complex disc $D_t$ about $t$ such that $f$ is
holomorphic near $\overline D_t\times C_t\times P$. Nearby slices
remain in $\Int C_t$: otherwise a convergent sequence of points in
those slices would give a point of $F_t$ outside $\Int C_t$.
A finite subdivision of $I$ therefore assigns to each subinterval a
product $D_i\times C_i\times P$ containing all its slices with
positive margins.

At a division point $b$, shrink the joining disc and choose compact
convex sets $C_b^-\Subset\Int C_b$ inside
$\Int(C_i\cap C_{i+1})$, with every nearby slice contained in
$\Int C_b^-$. After all division points have been treated, choose a
closed polydisc $B$ containing all the $C_i$ in its interior. Property $\mathsf A_{r-1}$, with $(t,u)$ passive,
gives entire maps $g_i$ as close as desired to $f$ on the assigned
products. A chart spray of $f$ on a joining product writes the two
adjacent maps as
\[
 g_i=f^\sharp(\,\cdot\,,a_i),\qquad
 g_{i+1}=f^\sharp(\,\cdot\,,a_{i+1}),
\]
where both fiber coordinates are uniformly small. If
$\chi:[0,1]\to[0,1]$ is smooth, zero near $0$ and one near $1$, then
\[
 h_\tau=f^\sharp\bigl(\,\cdot\,,
          (1-\chi(\tau))a_i+\chi(\tau)a_{i+1}\bigr)
\]
is a local joining path uniformly close to $f$. Its chart and fiber
coordinates are defined only on the joining product, so the path
does not yet extend over $B$. Apply
Lemma~\ref{convex:path} with active variables $y$, with the smaller
joining disc and $u$ as passive variables, and with the common
polydisc $B$. We obtain a path $\widehat h^{(i)}_\tau$ defined over
$B$, still close to $f$ on the slices near $b$, and agreeing exactly
with $g_i,g_{i+1}$ on endpoint intervals.

The path parameter $\tau$ is smooth, while $t$ remains a passive
holomorphic variable in $\widehat h^{(i)}_\tau$.
Choose disjoint small real joining intervals $J_i$ and smooth
functions $\rho_i:J_i\to[0,1]$ equal to $0$ and $1$ near their left
and right ends. Define
\[
 H(t,y,u)=
 \begin{cases}
 \widehat h^{(i)}_{\rho_i(t)}(t,y,u),&t\in J_i,\\
 g_i(t,y,u),&t\text{ between successive joining intervals}.
 \end{cases}
\]
Exact endpoint agreement makes $H$ smooth in $t$ and holomorphic in
$(y,u)$ on a neighborhood of $I\times B\times P$. Extend it a little
past the endpoints of $I$ using the first and last entire maps.
It approximates $f$ uniformly on $F\times P$ as closely as desired.
Lemma~\ref{convex:complexify} now makes $H$ holomorphic near the
whole product. Apply Lemma~\ref{convex:interval} with $s=0$, since
$B\times P$ is a polydisc, to obtain the required entire map.
Here $\mathsf A_0$ follows from $\mathsf A_{r-1}$ by taking the
active convex factor to be a point and restricting the resulting
entire map to that point. When $r=1$, omit $y$ and the corresponding
convex factors throughout the construction.
\end{proof}

\subsection{Fixing the collar before gluing}\label{sec:fixed-collar}

Kusakabe~\cite{Kus21} uses a related
construction combining a local fiberwise right inverse, Oka--Weil
approximation, and gluing of sprays. A close approximation on a real face need not be close on a
predetermined complex collar. The accuracy of the first entire
approximation is used only to obtain a fiber inverse on some collar.
After fixing this collar and its splitting constant, polynomial
approximation of the fixed inverse makes the gluing error arbitrarily
small.

\begin{lemma}\label{convex:retune}
Let $M\subset\C^m$ be compact and
$F:\Op(M)\times\Delta_{4b}^n\to Y$, $b>0$, be a chart spray, so
$(Z,v)\mapsto(Z,F(Z,v))$ is biholomorphic onto its image.
If an entire map $G:\C^m\times\C^n\to Y$ is sufficiently close to
$F$ on $M\times\overline\Delta_{3b}^n$, then there are a complex
neighborhood $U$ of $M$ and a holomorphic map
$W:U\times\Delta_{2b}^n\to\C^n$ such that
\begin{equation}\label{convex:exact-inverse}
 G(Z,W(Z,v))=F(Z,v).
\end{equation}
For every compact convex $C\subset U$, polynomial maps $q(Z,v)$
can make the transition
\[
 F(Z,\,\cdot\,)^{-1}G(Z,q(Z,v))=v+c(Z,v)
\]
well-defined near $C\times\overline\Delta_b^n$, with both
$\sup\norm{c}$ and $\sup\norm{D_vc}$ arbitrarily small there.
The maps $F,G,W$, the neighborhood $U$, and the compact set $C$ remain
fixed during this final approximation.
\end{lemma}
\begin{proof}
For sufficiently close approximation the old chart inverse defines
\[
 \Gamma(Z,w)=F(Z,\,\cdot\,)^{-1}G(Z,w)
\]
near $M\times\overline\Delta_{3b}^n$, uniformly close to $w$.
Cauchy estimates only in $w$ give, after choosing the approximation
sufficiently accurate,
\[
 \norm{\Gamma-\mathrm{id}}<b/4,\qquad
 \norm{D_w\Gamma-I}<1/2
\]
on $M\times\overline\Delta_{5b/2}^n$, using the maximum norm and
its induced operator norm. Fix $G$. The strict inequalities persist
on $U\times\overline\Delta_{5b/2}^n$ for a complex neighborhood
$U$ of $M$. For $v\in\Delta_{2b}^n$, the map
\[
 w\longmapsto v-(\Gamma(Z,w)-w)
\]
is a contraction of $\overline\Delta_{5b/2}^n$ into itself.
Its unique fixed point $W(Z,v)$ satisfies
\eqref{convex:exact-inverse}. The holomorphic implicit function
theorem and uniqueness give joint holomorphicity in $(Z,v)$.

Now fix a compact convex $C\subset U$. Approximate $W$
componentwise by polynomials on the compact convex set
$C\times\overline\Delta_{3b/2}^n$. For these polynomial maps $q$,
\eqref{convex:exact-inverse} and uniform continuity of the fixed
$G$ show that $G(Z,q(Z,v))$ converges uniformly to $F(Z,v)$ there.
The old chart inverse consequently defines a transition converging
uniformly to $v$. Cauchy estimates across the fixed fiber gap from
$3b/2$ to $b$ give uniform convergence of $D_vc$ to zero as well.
The compositions $G(Z,q(Z,v))$ are entire, leaving the construction
unrestricted away from this compact overlap.
\end{proof}

\begin{lemma}\label{convex:one-cut}
Assume $\mathsf A_{r-1}$, $r\geq1$. Let $Q\subset\C^r$ be compact
convex and $P$ a closed polydisc. Set
$Q_-=Q\cap\{\Im z_1\leq0\}$ and suppose that $Q_-$ has nonempty
interior. Every map holomorphic near $Q_-\times P$ can be approximated
uniformly there by maps holomorphic near $Q\times P$.
\end{lemma}
\begin{proof}
There is nothing to prove if $Q_-=Q$. Otherwise the cutting
hyperplane meets $\Int Q$. Set $K=Q\times P$ and $Z=(z,u)$.
By Lemma~\ref{lem:chart-spray}, fix a chart spray $F(Z,v)$ centered
at the given map $f$ over a neighborhood of $K_-=Q_-\times P$,
with fiber polydisc $\Delta_{4b}^n$. Choose an enlarged compact
convex real face $\Sigma^+\subset\R\times\C^{r-1}$ and an enlarged
passive polydisc $P^+$ such that
\[
 M=\Sigma^+\times P^+
\]
contains $K\cap\{\Im z_1=0\}$ in its relative interior and $F$ is
holomorphic near $M\times\overline\Delta_{3b}^n$.
Lemma~\ref{convex:face}, with $(u,v)$ passive, supplies an entire
$G(Z,w)$ sufficiently close to $F$ on that product for
Lemma~\ref{convex:retune}. Fix this $G$ and the resulting inverse
$W$ on $U\times\Delta_{2b}^n$.

Choose a bounded smoothly bounded strongly convex domain $D\supset K$
sufficiently close to $K$ that its closed lower part lies in the base
domain of $F$ and its closed section by $\Im z_1=0$ lies in $U$.
Compactness then gives $\delta>0$ such that
\[
 \overline D\cap\{\Im z_1\leq2\delta\}
 \subset\operatorname{dom}_Z F,
 \qquad
 C=\overline D\cap\{|\Im z_1|\leq2\delta\}\Subset U.
\]
Set
\[
 A=D\cap\{\Im z_1<\delta\},\qquad
 B=D\cap\{\Im z_1>-\delta\}.
\]
These domains have separated closed fringes, $A\cup B=D$, and
$\overline{A\cap B}\subset C$, where $C$ is compact convex.
In particular, they satisfy the bounded center-gluing hypotheses of
Lemma~\ref{lem:center-gluing}. Fix the domains, their splitting
constant, and all fiber margins.

Figure~\ref{fig:real-face-collar} summarizes the
source geometry and the order of these choices; the other source
coordinates are suppressed in the diagram.

Lemma~\ref{convex:retune} now gives a polynomial $q$ for which the
sprays
\[
 F_A(Z,v)=F(Z,v),\qquad F_B(Z,v)=G(Z,q(Z,v))
\]
have transition $v+c(Z,v)$ on the overlap, with both $c$ and $D_vc$
as small as desired on the fixed fiber polydisc $\Delta_b^n$.
The second family is entire. Lemma~\ref{lem:center-gluing} gives
small holomorphic sections $a$ on $A$ and $b_1$ on $B$ satisfying
\[
 F_A(Z,a(Z))=F_B(Z,b_1(Z))\quad (Z\in A\cap B).
\]
Their centers define a holomorphic map on $D$. As the polynomial
approximation improves, $\norm{a}_A\to0$, and hence this map
converges uniformly to $f=F(\,\cdot\,,0)$ on $K_-\Subset A$.
Since $D$ contains $K$, the assertion follows.
\end{proof}

\begin{figure}[htbp]
\centering
\begingroup
\begin{tikzpicture}[x=1cm,y=1cm,line cap=round,line join=round,
  every node/.style={font=\small,text=figInk},
  choice/.style={rounded corners=2pt,align=center,text width=4.25cm,
    inner xsep=7pt,inner ysep=3.5pt,line width=.65pt},
  orderarrow/.style={->,>=latex,draw=figInk!80,line width=.75pt}]
\begin{scope}[yshift=0cm]
  \begin{scope}
    \clip (0,0) circle[radius=2.45];
    \fill[figOld!22] (-3,-3) rectangle (3,.42);
    \fill[figNew!22] (-3,-.42) rectangle (3,3);
    \fill[figOverlap!25] (-3,-.84) rectangle (3,.84);
    \fill[figOverlap!55] (-3,-.42) rectangle (3,.42);
    \draw[figOverlap!75!black,line width=.8pt] (-3,.84)--(3,.84);
    \draw[figOverlap!75!black,line width=.8pt] (-3,-.84)--(3,-.84);
    \draw[figOld,densely dashed,line width=.7pt] (-3,.42)--(3,.42);
    \draw[figNew,densely dashed,line width=.7pt] (-3,-.42)--(3,-.42);
    \draw[figInk!85,densely dotted,line width=.8pt] (-3,0)--(3,0);
  \end{scope}
  \draw[figInk,line width=.85pt] (0,0) circle[radius=2.45];
  \draw[figInk!55,line width=.5pt] (-2.8,0)--(-2.45,0);
  \draw[->,figInk!65,line width=.5pt] (2.45,0)--(2.9,0);
  \draw[->,figInk!55,line width=.5pt] (0,-2.65)--(0,2.72);
  \node[above] at (0,2.72) {$\Im z_1$};
  \node[right] at (2.88,0) {$\Re z_1$};
  \node[align=center,text=figNew] at (-1.10,1.42)
    {$B$\\[-1pt]\scriptsize new side};
  \node[align=center,text=figOld] at (-1.10,-1.42)
    {$A$\\[-1pt]\scriptsize old side};
  \node[text=figInk] at (-1.35,.20) {$A\cap B$};
  \node[text=figInk] at (1.35,.64) {$C$};
  \node[below left,inner sep=1pt] at (-.04,-.02) {$0$};
  \node[above right,inner sep=1.5pt,text=figOld] at (.05,.42) {$\delta$};
  \node[above right,inner sep=1.5pt,text=figInk] at (.05,.84) {$2\delta$};
  \node[below right,inner sep=1.5pt,text=figNew] at (.05,-.42) {$-\delta$};
  \node[below right,inner sep=1.5pt,text=figInk] at (.05,-.84) {$-2\delta$};
  \node[above left,inner sep=1pt] at (-1.78,1.72) {$\partial D$};
  \node[align=center,text width=5.5cm,font=\footnotesize] at (0,-3.05)
    {Real face: $\Im z_1=0$.\\[2pt]
     $\overline{A\cap B}\subset C\Subset U$.};
\end{scope}
\node[font=\small\bfseries] at (6.45,2.98) {Order of choices};
\node[choice,draw=figNew!85,fill=figNew!10,
  minimum height=.95cm] (coarse) at (6.45,1.895)
  {\textcolor{figNew}{$G$}\\[-1pt]\scriptsize one entire approximant};
\node[choice,draw=figOverlap!80!black,fill=figOverlap!14,
  minimum height=.95cm,anchor=north] (inverse)
  at ([yshift=-.36cm]coarse.south)
  {\textcolor{figInk}{$U$}\\[-1pt]\scriptsize base domain for $W$};
\node[choice,draw=figOverlap!80!black,fill=figOverlap!14,
  minimum height=1.16cm,anchor=north] (fixed)
  at ([yshift=-.36cm]inverse.south)
  {Fix $D,\delta$, hence\\[-1pt]
   $\textcolor{figOld}{A},\ \textcolor{figNew}{B};\
    \textcolor{figInk}{C\Subset U}$\\[-1pt]
   \scriptsize and the gluing constants};
\node[choice,draw=figNew!85,fill=figNew!10,
  minimum height=1.16cm,anchor=north] (fine)
  at ([yshift=-.36cm]fixed.south)
  {\textcolor{figNew}{$q$}\\[-1pt]
   \scriptsize approximate $W$ by polynomials\\[-1pt]
   \scriptsize to any required accuracy};
\draw[orderarrow] (coarse.south)--(inverse.north);
\draw[orderarrow] (inverse.south)--(fixed.north);
\draw[orderarrow] (fixed.south)--(fine.north);
\end{tikzpicture}
\endgroup
\caption{A schematic slice of the one-cut construction. The dashed
lines bound the open overlap $A\cap B$; the solid
lines bound the larger compact collar $C$, which includes the intervening
boundary points of $D$. The central dotted line is the real face.
The entire approximant is fixed before the
collar geometry, and only then is the inverse approximation refined.}
\label{fig:real-face-collar}
\end{figure}
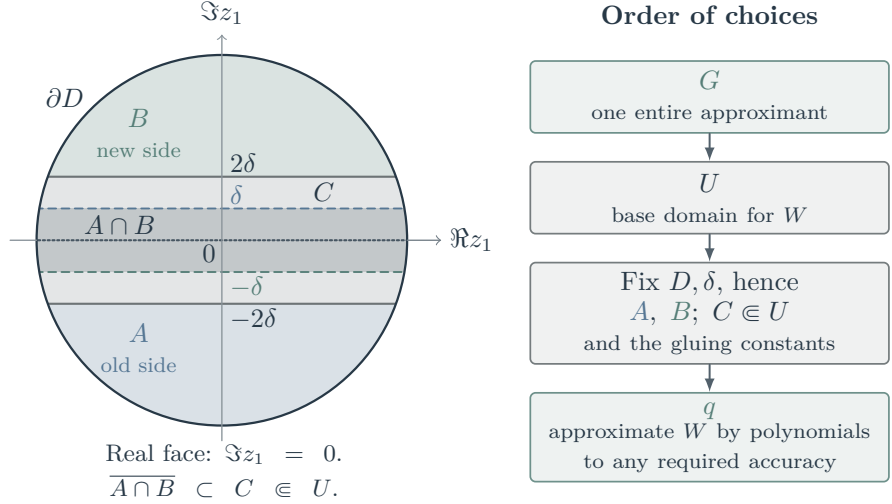

\subsection{Finite removal of the cuts}\label{sec:finite-cuts}

A polytope inside the
original domain provides the initial compact set; successively dropping
its defining inequalities enlarges it to a polydisc.

\begin{proof}[Proof of Theorem~\ref{thm:a0-cap}]
Assume inductively that $\mathsf A_{r-1}$ of
Definition~\ref{def:approximation-level} holds. Let $K\subset\C^r$
be compact convex, $P$ a closed polydisc, and $f$
holomorphic near $K\times P$. Fix $\epsilon>0$ and choose a
full-dimensional compact real convex polytope $K'$ such that
\[
 K\Subset\Int K',\qquad K'\times P\Subset\operatorname{dom}f.
\]
Write $K'=\bigcap_{j=1}^N\{\ell_j\leq c_j\}$ with real affine
functionals $\ell_j$, and choose $R$ with
$K'\Subset\Int\PD Rr$. Define
\[
 Q_k=\PD Rr\cap\bigcap_{j=k+1}^N\{\ell_j\leq c_j\},
 \qquad 0\leq k\leq N.
\]
Then $Q_0=K'$, $Q_N=\PD Rr$, and
$Q_{k-1}=Q_k\cap\{\ell_k\leq c_k\}$. Each $Q_{k-1}$ has
nonempty interior because it contains $K'$.

For a nonconstant $\ell_k(z)=\Re\sum_j a_jz_j+d$, the affine
coordinate
$w_1=i(\sum_j a_jz_j+d-c_k)$ transforms its inequality into
$\Im w_1\leq0$; complete the linear part to a complex coordinate
system. Convexity is preserved and the passive polydisc is unchanged.
Thus Lemma~\ref{convex:one-cut} applies at each nontrivial step;
a redundant cut needs no change of map. Starting from $f_0=f$, obtain
maps $f_k$ holomorphic near $Q_k\times P$ such that
\[
 \sup_{Q_{k-1}\times P}d_Y(f_k,f_{k-1})<\epsilon_k,
 \qquad \sum_{k=1}^N\epsilon_k<\epsilon/2.
\]
Finally, $\mathsf A_0$ approximates $f_N$ on the polydisc
$\PD Rr\times P$ by an entire map within $\epsilon/2$. The triangle
inequality gives approximation to $f$ within $\epsilon$ on
$K\times P$. This proves $\mathsf A_r$ and completes the induction.
Taking the passive factor to be a point gives the convex approximation
property.
\end{proof}

\begin{proof}[Completion of the proofs of Theorems~\ref{thm:families-oka}
 and~\ref{thm:strong-domination}]
\LAP{} gives local planar extension by
Lemma~\ref{family:local-sector}, polydisc approximation by
Theorem~\ref{thm:polydisc}, and convex approximation by
Theorem~\ref{thm:a0-cap}. This proves the \LAP{} criterion.
Proposition~\ref{family:strong-domination} therefore makes every
strongly dominable manifold Oka. The converse is classical
\cite{F17}; Proposition~\ref{family:strong-domination} then also
gives \LAP{} for every Oka manifold. In dimension zero all three
properties are immediate.
\end{proof}

\subsection{Direct consequences}
\label{sec:direct-consequences}

\begin{proof}[Proof of Corollary~\ref{cor:capn}]
Only sufficiency needs proof. For $s>0$, let $B_s$ denote the
Euclidean ball of radius $s$ in $\C^n$. Fix $x\in X$, choose $r>0$,
and take a coordinate parametrization $\phi:B_{2r}\to X$ with
$\phi(0)=x$. By $\CAP_n$,
an entire map $G:\C^n\to X$ approximates $\phi$ arbitrarily closely
on $\overline B_{3r/2}$. Choose the approximation so close that its
image on this ball lies in $\phi(B_{2r})$ and, for
$g=\phi^{-1}\circ G$,
\[
 \sup_{\overline B_r}|g(z)-z|<r/2,
 \qquad
 \sup_{\overline B_r}\norm{dg_z-I}<1/2.
\]
The second estimate follows from Cauchy estimates on the larger ball.
The map $z\mapsto z-g(z)$ is therefore a contraction of
$\overline B_r$ into itself. Its fixed point $a$ satisfies $g(a)=0$,
and $dg_a$ is invertible. Thus $z\mapsto G(z+a)$ dominates $X$ at
$x$. Apply Theorem~\ref{thm:strong-domination}.

The same argument applies to the special convex sets in
Forstneri\v{c}'s original definition
\cite[Definition~1.1 and p.~692]{F06}: take two nested balls in
the interior of such a set and rescale the coordinate parametrization
to be defined near the whole set.
\end{proof}

\begin{proof}[Proof of Corollary~\ref{cor:stratification}]
Combine the strong dominability theorem of Forstneri\v{c} and
L\'arusson \cite{FL14} with Theorem~\ref{thm:strong-domination},
componentwise.
\end{proof}

\begin{proof}[Proof of Corollary~\ref{cor:euclidean-localization}]
By Theorem~\ref{thm:strong-domination}, it suffices to observe that
a dominating map into an open subset also dominates the ambient manifold.
\end{proof}

\section{Removability and holomorphic images}\label{sec:removability-images}

\subsection{Removability across thin exceptional sets}
\label{sec:removability}

The \LAP{} criterion also applies when \LAP{} is known
only away from a sufficiently thin set, more precisely, a set of zero
$(2n-1)$-dimensional Hausdorff measure.
The exponent $2n-1$ comes from
the parameters of a coordinate circle about a fixed point: its complex
line contributes $2n-2$ real dimensions and its radius contributes one.
A set of zero $(2n-1)$-dimensional Hausdorff measure cannot meet every
such circle.

For a closed exceptional set with this Hausdorff-measure bound,
Alarc\'on and Forstneri\v{c} proved that dominability off the set
implies the Oka-1 property~\cite{AF25}. Theorem~\ref{thm:thin-removability}
gives the full Oka property and also allows a nonclosed exceptional set.

We first recenter the families supplied by
Lemma~\ref{poly:boundary-disc}.

\begin{lemma}\label{family:recentering}
For $r>0$ and closed polydiscs $P_j\subset\C^{n-1}$ with
$0\in\operatorname{Int}P_j$, $1\leq j\leq n$, suppose that
$\psi_j$ is a biholomorphic coordinate parametrization near
$\overline\Delta_r\times P_j$, $\psi_j(0,0)=x$, and the vectors
$\partial_z\psi_j(0,0)$ form a basis of $T_xY$. Also suppose that for
some $S_\nu\to\infty$ there are holomorphic maps
\[
 \sigma_{j,\nu}:\operatorname{Op}
 (\overline\Delta_{S_\nu}\times P_j)\longrightarrow Y
\]
converging uniformly to $\psi_j$ on
$\overline\Delta_r\times P_j$. Then $x\in G_Y$.
\end{lemma}
\begin{proof}
Fix small source polydiscs $Q_j$ around $(0,0)$, with closures in
$\Delta_r\times\operatorname{Int}P_j$. On slightly larger fixed
polydiscs, uniform convergence places the images of $\sigma_{j,\nu}$
in the $\psi_j$-chart for all sufficiently large $\nu$.
Cauchy estimates give
\[
 \psi_j^{-1}\circ\sigma_{j,\nu}
   =\operatorname{id}+e_{j,\nu},\qquad
 e_{j,\nu}\longrightarrow0\quad\hbox{in }C^1(\overline Q_j).
\]
For example, if $Q_j=\Delta_{\rho_j}^n$ after separate coordinate
rescalings, arrange
$\|e_{j,\nu}\|<\rho_j/4$ and $\|De_{j,\nu}\|<1/2$ there.
For a point in $\Delta_{\rho_j/2}^n$, the equation
$z+e_{j,\nu}(z)=w$ has a unique solution by contraction on the
closed $3\rho_j/4$ polydisc. The solutions depend holomorphically
on $w$ and converge uniformly on smaller products to $w$.

Intersect the resulting finitely many target neighborhoods and
shrink once to a relatively compact neighborhood $V\ni x$.
After discarding a common finite initial segment, this gives
holomorphic branches
\[
 \ell_{j,\nu}=(s_{j,\nu},p_{j,\nu}):V\longrightarrow
             \Delta_r\times\operatorname{Int}P_j,
 \qquad \sigma_{j,\nu}\circ\ell_{j,\nu}=\operatorname{id}_V,
\]
converging to $\psi_j^{-1}$, with images in fixed compact
subproducts. In particular, $|s_{j,\nu}(y)|<C$ for one constant
$C$, independent of $j,\nu,y$.

Set $R_\nu=S_\nu-C$ and discard further initial terms so that
$R_\nu>0$. Define
\[
 h_{j,\nu}(y,t)=\sigma_{j,\nu}
       \bigl(s_{j,\nu}(y)+t,p_{j,\nu}(y)\bigr).
\]
These maps are holomorphic on the common products
$V\times\Delta_{R_\nu}$ and satisfy $h_{j,\nu}(y,0)=y$ exactly.
The fixed compact margins permit one $\tau>0$ such that, for
$y\in V$ and $|t|<\tau$, all translated arguments remain inside
the original convergence products. Thus the new families converge
locally uniformly there to
\[
 h_j(y,t)=\psi_j\bigl(\psi_j^{-1}(y)+(t,0)\bigr).
\]
Their initial velocities at $x$ are $\partial_z\psi_j(0,0)$.
These are precisely the data of Definition~\ref{family:condition}.
\end{proof}

\begin{proposition}\label{family:boundary-propagation}
Let $\vartheta$ be a target chart about $x\in Y$. If the closed
coordinate disc
$\vartheta(x)+\overline\Delta_r v$, $v\ne0$, lies in the chart
image and its boundary lies in $\vartheta(G_Y)$, then $x\in G_Y$.
\end{proposition}
\begin{proof}
Compactness of the boundary and openness of $G_Y$ permit sufficiently
small perturbations of $v$ while retaining both chart containment
and the good-boundary condition. We choose $n$ linearly independent
vectors $v_j$ among these perturbations, complete each to a basis
$v_j,e_{j,2},\ldots,e_{j,n}$, and set
\[
 \psi_j(z,p)=\vartheta^{-1}\left(
 \vartheta(x)+zv_j+\sum_{k=2}^n p_{k-1}e_{j,k}\right).
\]
For sufficiently small transverse polydiscs $P_j$, these are
coordinate parametrizations near
$\overline\Delta_r\times P_j$. To choose $P_j$, first cover the
central boundary circle by finitely many arcs whose images lie
relatively compactly in neighborhoods supplied by
Lemma~\ref{family:local-sector}, and then shrink the transverse
polydisc so that every corresponding closed arc product lies in
its chosen neighborhood. This choice is independent of all later
approximation indices.

For any common sequence $S_\nu\to\infty$,
Lemma~\ref{poly:boundary-disc} gives maps $\sigma_{j,\nu}$
holomorphic near $\overline\Delta_{S_\nu}\times P_j$, converging
uniformly to $\psi_j$ on the original closed product. Their limiting
active velocities are $(d\vartheta_x)^{-1}v_j$, a basis.
Lemma~\ref{family:recentering} now gives $x\in G_Y$.
\end{proof}

\begin{proof}[Proof of Theorem~\ref{thm:thin-removability}]
By hypothesis, $Y\setminus E\subset G_Y$. Fix $x\in Y$ and take coordinates centered
at $x$. Choose $0<a<b$ so that the closed ball of radius $b$ lies in
the chart. On the annulus $a\leq|z|\leq b$, the smooth map
\[
 q(z)=([z],|z|)\in\PP^{n-1}\times[a,b]
\]
is Lipschitz, so the image of the coordinate representative of $E$
there has zero $(2n-1)$-dimensional Hausdorff outer measure. The
manifold $\PP^{n-1}\times(a,b)$ has positive measure in that
dimension. Choose $([v],r)$ outside this image, with $|v|=1$.
Its fiber is the circle $\{\zeta v:|\zeta|=r\}$, whose inverse
image under the chart lies in $G_Y$; the filled disc lies in the chart
image. Proposition~\ref{family:boundary-propagation} gives $x\in G_Y$.
This also covers $n=1$, when $\PP^0$ is a
point. Thus $G_Y=Y$, and Theorem~\ref{thm:families-oka} makes $Y$ Oka.
\end{proof}

\begin{proof}[Proof of Theorem~\ref{thm:analytic-removability}]
Set $n=\dim Y$. The case $n=0$ is immediate. For $n\geq1$, every
subset of a proper complex analytic subset of a connected $n$-manifold has
zero $(2n-1)$-dimensional Hausdorff outer measure. A countable cover
by the given neighborhoods therefore shows that $Y\setminus G_Y$
has zero measure in this dimension, so
Theorem~\ref{thm:thin-removability} applies.
For the particular case, dominating maps into $X\setminus A$ also
dominate $X$ after composition with the inclusion. Thus
Proposition~\ref{family:strong-domination} gives \LAP{} on
$X\setminus A$, and the first assertion applies.
\end{proof}

By Proposition~\ref{family:strong-domination}, Theorem~\ref{thm:thin-removability}
also applies when $Y$ is dominable at every point outside $E$.
In particular, if such an exceptional set $E$ is closed and $Y\setminus E$ is Oka, then $Y$
is Oka.

The following result was proved by Wang, Xie, and Zhao
\cite[Corollary~1.3]{WXZ26}; it also follows from OpenAI's theorem
that every complex K3 surface is Oka \cite[Theorem~1.1]{OAI26}.

\begin{corollary}\label{cor:kummer-recovery}
Every complex Kummer surface is Oka.
\end{corollary}
\begin{proof}
For a Kummer surface $X$, the complement of its sixteen exceptional
curves is Oka by Forstneri\v{c} and L\'arusson
\cite[Lemma~7]{FL14}, so Theorem~\ref{thm:analytic-removability}
applies.
\end{proof}

For every closed subset $K$ of an affine complex hyperplane
$H\subset\C^n$, $n\geq1$, the complement $\C^n\setminus K$ is Oka.
Indeed, apply Theorem~\ref{thm:thin-removability} with
$Y=\C^n\setminus K$ and $E=H\setminus K$. The set $E$ has zero
$(2n-1)$-dimensional Hausdorff measure, and
$Y\setminus E=\C^n\setminus H\cong\C^{n-1}\times\C^*$ is Oka. Set
$T=\{0\}\cup\{1/j:j\geq1\}$. For $n\geq3$, this includes
\[
 \C^n\setminus\bigl(T^2\times\{0\}^{n-2}\bigr).
\]
A complex manifold $Y$ is \emph{weakly subelliptic} if, for every
compact set $K\subset Y$, there are finitely many holomorphic sprays
on $Y$ whose vertical differentials jointly span $T_yY$ at every
$y\in K$. Kusakabe proved that this example is Oka but not weakly
subelliptic~\cite{KusCount20}.

\begin{remark}
A complex manifold has the \emph{density property} if every holomorphic
vector field can be approximated uniformly on compact sets by Lie
combinations of holomorphic vector fields whose flows exist for every
complex time. Another direct recovery is Kusakabe's theorem on compact
holomorphically convex complements \cite{KusConvex24}. If $X$ is
a Stein manifold of dimension $n>1$ with the density property and
$K\subset X$ is compact and $\OO(X)$-convex, the construction of
Forstneri\v{c} and Wold places every point of $X\setminus K$ in
an open subset biholomorphic to $\C^n$ \cite{FWFB20}.
Thus $X\setminus K$ is strongly dominable, and
Theorem~\ref{thm:strong-domination} makes it Oka.
\end{remark}

\subsection{Surjective holomorphic images, complements, and group orbits}
\label{sec:permanence}

The proof for a surjective holomorphic map into $Y$ uses the fact
that its critical values have zero $(2\dim Y-1)$-dimensional Hausdorff measure, even
when the map is not proper.

\subsubsection{Surjective holomorphic images}

\begin{theorem}\label{thm:surjective-descent}
Let $p:X\to Y$ be a surjective holomorphic map of connected complex
manifolds. If $X$ is Oka, then $Y$ is Oka.
\end{theorem}
\begin{proof}
Set $m=\dim X$ and $n=\dim Y$. The case $n=0$ is immediate.
For $n\geq1$, let
\[
 C_p=\{x\in X:\operatorname{rank}dp_x<n\},
 \qquad E=p(C_p).
\]
The closed analytic set $C_p$ admits a locally finite stratification
by complex submanifolds on which the restrictions of $p$ have constant
rank \cite[Chapter~2, Proposition~3]{Schwartz66}. There are countably many strata, and each has
a countable constant-rank coordinate cover. On every such coordinate
neighborhood the image is contained in a complex submanifold of $Y$
of dimension at most $n-1$. Hence
$\mathcal H^{2n-1}(E)=0$ in the outer-measure sense.

For $y\in Y\setminus E$, choose $x\in p^{-1}(y)$.
Strong dominability of $X$ gives an entire map $F:\C^m\to X$ with
$F(0)=x$ and $dF_0$ invertible. Since $dp_x$ is surjective, there is
a linear map $\iota:\C^n\to\C^m$ for which
$d(p\circ F)_0\circ\iota$ is an isomorphism.
Thus $p\circ F\circ\iota$ dominates $Y$ at $y$.
Theorem~\ref{thm:thin-removability} completes the proof.
\end{proof}

\begin{corollary}\label{cor:blowdown}
Let $\pi:\widetilde X\to X$ be a proper bimeromorphic holomorphic
map of connected complex manifolds. If $\widetilde X$ is Oka,
then $X$ is Oka. In particular, this holds when
$\widetilde X=\operatorname{Bl}_Z X$ for a smooth closed complex
submanifold $Z\subsetneq X$.
\end{corollary}
\begin{proof}
This follows directly from Theorem~\ref{thm:surjective-descent},
since a proper bimeromorphic map is surjective.
\end{proof}

This answers the blow-down question of Forstneri\v{c} and
L\'arusson \cite[Section~6]{FL14}.

Following Forstneri\v{c} \cite[Definition~1.2]{FSurj17}, a connected complex
manifold $Y$ has the \emph{basic Oka property with surjectivity}
(BOPS) if every continuous map $S\to Y$ from a Stein manifold
with $\dim S\geq\dim Y$ is homotopic to a surjective holomorphic
map.

\begin{corollary}\label{cor:surjective-entire}
For a connected complex manifold $Y$ of dimension $n\geq1$, the
following conditions are equivalent:
\begin{enumerate}[label=\textup{(\roman*)}]
\item $Y$ is Oka;
\item there is a surjective holomorphic map $\C^n\to Y$;
\item there is a surjective holomorphic map $\C^N\to Y$ for some
finite integer $N\geq1$;
\item $Y$ has BOPS.
\end{enumerate}
\end{corollary}
\begin{proof}
Forstneri\v{c}'s surjectivity theorem gives both \textup{(ii)} and
\textup{(iv)} from \textup{(i)} \cite[Theorem~1.1]{FSurj17}. Clearly
\textup{(ii)} implies \textup{(iii)}, while
Theorem~\ref{thm:surjective-descent} gives
\textup{(iii)}$\Rightarrow$\textup{(i)}. Applying BOPS to a
constant map $\C^n\to Y$ gives \textup{(iv)}$\Rightarrow$\textup{(ii)}.
\end{proof}

This answers Forstneri\v{c}'s question whether BOPS implies the Oka
property \cite[Problem~1.4(a)]{FSurj17}.

\subsubsection{Product complements and blowups}

\begin{theorem}\label{appmech:product-complement}
Let $X$ and $Y$ be connected positive-dimensional Oka manifolds,
with $A\subsetneq X$ and $B\subsetneq Y$ closed complex
analytic subsets. If $Z\subset X\times Y$ is a closed complex
analytic subset contained in $A\times B$, then
$(X\times Y)\setminus Z$ is Oka.
\end{theorem}
\begin{proof}
First consider $Z=A\times B$ and fix $(p,q)$ in its complement.
After exchanging the factors, assume $p\notin A$. Set $n=\dim X$
and $m=\dim Y$. Choose entire maps $f:\C^n\to X$ and
$g:\C^m\to Y$, centered at $p$ and $q$ and nonsingular at zero.
The image of $g$ contains a neighborhood of $q$. Since $B$ has
empty interior, some $c\in\C^m$ satisfies $g(c)\notin B$.

The analytic set $f^{-1}(A)$ avoids zero. Cartan's theorem~A
applied to its ideal sheaf gives $h\in\OO(\C^n)$ with
$h|_{f^{-1}(A)}=0$ and $h(0)=1$ \cite{F17}. Define
\[
 \Phi(u,v)=\bigl(f(u),g(c+h(u)(v-c))\bigr).
\]
Whenever $f(u)\in A$, the second component equals $g(c)\notin B$.
Thus $\Phi$ takes values in the complement and
$\Phi(0,0)=(p,q)$. Its differential at zero is block triangular
with invertible diagonal blocks $df_0$ and $dg_0$.
The complement is strongly dominable, hence Oka.

For general $Z$, the manifold $W=(X\times Y)\setminus Z$ contains
the connected dense Oka open set
$U=(X\times Y)\setminus(A\times B)$. The complement $W\setminus U$
is proper and closed analytic in the connected manifold $W$.
Theorem~\ref{thm:analytic-removability} therefore makes $W$ Oka.
\end{proof}

\begin{corollary}\label{appmech:product-blowup}
Let $X,Y,A,B$ be as in Theorem~\ref{appmech:product-complement}.
\begin{enumerate}[label=\textup{(\roman*)}]
\item If $Z\subset A\times B$ is a smooth closed complex submanifold
of $X\times Y$, possibly disconnected, then
$\operatorname{Bl}_Z(X\times Y)$ is Oka. In particular, this applies
to $Z=A\times B$ when both $A$ and $B$ are smooth.
\item For every finite set $D\subset X\times Y$, the complement
$(X\times Y)\setminus D$ and the blowup of $X\times Y$ at the
distinct points of $D$ are Oka.
\end{enumerate}
\end{corollary}
\begin{proof}
For a nonempty smooth center $Z$, the blowup is a connected complex
manifold whose exceptional divisor has complement biholomorphic to
$(X\times Y)\setminus Z$. The latter is Oka by
Theorem~\ref{appmech:product-complement}, so
Theorem~\ref{thm:analytic-removability} applies upstairs.
An empty center gives the Oka product.

For a nonempty finite set $D$, its projections
$A_D=\operatorname{pr}_X(D)$ and $B_D=\operatorname{pr}_Y(D)$
are proper closed analytic subsets because both factors have
positive dimension. Since $D\subset A_D\times B_D$, the theorem
gives the Oka complement, and \textup{(i)} gives the Oka blowup.
The case $D=\varnothing$ is immediate.
\end{proof}

\subsubsection{Open group orbits}

\begin{corollary}\label{cor:open-orbit}
Let a complex Lie group $G$ act holomorphically on a connected
complex manifold $X$. If one $G$-orbit is open, then $X$ is Oka.
\end{corollary}
\begin{proof}
The zero-dimensional case is immediate. For $n=\dim X>0$, take
a basis $\xi_1,\ldots,\xi_N$ of the Lie algebra of $G$, and let
$V_1,\ldots,V_N$ be the induced holomorphic vector fields. The set
\[
 A=\{x\in X:\operatorname{rank}(V_1(x),\ldots,V_N(x))<n\}
\]
is closed analytic, since it is defined by the $n$-fold exterior
products of these fields. It is proper because the fields span
the tangent spaces along the open orbit.

At any $x\notin A$, choose $n$ of the fields forming a basis of
$T_xX$, with indices $j_1,\ldots,j_n$. The entire map
\[
 (t_1,\ldots,t_n)\longmapsto
 \exp(t_n\xi_{j_n})\cdots\exp(t_1\xi_{j_1})\cdot x
\]
dominates $X$ at $x$. A proper analytic subset of the connected
$n$-manifold $X$ has zero $(2n-1)$-dimensional Hausdorff measure.
Theorem~\ref{thm:thin-removability} now applies.
\end{proof}

\section{Rationally connected projective manifolds}
\label{sec:rc}

Campana and Winkelmann constructed dense entire curves in rationally
connected projective manifolds whose images contain any prescribed
countable subset \cite{CW}. Alarc\'on and Forstneri\v{c} conjectured
that these manifolds are Oka-1 \cite{AF25}; Du, Guo, and Xie proved
this in a recent preprint \cite{DGX26b}.
Theorem~\ref{thm:rc} gives the full Oka property, settling the rationally connected case of
Forstneri\v{c}'s question on Oka projective manifolds
\cite{Forstneric23}.

For a complex manifold $X$ of dimension $n$, a rational curve
$g:\PP^1\to X$ is \emph{very free} if $g^*TX$ is ample. By the splitting theorem on $\PP^1$, this means
\[
 g^*TX\simeq\bigoplus_{i=1}^n\OO_{\PP^1}(a_i),\qquad a_i\geq1.
\]

Very free rational curves provide the centered entire families required
by the analytic criterion. The essential deformation statement concerns
first jets: one local section of first-jet evaluation supplies all tangent
directions at once.

Fix distinct points $0,\infty\in\PP^1$ and an affine coordinate $t$ on
$\PP^1\setminus\{\infty\}$ with $t(0)=0$. For a parametrized curve $g$
we write $g'(0)=dg_0(\partial/\partial t)$. The space of first jets of
maps from this fixed pointed source to a manifold $Y$ is $TY$: a jet
records its value and its velocity.

\begin{lemma}\label{rc:jet-families}
Let $Y$ be a complex $n$-manifold, $S$ a complex manifold,
and $F:S\times\C\to Y$ a holomorphic family. Suppose
\[
 J:S\longrightarrow TY,\qquad
 J(s)=\bigl(F(s,0),\partial_tF(s,0)\bigr)
\]
is a submersion at some $s_0\in S$. Then, on a neighborhood $V$ of
$x=F(s_0,0)$, there are holomorphic maps
$h_j:V\times\C\to Y$, $1\leq j\leq n$, such that
$h_j(y,0)=y$ and the vectors $\partial_th_j(y,0)$ form a frame of
$TY|_V$. If $F$ extends to $S\times\PP^1$, so do these families.
\end{lemma}
\begin{proof}
Choose a local holomorphic section $\sigma:W\to S$ of $J$, where
$W\subset TY$ is a neighborhood of $J(s_0)$. Every nonempty open
subset of a finite-dimensional vector space contains a basis, so choose
$v_1,\ldots,v_n\in W\cap T_xY$ forming a basis. Extend them to a local
holomorphic frame $v_j(y)$ and shrink $V$ until all its graphs lie in
$W$. Set
\[
 h_j(y,t)=F\bigl(\sigma(y,v_j(y)),t\bigr).
\]
The identity $J\circ\sigma=\operatorname{id}_W$ gives the asserted
values and velocities. If $F$ extends to $S\times\PP^1$, the same
formula defines $h_j$ on $V\times\PP^1$.
\end{proof}

We use the strong form of rational connectedness: a smooth projective
rationally connected variety over $\C$ admits a very free curve through
every prescribed point~\cite[Theorem~2.40]{Debarre}.

\begin{lemma}\label{rc:jet-smooth}
Let $g:\PP^1\to X$ be very free, where $X$ is a smooth projective
complex manifold. In the scheme $H$ of morphisms of its fixed
projective degree, both $H$ and the first-jet morphism
\[
 J:H\longrightarrow TX,\qquad u\longmapsto(u(0),u'(0)),
\]
are smooth at $g$.
\end{lemma}
\begin{proof}
Let $2[0]$ denote the length-two subscheme defined by $t^2=0$,
and set $E=g^*TX$. In the splitting above, $a_i\geq1$, so
$E(-2[0])=\bigoplus_i\OO_{\PP^1}(a_i-2)$ has summand degrees
$a_i-2\geq-1$. Since $H^1(\PP^1,\OO(k))=0$ for $k\geq-1$,
\[
 H^1(\PP^1,E)=H^1(\PP^1,E(-2[0]))=0.
\]
Fixed-source deformation
theory identifies $T_gH$ with $H^0(\PP^1,E)$ and makes $H$ smooth at
$g$~\cite[Proposition~1.6 and Theorem~1.9]{Debarre}.
Infinitesimal deformation of a first jet identifies
$T_{J(g)}TX$ with $H^0(2[0],E|_{2[0]})$, and $dJ_g$ is restriction.
The sequence
\[
 0\longrightarrow E(-2[0])\longrightarrow E
 \longrightarrow E|_{2[0]}\longrightarrow0
\]
gives the exact segment
\[
 H^0(\PP^1,E)\longrightarrow H^0(2[0],E|_{2[0]})
 \longrightarrow H^1(\PP^1,E(-2[0]))=0.
\]
Thus $dJ_g$ is surjective. Since source and target are smooth
at the relevant points, $J$ is smooth at $g$.
\end{proof}

\begin{proof}[Proof of Theorem~\ref{thm:rc}]
The zero-dimensional case is immediate. Otherwise fix $x\in X$ and
choose one very free curve $g$ with $g(0)=x$. Lemma~\ref{rc:jet-smooth}
and universal evaluation on $H\times\PP^1$, restricted to a smooth
neighborhood of $g$, meet the hypotheses of
Lemma~\ref{rc:jet-families}. We obtain centered entire families on a
common neighborhood of $x$ whose initial velocities span $T_xX$.
Since $x$ was arbitrary, Lemma~\ref{family:entire} and
Theorem~\ref{thm:families-oka} apply.
The Fano assertion follows from the rational connectedness theorem
of Campana and Koll\'ar--Miyaoka--Mori \cite{Campana,KMM}.
\end{proof}

The same argument survives deletion of an algebraic subset of
codimension at least two. To obtain families at every remaining point,
we first find an avoiding very free curve through that point.

\begin{lemma}\label{rc:avoid}
Let $X$ be as in Theorem~\ref{thm:rc} and $A\subset X$ a closed
algebraic subset whose irreducible components have codimension at
least two. For every $x\in X\setminus A$ there is a very free curve
$g:\PP^1\to X$ with $g(0)=x$ and $g(\PP^1)\cap A=\varnothing$.
\end{lemma}
\begin{proof}
Choose a very free curve through $x$. Very freeness is open in its
fixed-degree morphism scheme, by semicontinuity of
$H^1(\PP^1,g^*TX(-2[0]))$. Evaluation at $0$ is smooth there,
with fiber tangent space $H^0(\PP^1,g^*TX(-[0]))$
\cite[Section~1.5, paragraph~1.10]{Debarre}.
Take a nonempty smooth irreducible open subset $S$ of this fiber
within the very-free locus. For every $g\in S$ and $t\ne0$, the vanishing
\[
 H^1\bigl(\PP^1,g^*TX(-[0]-[t])\bigr)=0
\]
shows that evaluation of these tangent sections at $t$ is surjective.
Thus universal evaluation
$e:S\times(\PP^1\setminus\{0\})\to X$ is smooth.

The incidence set $I=\{(g,t):g(t)\in A\}$ is closed in the full
product $S\times\PP^1$ and misses $S\times\{0\}$. Since $e$ is
smooth, the inverse image of each component of $A$, if nonempty,
has the same codimension in $S\times(\PP^1\setminus\{0\})$
as that component has in $X$.
Consequently
\[
 \dim I\leq\dim S+1-2=\dim S-1.
\]
The projection $\operatorname{pr}_S:S\times\PP^1\to S$ is proper,
so $\operatorname{pr}_S(I)$ is closed. Moreover,
\[
 \dim\operatorname{pr}_S(I)\leq\dim I<\dim S,
\]
so this image is a proper subset. Any parameter outside it gives
the required curve.
\end{proof}

\begin{corollary}\label{cor:rc-codim-two}
Under the hypotheses of Lemma~\ref{rc:avoid}, $X\setminus A$ is Oka.
\end{corollary}
\begin{proof}
Fix $x\in X\setminus A$ and choose an avoiding very free curve
through $x$. Avoidance of $A$ on all of $\PP^1$ is an open condition
in the unrestricted morphism scheme: its complement is the proper
projection of the closed universal incidence set. On this open locus,
Lemma~\ref{rc:jet-smooth} still makes first-jet evaluation a submersion
near the chosen curve. Lemma~\ref{rc:jet-families} now gives families
into $X\setminus A$. They exist at every point, so
Lemma~\ref{family:entire} and Theorem~\ref{thm:families-oka} prove the claim.
\end{proof}

\subsection{Ellipticity and birational models}

\begin{corollary}\label{cor:rc-elliptic}
Every smooth connected complex projective rationally connected
manifold, in particular every complex projective Fano manifold, is
holomorphically elliptic. Every smooth projective birational model of
such a manifold is also Oka and holomorphically elliptic.
\end{corollary}
\begin{proof}
Theorem~\ref{thm:rc} gives the Oka property. The theorem of
Forstneri\v{c}--L\'arusson that every projective Oka manifold is
holomorphically elliptic \cite{FLProjectiveElliptic} gives the first
assertion. Rational connectedness is invariant under birational maps
between smooth projective manifolds~\cite[Remark~2.21]{Debarre}, so the same
argument applies to every indicated birational model.
\end{proof}

\subsection{Compact K\"ahler consequences}
Two structure theorems reduce the following compact K\"ahler cases
to the projective rationally connected case.

\begin{corollary}\label{cor:kahler-hsc}
Every compact K\"ahler manifold admitting a K\"ahler metric with
nonnegative holomorphic sectional curvature is Oka.
\end{corollary}
\begin{proof}
Let $X$ be such a manifold; we may assume that it is connected.
Matsumura's structure theorem \cite[Theorem~1.1 and Remark~1.2]{Matsumura25} supplies an integer
$q\geq0$ and a smooth projective rationally connected manifold $F$,
possibly a point, such that the universal cover of $X$ is
biholomorphic to $\C^q\times F$. By Theorem~\ref{thm:rc}, $F$
is Oka. Products of Oka manifolds are Oka, and the Oka property
descends through unramified holomorphic coverings \cite[Corollary~1.3]{F09}.
Hence $X$ is Oka.
\end{proof}

On a compact complex manifold $Z$ with a fixed Hermitian form
$\omega$, a holomorphic line bundle $L$ is \emph{nef} if, for every
$\epsilon>0$, it admits a smooth Hermitian metric whose normalized
real Chern curvature is bounded below by $-\epsilon\omega$.
For a holomorphic vector bundle $E\to Z$, let $\PP(E)$ parametrize
its one-dimensional quotients. The bundle $E$ is \emph{nef} if
$\OO_{\PP(E)}(1)$ is nef~\cite[Definitions~1.2 and~1.9]{DPS94}.

\begin{corollary}\label{cor:kahler-nef-tangent}
Every compact K\"ahler manifold with nef tangent bundle is Oka.
\end{corollary}
\begin{proof}
Let $X$ be such a manifold; we may assume that it is connected.
The structure theorem of
Demailly--Peternell--Schneider \cite[Theorem~3.14 and Proposition~3.9]{DPS94} gives a connected finite
\'etale cover $p:X'\to X$ whose Albanese map
\[
 a:X'\longrightarrow\operatorname{Alb}(X')
\]
is a surjective holomorphic submersion with connected Fano fibers.
If the fibers are points, $a$ is an isomorphism, so $X'$ is a complex
torus and is Oka.

Otherwise $TX'=p^*TX$ is nef, and hence so is
$-K_{X'}=\det TX'$. The general fiber of $a$ is Fano, so its
anticanonical bundle is nef and big. Cao--H\"oring's local-triviality
theorem \cite{CaoHoring17} therefore makes $a$ a holomorphic fiber
bundle. Its fibers are Oka by Theorem~\ref{thm:rc}, and its base is
a complex torus. The Oka fiber bundle theorem \cite[Corollary~1.3]{F09}
shows that $X'$ is Oka. In either case the Oka property descends
along the unramified covering $p$.
\end{proof}

\section{Divisor complements}\label{sec:divisor-complements}

\subsection{Logarithmic rational curves and divisor complements}
\label{sec:log}

Let $D=\sum_{i=1}^rD_i$ be a reduced simple normal crossings divisor
on a smooth projective manifold $X$ of dimension $n$. Thus every $D_i$
is smooth and
locally their union is a union of coordinate hyperplanes. The bundle
$T_X(-\log D)$ consists of vector fields tangent to all $D_i$; in
coordinates where $D=\{x_1\cdots x_k=0\}$ it has frame
\[
 x_1\frac{\partial}{\partial x_1},\ldots,
 x_k\frac{\partial}{\partial x_k},
 \frac{\partial}{\partial x_{k+1}},\ldots,
 \frac{\partial}{\partial x_n}.
\]
We allow $D=0$ and use the fixed source coordinate of Section~\ref{sec:rc}.

For a divisor complement, deformation of an ordinary rational curve
need not preserve avoidance of the boundary. The appropriate parameter
space keeps the complete pullback divisor fixed at infinity. Its
infinitesimal kernel is the logarithmic tangent bundle, and positivity
of that bundle again gives a submersive first-jet map. Here the
hypothesis supplies just one curve. Its deformations give families on
a nonempty Zariski-open subset of the complement, after which analytic
propagation reaches every point. We use the standard logarithmic
deformation principle recalled in \cite[Section~2.2]{ChenZhu}; here
the contact conditions are realized as a fiber of an ordinary divisor
map, and first-jet smoothness is checked directly.

\begin{theorem}\label{thm:log}
Let $X$ be a connected complex projective manifold and $D$ a
reduced simple normal crossings divisor. Suppose there is a nonconstant
map $f:\PP^1\to X$ such that
\[
 f^{-1}(D)\subseteq\{\infty\},\qquad
 f^*T_X(-\log D)\ \text{is ample}.
\]
Then $U=X\setminus D$ is Oka.
\end{theorem}

The proof follows Lemma~\ref{log:contact-jets} below.

\subsubsection{Fixed contact and first jets}
Write $f^*D_i=m_i[\infty]$, where $m_i\geq0$.

Fix a projective
embedding of $X$ and work in the finite-type scheme of morphisms
$\PP^1\to X$ of degree $\deg f^*\OO_X(1)$
\cite{Debarre}. Retain its open-and-closed locus where
$\deg g^*\OO_X(D_i)=m_i$ for each $i$, and let $H$ be the open
subset of maps contained in no $D_i$. For $m\geq0$, the symmetric power $\operatorname{Sym}^{m}(\PP^1)$
parametrizes effective divisors of degree $m$; the factor for $m=0$,
and an empty product, is a point. There is an algebraic divisor map
\[
 c:H\longrightarrow B:=\prod_{i=1}^r\operatorname{Sym}^{m_i}(\PP^1),
 \qquad c(g)=(g^*D_i)_i.
\]
To obtain $c$ scheme-theoretically, pull back the defining section of
each $D_i$ along universal evaluation. It is not identically zero on
any source fiber. Its zero scheme is therefore a relative effective
Cartier divisor, flat over $H$, by the fiberwise Cartier criterion
\cite{Stacks062Y}. It is proper and fiberwise finite of
constant degree $m_i$, so it determines the indicated morphism to the
symmetric power of the smooth curve $\PP^1$~\cite[Remark~9.3.9]{Kleiman05}.

Set $b_0=(m_i[\infty])_i$ and take the scheme-theoretic fiber
\[
 M=c^{-1}(b_0).
\]
Thus all contact points, including their multiplicities, are fixed,
while the maps remain parametrized. Universal evaluation restricts to
\begin{equation}\label{log:universal}
 \mathcal F:M\times\mathbb A^1\longrightarrow U.
\end{equation}
The fixed coordinate identifies first jets at $0$ with $TU$, giving
an algebraic morphism $J:M\to TU$, $J(g)=(g(0),g'(0))$.

\begin{lemma}\label{log:contact-jets}
If $g\in M$ and $E_g=g^*T_X(-\log D)$ is ample, then $J:M\to TU$
is smooth at $g$. Moreover $M$ is smooth at $g$ and
$T_gM=H^0(\PP^1,E_g)$.
\end{lemma}
\begin{proof}
We first prove smoothness of the contact fiber and compute its tangent
space. Set
\[
 Z_i=m_i[\infty],\qquad
 Q_g=\bigoplus_{i:m_i>0}g^*\OO_X(D_i)|_{Z_i}.
\]
The logarithmic inclusion gives an exact sequence
\begin{equation}\label{log:contact-sequence}
 0\longrightarrow E_g\longrightarrow g^*TX
 \longrightarrow Q_g\longrightarrow0.
\end{equation}
Indeed, near $g(\infty)$ take normal crossings coordinates and a
source coordinate $z$ vanishing at infinity. For the components met by
$g$, one has $x_i\circ g=z^{m_i}u_i(z)$ with $u_i(0)\ne0$.
In the corresponding normal direction, the inclusion is multiplication
by $z^{m_i}u_i(z)$, with quotient modulo $z^{m_i}$. All other directions
are unchanged. Away from infinity the inclusion is an isomorphism.
This proves both injectivity and the asserted torsion quotient.

The tangent space to the divisor target is
\[
 T_{Z_i}\operatorname{Sym}^{m_i}(\PP^1)
 =H^0\bigl(Z_i,\OO_{\PP^1}(Z_i)|_{Z_i}\bigr)
 =H^0\bigl(Z_i,g^*\OO_X(D_i)|_{Z_i}\bigr).
\]
The final identification uses the defining section with zero divisor
$Z_i$. With the corresponding convention for infinitesimal divisor
displacement, $dc_g$ is the quotient map to $H^0(Q_g)$: a variation
of $x_i\circ g$ is taken modulo $z^{m_i}$. Thus $Q_g$ records the
infinitesimal motion of the contact divisors.

Ampleness gives
\[
 H^1(\PP^1,E_g)=H^1(\PP^1,E_g(-2[0]))=0.
\]
Since $Q_g$ is supported at infinity, $H^1(\PP^1,Q_g)=0$.
Sequence~\eqref{log:contact-sequence} therefore gives
$H^1(\PP^1,g^*TX)=0$ and makes the quotient map
$H^0(\PP^1,g^*TX)\to H^0(Q_g)$ surjective. Fixed-source
deformation theory makes $H$ smooth at $g$, with tangent space
$H^0(\PP^1,g^*TX)$. The quotient map is $dc_g$, so $c$ is
smooth at $g$, since $B$ is smooth. Its fiber $M$ is consequently
smooth at $g$, with
\[
 T_gM=\ker dc_g=H^0(\PP^1,E_g).
\]

The supports of $2[0]$ and the contact divisors at infinity are
disjoint, so $E_g|_{2[0]}=g^*TX|_{2[0]}$ under the natural
inclusion. As in the proof of Lemma~\ref{rc:jet-smooth}, $dJ_g$
is the restriction map
\[
 H^0(\PP^1,E_g)\longrightarrow H^0(2[0],E_g|_{2[0]})
 \simeq T_{J(g)}TU.
\]
It is surjective because $H^1(\PP^1,E_g(-2[0]))=0$.
Since $M$ and $TU$ are smooth at the relevant points, $J$ is
smooth at $g$.
\end{proof}

\begin{proof}[Proof of Theorem~\ref{thm:log}]
By Lemma~\ref{log:contact-jets}, $f$ belongs to the smoothness locus
of the morphism $J:M\to TU$. Choose a nonempty Zariski-open
neighborhood $S\subset M$ of $f$ on which $J$ is smooth. Then $S$
is smooth. Let $\pi:TU\to U$ be the bundle projection.
Evaluation at $0$ is the smooth map
$\operatorname{ev}_0=\pi\circ J:S\to U$, and its image
\[
 U^+=\operatorname{ev}_0(S)
\]
is therefore nonempty and Zariski open.

For each $x\in U^+$ choose $g\in S$ over $x$. The analytification of
\eqref{log:universal} on $S\times\C$ and
Lemma~\ref{rc:jet-families} yield centered entire families near $x$
whose initial velocities form a frame. Hence $U^+$ lies in the
\LAP{} locus $G_U$ of
Definition~\ref{family:condition}. Since $X$ is smooth and connected it
is irreducible, and $U\setminus U^+$ is a proper closed algebraic,
hence analytic, subset of $U$.
Theorem~\ref{thm:analytic-removability} now shows that $U$ is Oka.
\end{proof}

\begin{remark}\label{log:degree-bound}
The positivity hypothesis in Theorem~\ref{thm:log} has a necessary
numerical consequence. Write
\[
 f^*T_X(-\log D)=\bigoplus_{\nu=1}^n\OO_{\PP^1}(a_\nu),
 \qquad a_\nu\geq1.
\]
The pushforward cycle $f_*[\PP^1]$ counts the image curve with the
degree of $f$ onto its image. Since
$\det T_X(-\log D)=\OO_X(-(K_X+D))$, we obtain
\[
 -(K_X+D)\cdot f_*[\PP^1]
 =\deg f^*T_X(-\log D)=\sum_{\nu=1}^n a_\nu\geq n.
\]
The inequality constrains the logarithmic-curve
criterion, not the Oka property of the complement; the critical-degree
case in Subsection~\ref{sec:sharp-complements} makes this distinction explicit.
\end{remark}

\subsubsection{Ample boundary and the log-Fano application}
An algebraic variety $U$ is \emph{$\mathbb A^1$-connected} if there is
an algebraic family $a:T\times\mathbb A^1\to U$ of nonconstant proper
maps whose two-point evaluation
\[
 T\times\mathbb A^1\times\mathbb A^1\longrightarrow U\times U,
 \qquad(s,t_1,t_2)\longmapsto(a(s,t_1),a(s,t_2))
\]
is dominant~\cite[Definitions~1.1--1.2]{ChenZhu}.

\begin{corollary}\label{cor:a1-connected}
Let $X$ be a connected complex projective manifold with a reduced
simple normal crossings divisor $D$. If $X\setminus D$ is
$\mathbb A^1$-connected, then it is Oka.
\end{corollary}
\begin{proof}
In characteristic zero the dominant two-point evaluation is separable.
Chen--Zhu's very-free-curve criterion~\cite[Proposition~1.13(2)]{ChenZhu} gives a
nonconstant map $f:\PP^1\to X$ with
$f^{-1}(D)=\{\infty\}$ and $f^*T_X(-\log D)$ ample.
Theorem~\ref{thm:log} applies.
\end{proof}

\begin{theorem}\label{thm:log-fano}
Let $X$ be a connected complex projective manifold of dimension at
least two and $D\subset X$ a smooth irreducible ample divisor.
If $D$ is rationally connected, then $X\setminus D$ is Oka and
holomorphically elliptic. The same conclusion holds if
$-(K_X+D)$ is ample.
\end{theorem}
\begin{proof}
We use the following specialization of Chen--Zhu's fully-free-center
criterion~\cite[Theorem~1.3 and Remark~1.4]{ChenZhu}: if a smooth
complex projective variety
of dimension at least two has a single smooth irreducible ample
boundary divisor that is rationally connected, then its complement is
$\mathbb A^1$-connected. To see how the hypotheses
enter, the unique boundary stratum is $D$ itself.
Choose a very free curve $a:\PP^1\to D$. If $\dim D=1$, use an
isomorphism $\PP^1\simeq D$. Ampleness gives
$\deg a^*\OO_D(D)>0$, so the contact intersection matrix consists
of one positive entry and has rank one. Thus $D$ is fully free,
as required by the criterion.
Corollary~\ref{cor:a1-connected} gives the Oka property.

Choose $k>0$ with $\OO_X(kD)$ very ample. In the associated
projective embedding, the section defining $kD$ cuts out a hyperplane.
Consequently $X\setminus D$ is closed in the affine complement of
that hyperplane. It is affine, hence Stein. A Stein Oka manifold is
holomorphically elliptic \cite{Gro89}, which proves
the first assertion.

If $-(K_X+D)$ is ample, adjunction gives
$-K_D=-(K_X+D)|_D$. Thus $D$ is Fano and is rationally connected
\cite{Campana,KMM}; the first assertion applies.
\end{proof}

\begin{corollary}\label{cor:log-complete-intersections}
Let $X\subset\PP^N$ be a smooth complete intersection of
multidegree $(d_1,\ldots,d_c)$ and dimension at least two, and $D$
be a smooth hypersurface section of degree $e\geq1$. If
\[
 d_1+\cdots+d_c+e\leq N,
\]
then $X\setminus D$ is Oka and holomorphically elliptic.
\end{corollary}
\begin{proof}
Let $H$ denote the hyperplane class. Adjunction gives
\[
 -(K_X+D)=(N+1-d_1-\cdots-d_c-e)H|_X,
\]
which is ample under the stated inequality. The divisor $D$ is also
ample, since $D\sim eH|_X$.

Positive-dimensional projective complete intersections are connected.
Hence $X$ and $D$ are connected, and smoothness makes $D$ irreducible.
Theorem~\ref{thm:log-fano} applies.
\end{proof}

\subsection{Smooth hypersurface complements}
\label{sec:sharp-complements}

We now determine exactly when the complement of a smooth degree-$d$
hypersurface in $\PP^n$ is Oka. The logarithmic construction already
handles $d\leq n$ when $n\geq2$. At the remaining positive degree
$d=n+1$, the logarithmic tangent bundle has trivial determinant, so its
pullback to a rational curve cannot be ample. Instead, we use lines
with contact of order at least $d-1$ with the hypersurface. Their
complements are copies of $\C$ or $\C^*$, and their entire
parametrizations vary holomorphically even when these two cases meet.

Kusakabe proved that complements of projective quadrics are Oka
\cite{Kus21}, and Xie established the Oka property and holomorphic
ellipticity for smooth cubic complements in $\PP^n$, $n\geq2$
\cite{X26}. Theorem~\ref{thm:sharp-hypersurfaces} gives the exact
threshold in every degree. It answers, for smooth hypersurfaces,
Forstneri\v{c}'s question of which projective hypersurfaces have
Oka complements \cite{Forstneric23}.

The proof has three ingredients: generic projection makes successive
polar equations transverse; these equations supply enough directions
through a general point of the complement; and an exponential formula
turns the directions into centered entire families. Analytic
propagation then reaches the remaining points. Carlson--Griffiths'
rank obstruction gives the converse \cite{CarlsonGriffiths}.

\subsubsection{Successive polars and generic projection}

In the next two lemmas, $N$ is the vector-space dimension and
$X\subset\PP^{N-1}$ denotes the hypersurface. In the application to
$\PP^n\setminus D$, these symbols mean $N=n+1$ and $X=D$.

Successive polars encode high contact of lines with a hypersurface;
see \cite{BCFS23}. Use coordinates $v=(v_1,\ldots,v_N)$ on $\C^N$. For a homogeneous
polynomial $F$ of degree $d$ and a nonzero vector
$a=(a_1,\ldots,a_N)$, write
\[
 D_a=\sum_{i=1}^N a_i\frac{\partial}{\partial v_i}.
\]
The equations $D_a^jF(v)=0$ are bihomogeneous in $(a,v)$, so their
projective zero sets depend only on $[a]$ and $[v]$.

Suppose $d\geq2$, $F(a)\ne0$, and $v\ne0$ satisfies
$D_a^jF(v)=0$ for $0\leq j\leq d-2$. Write $dF_a(v)$ for the
differential of $F$ at $a$ applied to $v$; by contrast, $D_aF(v)$
differentiates at $v$ in the direction $a$. For $s\in\C$,
homogeneity gives
\begin{equation}\label{hyp:polar-line-identity}
 F(a+sv)=\sum_{\ell=0}^d
          \frac{s^{d-\ell}}{\ell!}D_a^\ell F(v)
        =F(a)+s\,dF_a(v).
\end{equation}
The first equality follows for $s\ne0$ by expanding $s^dF(v+a/s)$
and hence holds identically as a polynomial in $s$. The polar
equations kill precisely the coefficients of $s^2,\ldots,s^d$.
Thus the restriction of $F$ to the line through $[a]$ and $[v]$
has degree at most one in the affine parameter $s$, and the line
meets the hypersurface with multiplicity at least $d-1$ at $[v]$.

For a general point $[a]$ of the complement, we therefore seek
points $[v]$ satisfying these polar equations. Two properties are
needed: transversality, so that the points $[v]$ can be continued
holomorphically as $[a]$ moves, and spanning, so that the resulting
line directions generate the tangent space at $[a]$.

\begin{lemma}\label{hyp:polar-transversality}
Suppose $N\geq d\geq3$ and $X=\{F=0\}\subset\PP^{N-1}$ is
smooth. There is a nonempty Zariski-open subset
$\Omega\subset\PP^{N-1}\setminus X$ such that, for every
$[a]\in\Omega$ and $1\leq k\leq d-2$, the differentials of
$D_aF,\ldots,D_a^kF$ restricted to $X$ are linearly independent
at their common zeros. The differentials here are taken in local
trivializations of the corresponding line bundles.
\end{lemma}
\begin{proof}
The polar equations describe successive tangencies of projection
from $[a]$. We first use generic projection to make those tangency
conditions transverse, then identify their local equations with the
polars.
Fix a hyperplane $L\simeq\PP^{N-2}$. Projection from
$[a]\notin X\cup L$ gives a finite map $p_a:X\to L$: a
positive-dimensional fiber would contain its whole line through
$[a]$, contradicting $[a]\notin X$. Its differential has kernel of
dimension at most one. The concrete generic-projection consequence
we need is the following: for a general center, the critical locus
is smooth of codimension one in $X$ or empty; recursively, the locus
where the kernel line is tangent to the preceding locus is smooth
of codimension $k$ in $X$ or empty, for $2\leq k\leq d-2$.

More precisely, we use the jet-transversality statement in
Alzati--Ballico--Ottaviani~\cite[Section~2]{ABO} for the smooth
hypersurface $X$ and point centers $T=\{[a]\}$: in their notation,
$c=1$ and $t=0$. The center space is
$\mathcal F=\PP^{N-1}\setminus(X\cup L)$, and
$\dim L=N-2=\dim X$.
For each $k$, their statement gives a nonempty Zariski-open subset
$U_k\subset\mathcal F$ on which $j^kp_a$ is transverse to
$\Sigma^{I_k}\subset J^k(X,L)$, where
$I_k=(\underbrace{1,\ldots,1}_{k})$.
This finite symbol prescribes the first $k$ entries, with no terminal
zero, so its inverse image includes points with further tangencies.
For a corank-one map between manifolds of equal dimension,
$\Sigma^{I_k}$ has codimension $k$; the local equations below
identify its inverse image with the successive tangency locus.
The finite intersection
$\Omega=\bigcap_{k=1}^{d-2}U_k$ is nonempty and Zariski open,
since $\mathcal F$ is irreducible.

We identify the equations supplied by this transversality. Near a
common zero $[v]$ of the polars, choose a linear form $\ell$ with
$\ell(a)=0$ and $\ell(v)\ne0$. In the chart $\ell=1$, choose
coordinates $(z,u,t)$ with $z\in\C^{N-3}$ so that translation by
$a$ is translation in $t$ and the projection forgets $t$. If $P$ is
the equation of $X$ in this chart, then
\[
 \partial_t^jP=(D_a^jF)|_{\ell=1}.
\]
At the point under consideration $P_t=0$. Smoothness lets us choose
$u$ with $P_u\ne0$ and write $X$ locally as $u=g(z,t)$. Set
\[
 A=P_u(z,g(z,t),t),\qquad
 q_j=(\partial_t^jP)(z,g(z,t),t),\qquad
 h_j=\partial_t^jg.
\]
After shrinking, $A$ is a unit. Differentiating
$P(z,g(z,t),t)=0$ repeatedly gives
\begin{equation}\label{hyp:triangular-polars}
 q_j+Ah_j\in(h_1,\ldots,h_{j-1}),\qquad
 (q_1,\ldots,q_j)=(h_1,\ldots,h_j).
\end{equation}
Indeed, the first identity starts with $q_1+Ah_1=0$, and
$\partial_tq_j=q_{j+1}+P_{ut^j}(z,g(z,t),t)h_1$
proves the induction. The equality of ideals follows because $A$
is invertible.

In these coordinates $p_a(z,t)=(z,g(z,t))$. Its critical locus is
$S_1=\{h_1=0\}$, and transversality gives $dh_1\ne0$ there.
On this locus its kernel is spanned by $\partial_t$. Inductively,
if $S_j=\{h_1=\cdots=h_j=0\}$ is the $j$th projection locus,
then
\[
 \partial_t\in T_pS_j
 \quad\Longleftrightarrow\quad h_{j+1}(p)=0,
 \qquad p\in S_j,
\]
because $\partial_th_i=h_{i+1}$. Thus the next corank-one locus
is $S_{j+1}$. This identification also holds for the defining ideals: the Jacobian
of $(p_a,h_1,\ldots,h_j)$ has an identity block in the $z$
variables, so its maximal-minor ideal is
$(h_1,\ldots,h_{j+1})$. Modulo the preceding equations the
new equation is $h_{j+1}$. Thus the successive Thom--Boardman
equations pull back to these generators, and jet transversality
gives their differential rank $j+1$.
Induction proves this rank statement up to $j=d-2$; when $d=3$
only the first step is needed. At a common zero,
\eqref{hyp:triangular-polars} relates the differentials of the
$q_j$ and $h_j$ by an invertible triangular matrix, so the same
rank statement holds for the polars.
\end{proof}

\begin{lemma}\label{lem:polar-directions}
Under the hypotheses and with the open set $\Omega$ of
Lemma~\ref{hyp:polar-transversality}, set
\[
 Z_a=\{F=D_aF=\cdots=D_a^{d-2}F=0\}\subset\PP^{N-1}.
\]
Then $Z_a$ is a nonempty smooth complete intersection of multidegree
$(d,d-1,\ldots,2)$, dimension $N-d$, and degree $d!$. Its points
span $\PP^{N-1}$. Moreover, the incidence variety
\[
 \mathcal I=
 \{([a],[v])\in\Omega\times X:[v]\in Z_a\}
\]
is smooth, and its projection $\pi:\mathcal I\to\Omega$ is a
surjective holomorphic submersion with fibers $Z_a$.
\end{lemma}
\begin{proof}
Let $S=\C[v_1,\ldots,v_N]$ and
$J=(F,D_aF,\ldots,D_a^{d-2}F)$. There are $d-1<N$ homogeneous
generators. If the projective zero set were empty, the radical of $J$
would be the irrelevant maximal ideal of height $N$, contrary to
the height theorem. Thus $Z_a$ is nonempty. The preceding lemma
shows that the equations, including $F$, have independent
differentials in projective coordinates along $Z_a$. Hence it is
smooth of dimension $N-d$, and $J$ has height $d-1$.
The generators therefore form a regular sequence in the polynomial
ring. This proves the complete-intersection assertion and gives
\[
 \deg Z_a=d(d-1)\cdots2=d!.
\]

To prove spanning, let $\mathcal I_{Z_a}$ be the ideal sheaf of
$Z_a$. Twist its Koszul resolution by $\OO(1)$. Every line-bundle
summand has negative degree, and the resolution has length
$d-2\leq N-2$, below the top cohomology degree of $\PP^{N-1}$.
The vanishings
\[
 H^i(\PP^{N-1},\OO(-k))=0
 \qquad(k>0,\ 0\leq i\leq N-2)
\]
give $H^0(\PP^{N-1},\mathcal I_{Z_a}(1))=0$ as follows.
Write the twisted resolution as
\[
 0\longrightarrow K_{d-2}\longrightarrow\cdots
 \longrightarrow K_0\longrightarrow\mathcal I_{Z_a}(1)
 \longrightarrow0.
\]
Splitting it into short exact sequences and using the displayed
vanishings successively injects $H^0(\mathcal I_{Z_a}(1))$ into
$H^{d-2}(K_{d-2})=0$, since $d-2\leq N-2$.
The complete-intersection scheme is smooth, hence reduced.
A linear form vanishing at all its points belongs to this space
and must be zero. Thus $Z_a$ spans $\PP^{N-1}$, including when
$d=N$ and $Z_a$ is zero-dimensional.

Finally, in local projective coordinates on $\Omega\times X$,
the incidence equations have independent differentials in the $X$
variables. The implicit function theorem makes $\mathcal I$ smooth
and $\pi$ a holomorphic submersion. Every fiber is nonempty, as
already proved, so $\pi$ is surjective.
\end{proof}

\subsubsection{Entire parametrizations of the line complements}

Buzzard--Lu~\cite[proof of Proposition~5.1, formulas~(5.1)--(5.2)]{BL00}
use the following entire function in their parametrization of
tangent-line complements of a smooth plane cubic.
The entire function $E:\C^2\to\C$ is given, for
$(\lambda,t)\in\C^2$, by
\begin{equation*}
 E(\lambda,t)=\sum_{k=0}^{\infty}
       \frac{\lambda^kt^{k+1}}{(k+1)!}
 =\begin{cases}
 (e^{\lambda t}-1)/\lambda,&\lambda\ne0,\\
 t,&\lambda=0.
 \end{cases}
\end{equation*}
The series converges locally uniformly on $\C^2$ and gives
\[
 1+\lambda E(\lambda,t)=e^{\lambda t},\qquad
 \partial_tE(\lambda,0)=1.
\]
For $\lambda\ne0$ its image is $\C\setminus\{-1/\lambda\}$;
for $\lambda=0$ its image is all of $\C$.

\begin{lemma}\label{hyp:entire-families}
Let $F$ be a homogeneous polynomial on $\C^{n+1}$,
$Y=\PP^n\setminus\{F=0\}$, and $V\subset Y$ an open set.
Suppose there are nonzero holomorphic lifts
$a,v_1,\ldots,v_n:V\to\C^{n+1}$ such that $[a(y)]=y$, the
classes of $v_1(y),\ldots,v_n(y)$ form a basis modulo $\C a(y)$,
and
\begin{equation}\label{hyp:linear-restriction}
 F(a(y)+sv_j(y))=F(a(y))+s\,dF_{a(y)}(v_j(y))
 \quad(s\in\C).
\end{equation}
Then every point of $V$ admits centered entire families with
initial velocities spanning its tangent space.
\end{lemma}
\begin{proof}
Set $c(y)=F(a(y))$ and
$\lambda_j(y)=dF_{a(y)}(v_j(y))$. The required families are
\begin{equation*}
 h_j(y,t)=
 [a(y)+c(y)E(\lambda_j(y),t)v_j(y)].
\end{equation*}
The expression inside brackets has $F$-value
$c(y)e^{\lambda_j(y)t}\ne0$, so it is nonzero and defines a
holomorphic map $h_j:V\times\C\to Y$. It is centered at $y$.
For projectivization $q:\C^{n+1}\setminus\{0\}\to\PP^n$,
\[
 \partial_th_j(y,0)=dq_{a(y)}(c(y)v_j(y)).
\]
Since $\ker dq_{a(y)}=\C a(y)$ and $c(y)\ne0$, these velocities
form a basis. Lemma~\ref{family:entire} now applies.
\end{proof}

For the polar directions this construction has a direct geometric
meaning. Let $D=\{F=0\}$. If $F(a)\ne0$ and $[v]\in Z_a$,
write $c=F(a)$ and $\lambda=dF_a(v)$. With homogeneous line
coordinates $[r:s]\in\PP^1$, homogenizing
\eqref{hyp:polar-line-identity} gives
\[
 F(ra+sv)=r^{d-1}(rc+s\lambda).
\]
The line $L_{a,v}=\{[ra+sv]:[r:s]\in\PP^1\}$ is not contained
in the hypersurface, and its intersection divisor is
\[
 D|_{L_{a,v}}=
 \begin{cases}
 (d-1)[v]+[\lambda a-cv],&\lambda\ne0,\\
 d[v],&\lambda=0.
 \end{cases}
\]
Consequently its complement is $\C^*$ or $\C$, respectively.
The map $t\mapsto[a+cE(\lambda,t)v]$ parametrizes the entire
line complement in either case and depends holomorphically on
$\lambda$ through zero.

In the line coordinates $[r:s]$, Figure~\ref{fig:polar-line-contact}
shows how the two intersection points coalesce.

\begin{figure}[ht]
\centering
\begin{tikzpicture}[x=1cm,y=1cm,font=\small,
 line cap=round,line join=round,
 removed/.style={draw=figOverlap!90!black,fill=white,line width=1pt},
 kept/.style={draw=figOld!80!black,fill=figOld,line width=.5pt}]
 \node[font=\small\bfseries] at (2.5,2.4) {(a) $\lambda\ne0$};
 \node[font=\small\bfseries] at (9.7,2.4) {(b) $\lambda=0$};
 \path[draw=figInk,fill=figOld!8,line width=.65pt]
   (2.5,.25) ellipse[x radius=1.75,y radius=1.05];
 \path[draw=figInk,fill=figNew!8,line width=.65pt]
   (9.7,.25) ellipse[x radius=1.75,y radius=1.05];
 \draw[removed] (2.5,1.30) circle[radius=.085];
 \node[above=5pt,text=figInk] at (2.5,1.30) {$[0:1]$};
 \node[anchor=west,text=figOverlap!80!black,font=\footnotesize]
   at (3.00,1.55) {multiplicity $d-1$};
 \draw[removed] (4.1997,.50) circle[radius=.085];
 \node[anchor=west,text=figInk] at (4.40,.61) {$[\lambda:-c]$};
 \node[anchor=west,text=figOverlap!80!black,font=\footnotesize]
   at (4.40,.24) {multiplicity $1$};
 \draw[kept] (2.5,-.80) circle[radius=.055];
 \node[below=5pt,text=figOld!75!black] at (2.5,-.80) {$[1:0]$};
 \node[text=figInk] at (2.5,-1.67)
   {line complement $\simeq\C^*$};
 \draw[->,>=latex,figInk,line width=.7pt] (5.65,-1.45)--(6.95,-1.45);
 \node[above=5pt,text=figInk] at (6.30,-1.45) {$\lambda\to0$};
 \draw[removed] (9.7,1.30) circle[radius=.085];
 \node[above=5pt,text=figInk] at (9.7,1.30) {$[0:1]$};
 \node[anchor=west,text=figOverlap!80!black,font=\footnotesize]
   at (10.20,1.55) {multiplicity $d$};
 \draw[kept] (9.7,-.80) circle[radius=.055];
 \node[below=5pt,text=figOld!75!black] at (9.7,-.80) {$[1:0]$};
 \node[text=figInk] at (9.7,-1.67)
   {line complement $\simeq\C$};
\end{tikzpicture}
\caption{Intersection of a polar line with the hypersurface, in
homogeneous line coordinates. Each oval represents $\PP^1$
schematically; open marks are the deleted points, with their
intersection multiplicities. Since $c\ne0$, the simple point
$[\lambda:-c]$ is distinct from $[0:1]$ for $\lambda\ne0$ and
coalesces with it at $\lambda=0$. The base point $[1:0]$ remains
in the complement. The same formula $t\mapsto[1:cE(\lambda,t)]$
parametrizes both complements and stays holomorphic through the
coalescence.}
\label{fig:polar-line-contact}
\end{figure}
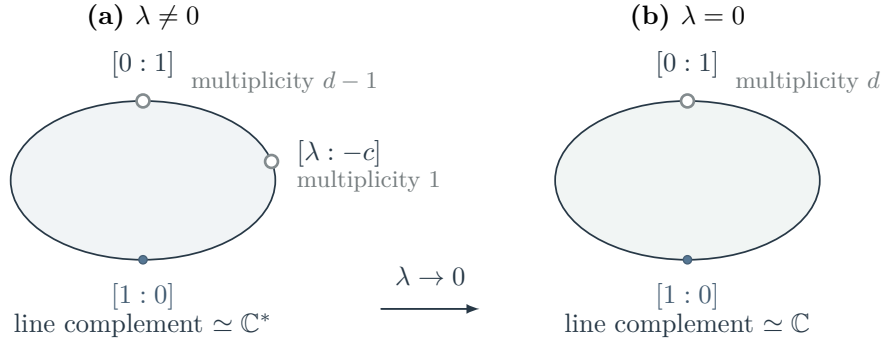

\subsubsection{The degree threshold}

\begin{proposition}\label{hyp:positive-degrees}
If $D\subset\PP^n$ is a smooth hypersurface of degree
$1\leq d\leq n+1$, where $n\geq1$, then $\PP^n\setminus D$
is Oka.
\end{proposition}
\begin{proof}
Write $D=\{F=0\}$ and $Y=\PP^n\setminus D$.
The manifold $Y$ is connected: a line joining two of its points
is not contained in $D$ and loses only finitely many points on
removing $D$.

For $d=1$, $Y\simeq\C^n$. For $d=2$, the smooth quadric spans
$\PP^n$. Indeed, its quadratic form is nondegenerate and its
isotropic vectors span $\C^{n+1}$: after diagonalization, the
vectors $e_i\pm\mathrm i e_j$ are isotropic and span. This also
covers $n=1$, when the quadric is two distinct points. At any
$y_0=[a_0]\in Y$, choose constant isotropic vectors
$v_1,\ldots,v_n$ whose classes form a basis modulo $\C a_0$, together
with a local lift $a(y)$ near $y_0$, and shrink its domain $V$
so that this basis property persists. The quadratic identity
\[
 F(a+sv_j)=F(a)+s\,dF_a(v_j)+s^2F(v_j)
         =F(a)+s\,dF_a(v_j)
\]
allows Lemma~\ref{hyp:entire-families} to be applied at every
point. The \LAP{} criterion, Theorem~\ref{thm:families-oka}, makes
$Y$ Oka.

Now suppose $3\leq d\leq n+1$ and apply
Lemma~\ref{lem:polar-directions} with $N=n+1$ and $X=D$.
For any $y_0=[a_0]\in\Omega$, spanning by $Z_{a_0}$ gives
points $[v_1^0],\ldots,[v_n^0]\in Z_{a_0}$ whose representatives
form a basis modulo $\C a_0$. The submersion
$\pi:\mathcal I\to\Omega$ has local holomorphic sections through
$([a_0],[v_j^0])$. On a common neighborhood $V$ of $y_0$, choose
nonzero holomorphic lifts $a,v_1,\ldots,v_n$ of these sections
and of the inclusion $V\hookrightarrow\PP^n$. Then
\[
 [a(y)]=y,\qquad [v_j(y)]\in Z_{a(y)},
\]
and the basis property persists after shrinking $V$ once.
Identity~\eqref{hyp:polar-line-identity} supplies
\eqref{hyp:linear-restriction}. Lemma~\ref{hyp:entire-families}
therefore gives centered entire families at every point of
$\Omega$.

Let $G_Y$ be the \LAP{} locus of Definition~\ref{family:condition}.
We have shown $\Omega\subset G_Y$. Since $Y\setminus\Omega$ is a
proper closed algebraic subset of $Y$,
Theorem~\ref{thm:analytic-removability} shows that $Y$ is Oka
throughout $d\leq n+1$.
\end{proof}

\begin{proof}[Proof of Theorem~\ref{thm:sharp-hypersurfaces}]
Proposition~\ref{hyp:positive-degrees} proves the positive direction.
Write $Y=\PP^n\setminus D$. For $d\geq n+2$, the line bundle
\[
 \OO_{\PP^n}(K_{\PP^n}+D)=\OO_{\PP^n}(d-n-1)
\]
is ample. Carlson--Griffiths' degeneracy theorem~\cite[Corollary~4.17]{CarlsonGriffiths}
for the complement of a simple normal crossings divisor with positive
log canonical class applies to the smooth divisor $D$.
It says that every holomorphic map $\C^n\to Y$ has rank less than
$n$ everywhere. An Oka manifold is strongly dominable by
Theorem~\ref{thm:strong-domination}, so $Y$ cannot be Oka. This
also treats $n=1$ and $d\geq3$.

Finally, $D$ is ample, so $Y$ is affine and hence Stein, by the
affine-complement argument in Subsection~\ref{sec:log}. A Stein Oka
manifold is holomorphically elliptic, and a holomorphically elliptic
manifold is Oka \cite{Gro89}. Thus the two analytic properties
are equivalent here, completing all three implications.
\end{proof}

\begin{example}\label{hyp:singular-high-degree}
Smoothness is essential to the degree threshold. In homogeneous
coordinates $[X:Y:Z]$ on $\PP^2$, for every $m\geq3$ the reduced
plane divisor
\[
 D_m=\{Z(YZ^{m-1}-X^m)=0\}\subset\PP^2
\]
has degree $m+1>3$ and is singular at $[0:1:0]$.
Its complement lies in the chart $Z\ne0$. In the affine coordinates
$x=X/Z$ and $y=Y/Z$, the polynomial automorphism
$(x,y)\mapsto(x,y-x^m)$ identifies it with
$\C\times\C^*$. It is therefore Oka and holomorphically elliptic.
This gives singular hypersurfaces with Oka complements in arbitrarily
high degree.
\end{example}

\section{Conjectures and further directions}\label{sec:conjectures}

A compact K\"ahler manifold is \emph{special} in
Campana's sense if it admits no meromorphic fibration onto a
positive-dimensional base of general type with its induced orbifold
structure~\cite[Definition~2.1]{Campana04}.

\begin{conjecture}\label{conj:special-oka1}
Every connected smooth complex projective manifold that is special
in Campana's sense is Oka-1.
\end{conjecture}

\begin{conjecture}\label{conj:oka1-oka}
Every connected smooth complex projective Oka-1 manifold is Oka.
\end{conjecture}

\begin{conjecture}
\label{conj:oka-deformation}
Let $\pi:\mathcal X\to S$ be a smooth projective morphism of complex
manifolds, with connected base $S$ and connected fibers. If one fiber
of $\pi$ is Oka, then every fiber is Oka.
\end{conjecture}

The first two conjectures together would make every special smooth
projective manifold Oka, answering this case of Campana--Winkelmann's
question~\cite[Question~1.8]{CW15}; see also
Forstneri\v{c}~\cite[Problem~11.3]{Forstneric23}. Related Oka-1
and deformation questions appear in
\cite[Problem~2.11(c)]{AF25} and \cite[Problem~F]{FL11}.

Every compact K\"ahler Oka manifold is special~\cite[Corollary~8.11]{Campana04}.
Thus the first two conjectures would identify the special, Oka-1,
and Oka classes in the projective category. Together with Campana's
conjecture that specialness is invariant under smooth deformation
\cite[Conjecture~6.23]{Campana04}, they would imply
Conjecture~\ref{conj:oka-deformation}.

Evidence includes rationally connected manifolds
(Theorem~\ref{thm:rc}), K3 surfaces~\cite[Theorem~1.1]{OAI26}, and
the Schoen fiber products with disjoint singular-value sets and
sufficiently general Gross--Popescu Calabi--Yau threefolds proved Oka
in \cite[Theorems~8.7--8.8]{ADGWX26}.
Analogous questions can be asked for compact K\"ahler manifolds.

\appendix
\section{The local extension condition and its equivalence with \LAP}
\label{app:L}

We use the local planar approximation condition of
OpenAI~\cite[Definition~2.1]{OAI26} and prove that it is equivalent
to \LAP{} at each point.

\begin{definition}\label{L:definition}
Let $Y$ be a complex manifold of dimension $n\geq1$. One says that
$Y$ satisfies $\LL$ at $x\in Y$ if there is an open neighborhood
$V\ni x$ such that, for every special planar pair $D\subset D'$,
every integer $q\geq0$, every pair of closed polydiscs
$P_0\Subset\operatorname{Int}P\subset\mathbb C^q$, and every
holomorphic map $f$ from a neighborhood of $D\times P$ into $V$,
the restriction $f|_{D\times P_0}$ is uniformly approximable by
maps holomorphic near $D'\times P_0$ with values in $Y$.
The set $V$ is called a witness neighborhood. The manifold $Y$
satisfies $\LL$ if it does so at every point.
\end{definition}

Special planar pairs are as in Definition~\ref{def:special-planar-pair}.
Passive parameters are retained through the following notational
reformulation. If
$f$ is holomorphic near $D\times Q$ into $V$, compactness permits
a slightly larger closed polydisc $Q^+$ such that
$Q\Subset\operatorname{Int}Q^+$ and $f$ is still holomorphic near
$D\times Q^+$ into $V$. Applying $\LL$ with $P_0=Q$ and $P=Q^+$
therefore retains the whole prescribed closed polydisc $Q$.
Conversely, this retained-parameter formulation immediately implies
the definition by restriction.

The next lemma is the disc-enlargement step of OpenAI
\cite[Proposition~2.9]{OAI26}, in arbitrary target dimension.

\begin{lemma}\label{L:enlargement}
Let $V$ be a witness neighborhood for $\LL$ and $f$ holomorphic
near $\overline\Delta_r\times P$ into $V$, where $r>0$ and $P$ is
a closed polydisc. For every $S>r$, the map $f$ is uniformly
approximable on $\overline\Delta_r\times P$ by maps holomorphic
near $\overline\Delta_S\times P$ with values in $Y$.
\end{lemma}
\begin{proof}
Use the construction in the proof of Lemma~\ref{poly:boundary-disc},
with $V$ for every boundary arc and the retained-parameter formulation
of $\LL$ in place of local planar extension. Compactness allows the
old chart spray and slightly enlarged passive products to take values
in $V$, so the sector collars also have their spray images in $V$.
Including the original passive variables and the spray variables in
$Q$, the extension and simultaneous center gluing retain a neighborhood
of the prescribed $P$ and give the required approximation.
\end{proof}

\begin{proposition}\label{L:equivalence}
For every positive-dimensional complex manifold $Y$ and every
$x\in Y$, the conditions
$\LL$ and \LAP{} at $x$ are equivalent.
\end{proposition}
\begin{proof}
Suppose first that \LAP{} holds at $x$.
Lemma~\ref{family:local-sector} provides a neighborhood $W_x$ on
which local planar extension retains every closed passive polydisc.
Restricting to $P_0$ gives $\LL$ with witness $W_x$.

Conversely, let $V$ witness $\LL$ at $x$, and choose a coordinate
parametrization $\psi$ about $0\in\mathbb C^n$, with
$\psi(0)=x$ and image in $V$. For each $1\leq j\leq n$, regard
the $j$th coordinate as the scalar variable $z$ and the remaining
coordinates as passive variables $p$. After restriction, this gives
biholomorphic parametrizations $\psi_j$ near fixed products
$\overline\Delta_r\times P_j$, where
$0\in\operatorname{Int}P_j$, and the vectors
$\partial_z\psi_j(0,0)$ form a basis of $T_xY$.

Choose $S_\nu\to\infty$. Lemma~\ref{L:enlargement} gives maps
\[
 \sigma_{j,\nu}:\operatorname{Op}
 (\overline\Delta_{S_\nu}\times P_j)\longrightarrow Y
\]
whose errors relative to $\psi_j$ on
$\overline\Delta_r\times P_j$ tend to zero.
The products $P_j$ remain fixed as $\nu$ increases.
Lemma~\ref{family:recentering} now supplies the exactly centered
families required for \LAP{} at $x$.
\end{proof}

\section*{Acknowledgments}

The authors thank Franc Forstneri\v{c} for suggesting the name
\emph{\LAPName} (\LAP) and for his valuable comments
on an earlier version of this paper.

\section*{Grant support}

S.-Y. Xie acknowledges partial support from the National Key R\&D
Program of China under Grants No.~2023YFA1010500 and
No.~2021YFA1003100, and from the National Natural Science Foundation of
China under Grants No.~12288201 and No.~12471081, as well as support
from the Xiaomi Young Talents Program.

\section*{Statement on the use of artificial intelligence}

This manuscript grew out of earlier drafts, working notes, and discussions
within the research group. ChatGPT-6-Astra was used in exploring and
organizing the arguments and in revising the manuscript. The authors reviewed
the text and take full responsibility for its mathematical content and final
form.

\raggedbottom

\end{document}